\documentclass[10pt]{amsart}

\usepackage{graphicx, amssymb, enumerate, pstricks, pst-node}
\usepackage{pst-all}
\usepackage{latexsym}
\usepackage{amsmath}
\usepackage{epsfig}
\usepackage{epic}
\usepackage{amssymb}
\usepackage{url}

\newtheorem{theorem}{Theorem}[section]
\newtheorem{lemma}[theorem]{Lemma}

\newtheorem{sublemma}{}[theorem]

\newcommand{\ba}{\backslash}
\newcommand{\cl}{{\rm cl}}
\newcommand{\fcl}{{\rm fcl}}

\newcommand{\co}{{\rm co}}
\newcommand{\si}{{\rm si}}
\newcommand{\cls}{{\cl^*}}

\newcommand{\thc}{$3$-connected}
\newcommand{\ths}{$3$-separation}
\newcommand{\ifc}{internally $4$-connected}
\newcommand{\sfc}{sequentially $4$-connected}

\newcommand{\ffsc}{$(4,4,S)$-connected}

\newcommand{\cn}{contradiction}
\newcommand{\btu}{\bigtriangleup}
\newcommand{\ftv}{$(4,3)$-violator}
\newcommand{\ffsv}{$(4,4,S)$-violator}
\newcommand{\ort}{orthogonality}

\newcommand{\ns}{non-sequential}
\newcommand{\al}{\alpha}
\newcommand{\be}{\beta}
\newcommand{\ga}{\gamma}
\newcommand{\de}{\delta}
\begin{document}

\title[Towards a Splitter Theorem  VI]{Towards a Splitter Theorem for Internally $4$-connected Binary Matroids VI}

\thanks{The first author was supported by NSF IRFP Grant OISE0967050, an LMS Scheme 4 grant, and an AMS-Simons travel grant. The second author was supported by the National Security Agency.}

\author{Carolyn Chun}
\address{School of Mathematical Sciences,
Brunel University,
London,  
England}
\email{chchchun@gmail.com}

\author{James Oxley}
\address{Department of
 Mathematics, Louisiana State University, Baton Rouge, Louisiana, USA}
\email{oxley@math.lsu.edu}

\date{\today}

\begin{abstract} 
Let $M$ be a $3$-connected binary matroid; $M$ is called internally $4$-connected if one side of every $3$-separation is a triangle or a triad, and $M$ is \ffsc\ if one side of every $3$-separation is a triangle, a triad, or a $4$-element fan. 
Assume $M$ is   \ifc\ and that neither $M$ nor its dual is  a cubic 
M\"{o}bius or planar ladder or a certain coextension thereof. Let $N$ be an \ifc\ proper minor of $M$.  
Our aim is to show that  $M$   has a proper \ifc\ minor with an $N$-minor that can be obtained from $M$ either by removing at most four elements, or by   removing elements in an easily described way from a special substructure of $M$. When this aim cannot be met, the earlier papers in this series showed that, 
 up to duality, $M$ has a good bowtie, that is, a pair, $\{x_1,x_2,x_3\}$ and $\{x_4,x_5,x_6\}$, of disjoint triangles and a cocircuit, $\{x_2,x_3,x_4,x_5\}$, where $M\ba x_3$ has an $N$-minor and is \ffsc. 
We also showed that, when $M$ has a good bowtie, either $M\ba x_3,x_6$ has an $N$-minor; or $M\ba x_3/x_2$ has an $N$-minor and is \ffsc. 
In this paper, we show that, when $M\ba x_3,x_6$ has an $N$-minor but is not \ffsc,   $M$ 
has an \ifc\ proper minor  with an $N$-minor that can be obtained  from $M$ by removing at most three elements, or by   removing elements in a well-described way from one of several  special substructures of $M$. 
This is a significant step towards obtaining a splitter theorem for the class of \ifc\ binary matroids.
\end{abstract}

\maketitle

\section{Introduction}
\label{introduction}

Seymour's Splitter Theorem \cite{seymour}  established that if $N$ is a proper $3$-connected minor of   
a $3$-connected matroid $M$,  then $M$ has a proper $3$-connected minor $M'$ with an $N$-minor such that $|E(M) - E(M')| = 1$ unless $r(M) \ge 3$ and $M$ is a wheel or a whirl. The current paper is the sixth in a series whose aim is to obtain a splitter theorem   for the class of \ifc\ binary matroids. Specifically, we believe we can prove that if $M$ and $N$ are
\ifc\ binary matroids, and $M$ has
a proper $N$-minor, then $M$ has a proper minor $M'$ such that
$M'$ is internally $4$-connected with an $N$-minor, and
$M'$ can be produced from $M$ by a small number of
simple operations.

Any unexplained matroid terminology used here will follow \cite{oxrox}. The only $3$-separations allowed in an internally $4$-connected matroid have a triangle or a triad on one side. A $3$-connected matroid  $M$ is {\it $(4,4,S)$-connected} if, for every $3$-separation $(X,Y)$ of $M$, one of $X$ and $Y$ is a triangle, a triad, or a $4$-element fan, that is, a $4$-element set $\{x_1,x_2,x_3,x_4\}$ that can be ordered so that $\{x_1,x_2,x_3\}$ is a triangle and $\{x_2,x_3,x_4\}$ is a triad. 

To provide a context for our main theorem, we briefly describe our progress towards obtaining the desired splitter theorem.  Johnson and Thomas \cite{johtho} showed that, even for graphs, a splitter theorem in the internally $4$-connected case must take account of some special examples. For $n \ge 3$, let  $G_{n+2}$  be the {\it biwheel} with $n+2$ vertices, that is, $G_{n+2}$ consists of  an $n$-cycle $v_1,v_2,\ldots,v_{n},v_1$, the {\it rim}, and two additional  vertices, $u$ and $w$, both of which are adjacent to every $v_i$. Thus  the dual of $G_{n+2}$ is a cubic planar ladder. Let $M$ be the cycle matroid of $G_{2n+2}$ for some $n \ge 3$ and 
let $N$ be the cycle matroid of the graph that is obtained by proceeding around the rim of $G_{2n+2}$ and alternately deleting the edges from the rim vertex to $u$ and to $w$. Both $M$ and $N$ are internally $4$-connected but there is no   internally $4$-connected proper minor of $M$ that has a proper $N$-minor. We can modify $M$ slightly and still see the same phenomenon. Let $G_{n+2}^+$ be obtained from $G_{n+2}$ by adding a new edge $z$ joining the hubs $u$ and $w$. Let $\Delta_{n+1}$ be the binary matroid that is obtained from 
$M(G_{n+2}^+)$ by deleting the edge $v_{n-1}v_n$ and adding the third element on the line spanned by $wv_n$ and $uv_{n-1}$. This new element is also on the line spanned by $uv_n$ and $wv_{n-1}$. For $r \ge 3$, Mayhew, Royle, and Whittle~\cite{mayroywhi} call $\Delta_r$ the {\it rank-$r$ triangular M\"{o}bius matroid} and note that $\Delta_r \ba z$ is the dual of the cycle matroid of a cubic M\"{o}bius ladder. 
The following is the main result of \cite[Theorem~1.2]{cmoIII}. 

\begin{theorem}
\label{44S}
Let $M$ be an \ifc~binary matroid with an \ifc~proper minor $N$ such that $|E(M)|\geq 15$ and $|E(N)|\geq 6$.  
Then 
\begin{itemize}
\item[(i)] $M$ has a proper minor $M'$ such that $|E(M)-E(M')|\leq 3$ and $M'$ is \ifc\ with an $N$-minor; or 
\item[(ii)] for some $(M_0,N_0)$ in $\{(M,N), (M^*,N^*)\}$, the matroid $M_0$ has a triangle $T$ that contains an element $e$ such that  $M_0\ba e$  is $(4,4,S)$-connected with an $N_0$-minor; or 
\item[(iii)] $M$ is isomorphic to $M(G_{r+1}^+)$, $M(G_{r+1})$, $\Delta_r$,   or $\Delta_r \ba z$ for some $r \ge 5$.
\end{itemize}
\end{theorem}

\begin{figure}[htb]
\centering
\includegraphics[scale = 0.7]{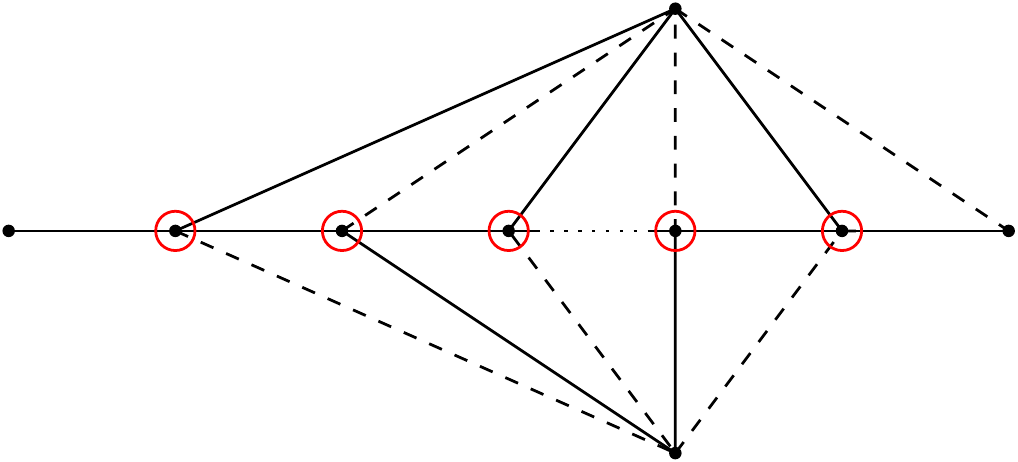}
\caption{All the elements shown are distinct. There are at least three dashed elements; and all dashed elements are deleted.}
\label{gmcdashed}
\end{figure}

That theorem prompted us to consider those matroids for which the second outcome in the theorem holds. In order to state the next result, we need to define some special structures. Let $M$ be an \ifc\ binary matroid and $N$ be an \ifc\ proper minor of $M$. Suppose $M$ has 
 disjoint triangles $T_1$ and $T_2$ and a $4$-cocircuit $D^*$ contained in their union. 
We call this structure a {\it bowtie} and denote it by $(T_1,T_2,D^*)$. If   $D^*$ has an element $d$ such that $M\ba d$ has an $N$-minor and $M\ba d$ is $(4,4,S)$-connected, then $(T_1,T_2,D^*)$ is a {\it good bowtie}. Motivated by (ii)  of the last theorem, we seek to determine more about the structure of $M$ when it has a triangle containing an element $e$ such that $M\ba e$ is \ffsc\ with an $N$-minor. One possible outcome here is that $M$ has a good bowtie. Indeed, as the next result shows, if that outcome or its dual does not arise, we get a small number of easily described alternatives. We shall need two more definitions.  A {\it terrahawk} is the graph  that is obtained  by adjoining a new vertex to a  cube and adding edges from the new  vertex to each of the four vertices that bound some fixed face of the cube. Figure~\ref{gmcdashed} shows  a modified graph diagram, which we use to keep track of some of the circuits and cocircuits in $M$.  
Each of the cycles in that diagram corresponds to a circuit  of $M$ while a circled vertex indicates a known cocircuit of $M$. At the end of Section~\ref{outline}, we shall say more about what can be inferred from such a diagram. 
We shall call a structure of the form shown in  Figure~\ref{gmcdashed} an {\it open rotor chain} 
noting that all of the elements in the figure are distinct and, for some $n \ge 3$, there are $n$ dashed elements. We will refer 
  to deleting the dashed elements from Figure~\ref{gmcdashed} as {\it trimming an open rotor chain}.   
The following is a special case of \cite[Corollary~1.4]{cmoV}.
 
 \begin{theorem}
\label{mainone4}
Let $M$ and $N$ be \ifc\ binary matroids such that $|E(M)| \ge 16$ and $|E(N)|\geq 6$. Suppose that $M$ has a triangle $T$ containing an element $e$ for which $M\ba e$ is \ffsc\ with an $N$-minor. Then one of the following holds.
\begin{itemize}
\item[(i)] $M$ has an \ifc\ minor $M'$ that has an $N$-minor such that  $1 \le |E(M) - E(M')| \le 4$;  or 
\item[(ii)] $M$ or $M^*$  has a good bowtie; or 
\item[(iii)]  $M$ is the cycle matroid of a terrahawk; or
\item[(iv)] for some $(M_0,N_0)$ in $\{(M,N), (M^*,N^*)\}$, the matroid $M_0$ contains 
an open rotor chain that can be trimmed to obtain an \ifc\ matroid with an $N_0$-minor.  
\end{itemize}
\end{theorem}

We remark that there is a small error in \cite[Theorem 1.1]{cmoV} since it requires at least five elements to be removed when trimming an open rotor chain. But, as the proof there makes clear, trimming exactly four elements is a possibility. Trimming exactly three elements is also possible but that is included under (i) of \cite[Theorem~1.1]{cmoV}. 
 
This theorem leads us to consider a good bowtie $(\{x_1,x_2,x_3\}, \linebreak \{x_4,x_5,x_6\}, \{x_2,x_3,x_4,x_5\})$ in an \ifc\ binary matroid $M$ where $M\ba x_3$ is \ffsc\ with an $N$-minor. In $M\ba x_3$, we see that $\{x_5,x_4,x_2\}$ is a triad and $\{x_6,x_5,x_4\}$ is a
triangle, so
$\{x_6,x_5,x_4,x_2\}$ is a $4$-element fan. It follows, by \cite[Lemma 2.5]{cmoIV}, that either
\begin{itemize}
\item[(i)] $M\ba x_3,x_6$ has an $N$-minor; or 
\item[(ii)] $M\ba x_3,x_6$ does not have an $N$-minor, but $M\ba x_3/x_2$ is \ffsc\ with an $N$-minor.
\end{itemize}
In this paper, we focus on the first of these two cases and assume, in addition, that 
$M\ba x_6$ is not \ffsc. In \cite{cmoVII}, we treat the second of these two cases. Finally, in \cite{cmoVIII}, we treat the remaining subcase of (i)

\begin{figure}[htb]
\center
\includegraphics{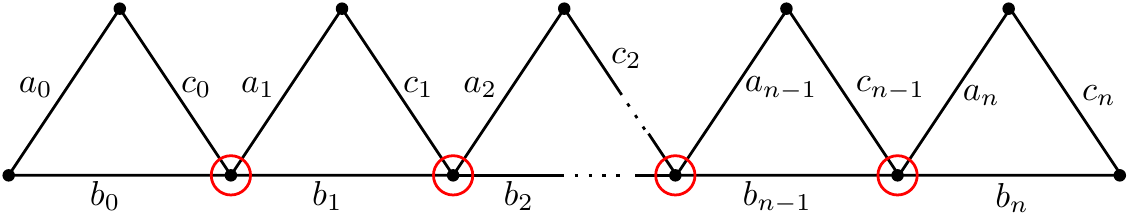}
\caption{A  string of bowties.  All elements are distinct except that $a_0$ may be the same as $c_n$.}
\label{bowtiechainfig}
\end{figure}

In a matroid $M$, a {\it string of bowties}  is a sequence $\{a_0,b_0,c_0\},\linebreak \{b_0,c_0,a_1,b_1\},\{a_1,b_1,c_1\}, 
\{b_1,c_1,a_2, 
 b_2\},\dots,\{a_n,b_n,c_n\}$ with $n \ge 1$ such that   
 \begin{itemize} 
 \item[(i)] $\{a_i,b_i,c_i\}$ is a triangle for all $i$ in $\{0,1,\dots, n\}$;  
 \item[(ii)] $\{b_j,c_j,a_{j+1},b_{j+1}\}$ is a cocircuit for all $j$ in $\{0,1,\dots ,n-1\}$; and 
 \item[(iii)] the elements $a_0,b_0,c_0,a_1,b_1,c_1,\ldots,a_n,b_n,c_n$ are distinct except that $a_0$ and $c_n$ may be equal.
 \end{itemize}

The reader should note that this differs slightly from the definition we gave in \cite{cmochain} in that  here we allow $a_0$ and $c_n$ to be equal instead of requiring all of the elements to be distinct. 
Figure~\ref{bowtiechainfig} illustrates a string of bowties, but this diagram may obscure the  potential complexity of such a string. Evidently $M\ba c_0$ has $\{c_1,b_1,a_1,b_0\}$ as a 4-fan. Indeed, $M\ba c_0,c_1,\ldots,c_i$ has a $4$-fan for all $i$ in $\{0,1,\ldots,n-1\}$. We shall say that the matroid $M\ba c_0,c_1,\ldots,c_n$ has been obtained from $M$ by {\it trimming a string of bowties}. This operation plays a prominent role in our main theorem, and is the underlying operation in trimming an open rotor chain.  
Before stating this result, we introduce the other operations that incorporate this process of trimming a string of bowties. 
Such a string can attach to the rest of the matroid in a variety of ways. In most of these cases, the operation of trimming the string will produce an \ifc\ minor of $M$ with an $N$-minor. But, in three cases, when the bowtie string is embedded in a modified  quartic ladder in certain ways, we need to adjust the trimming process. 

\begin{figure}[htb]
\center
\includegraphics[scale=0.72]{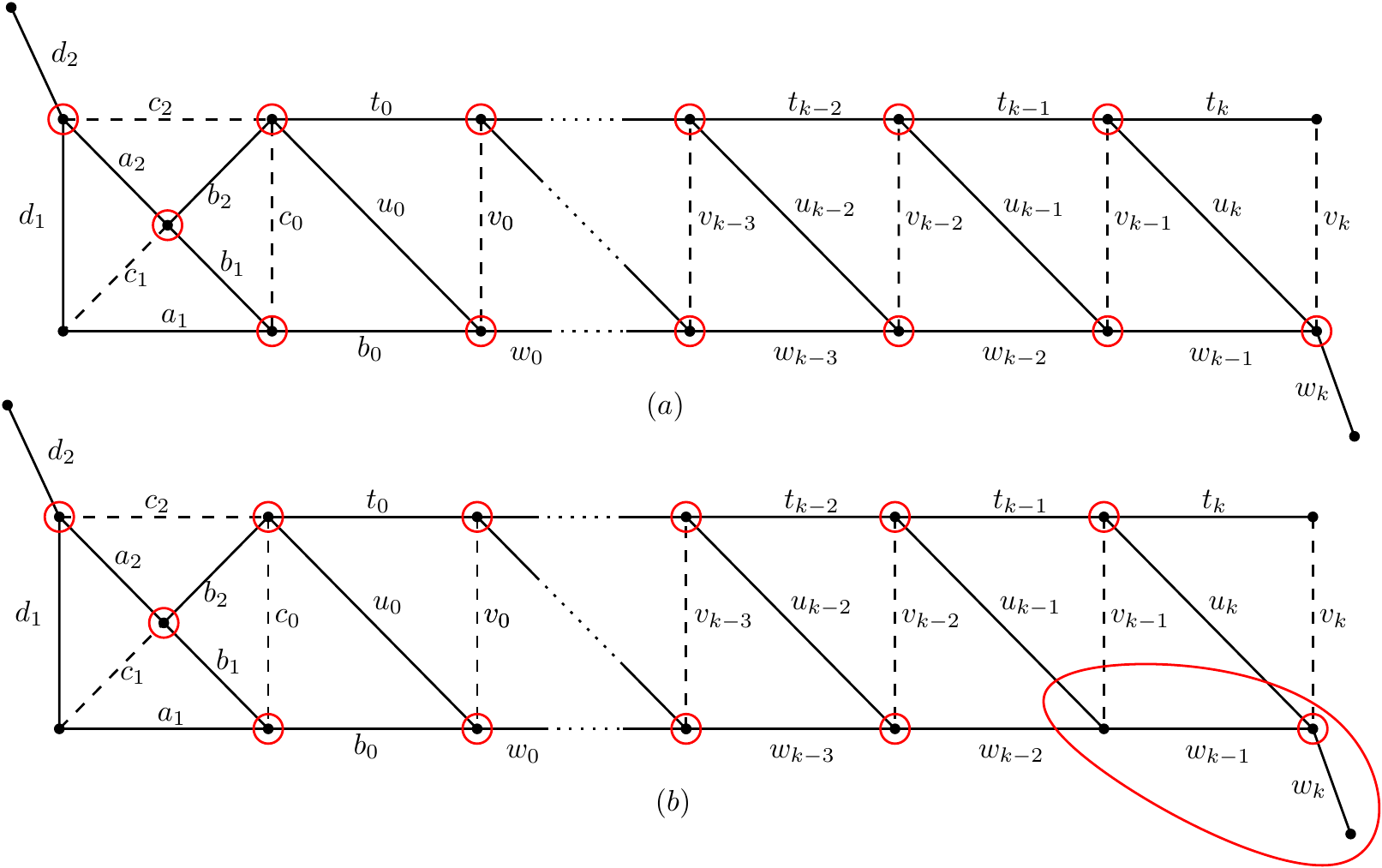}
\caption{In both (a) and (b), all  elements shown are distinct, except that $d_2$ may be $w_k$.  Furthermore, in (a), $k\geq 0$; and, in (b), $k \geq 1$ and $\{w_{k-2},u_{k-1},v_{k-1},u_k,v_k\}$ is a cocircuit.}
\label{bonesaws0}
\end{figure}

\begin{figure}[htb]
\center
\includegraphics[scale=0.72]{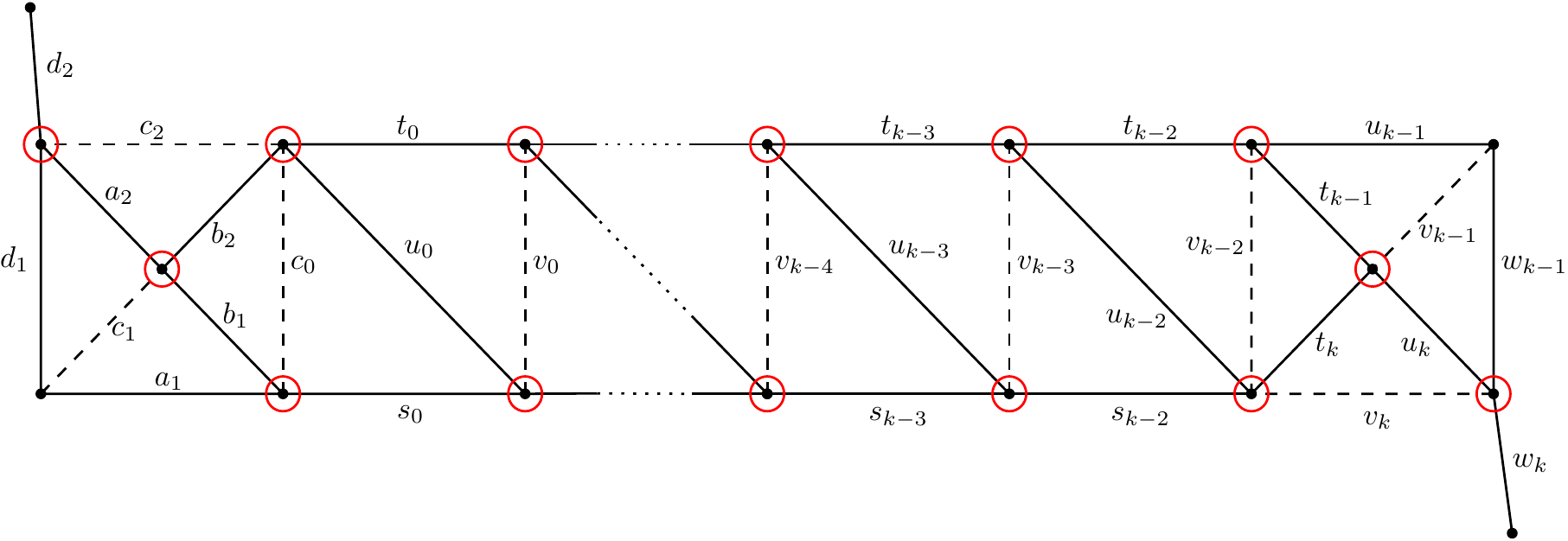}
\caption{In this configuration, $k\geq 2$ and the elements are all distinct except that $d_2$ may be $w_k$.}
\label{caterpillarwhole0}
\end{figure}

\begin{figure}[htb]
\center
\includegraphics[scale= 0.8]{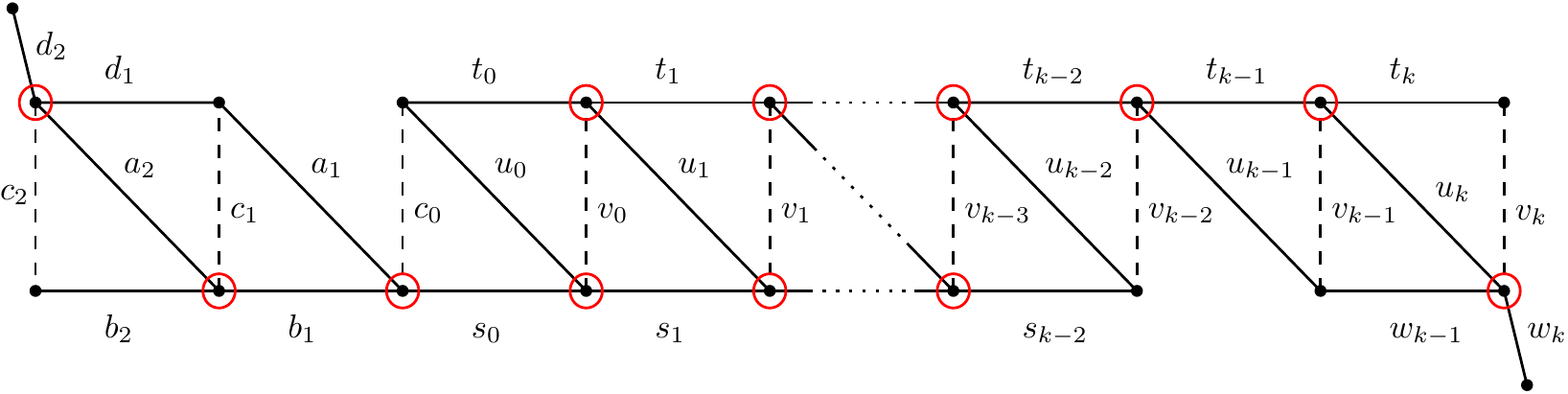}
\caption{The configuration in Figure~\ref{caterpillarwhole0} redrawn omitting two triangles and two $5$-cocircuits.}
\label{unravel}
\end{figure}

Consider the three configurations shown in  Figure~\ref{bonesaws0} and Figure~\ref{caterpillarwhole0} where the elements in each configuration are distinct except that $d_2$ may equal $w_k$. We refer to each of these configurations as an {\it enhanced quartic ladder}. Indeed, in each configuration, we can see a portion of a quartic ladder, which can be thought of as two interlocking bowtie strings, one pointing up and one pointing down. In each case, we focus on $M\ba c_2,c_1,c_0,v_0,v_1,\ldots,v_k$ saying that this matroid has been obtained from $M$ by an {\it enhanced-ladder move}. In Figure~\ref{unravel}, the configuration in 
Figure~\ref{caterpillarwhole0} has been redrawn omitting the triangles $\{c_0,b_1,b_2\}$ and $\{v_{k-2},t_{k-1},t_k\}$ as well as the cocircuits $\{b_2,c_0,c_2,t_0,u_0\}$ and $\{s_{k-2},u_{k-2},v_{k-2},t_k,v_k\}$. The ladder structure is evident there and the enhanced ladder move corresponds to deleting all of the dashed edges.

\begin{figure}[htb]
\center
\includegraphics[scale=0.72]{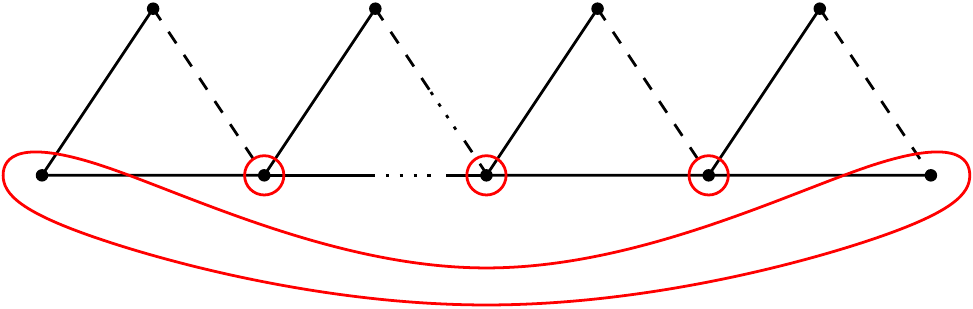}
\caption{A bowtie ring.  All elements are distinct.  The ring contains at least three triangles.}
\label{btringfig}
\end{figure}

For some $n \ge 2$, let $\{a_0,b_0,c_0\}, \{b_0,c_0,a_1,b_1\}, \{a_1,b_1,c_1\},\dots, \{a_n,b_n,c_n\}$ be a bowtie string in a matroid $M$.  Assume, 
in addition, that $\{b_n,c_n,a_0,b_0\}$ is a cocircuit. Then the string of bowties has wrapped around on itself as in Figure~\ref{btringfig}. We call the resulting structure a {\it ring of bowties} and denote it by $(\{a_0,b_0,c_0\}, \{b_0,c_0,a_1,b_1\}, \{a_1,b_1,c_1\},\dots, \{a_n,b_n,c_n\}, \{b_n,c_n,a_0,b_0\})$.  We also require that the elements in a bowtie ring are distinct, although this is guaranteed if $M$ is \ifc. 
We refer to each of the structures in Figure~\ref{laddery} as a {\it ladder structure} and we refer to removing the dashed elements in Figure~\ref{btringfig}  and Figure~\ref{laddery} as {\it trimming a ring of bowties} and  {\it trimming a ladder structure}, respectively.  

In the case that trimming a string of bowties in $M$ yields an \ifc\ matroid with an $N$-minor, we are able to ensure that the string of bowties belongs to one of the  more highly structured objects shown in one of Figures~\ref{bonesaws0}, \ref{caterpillarwhole0}, \ref{btringfig}, or \ref{laddery}.  
The following theorem is the main result of this paper.

\begin{figure}[htb]
\center
\includegraphics[scale=0.72]{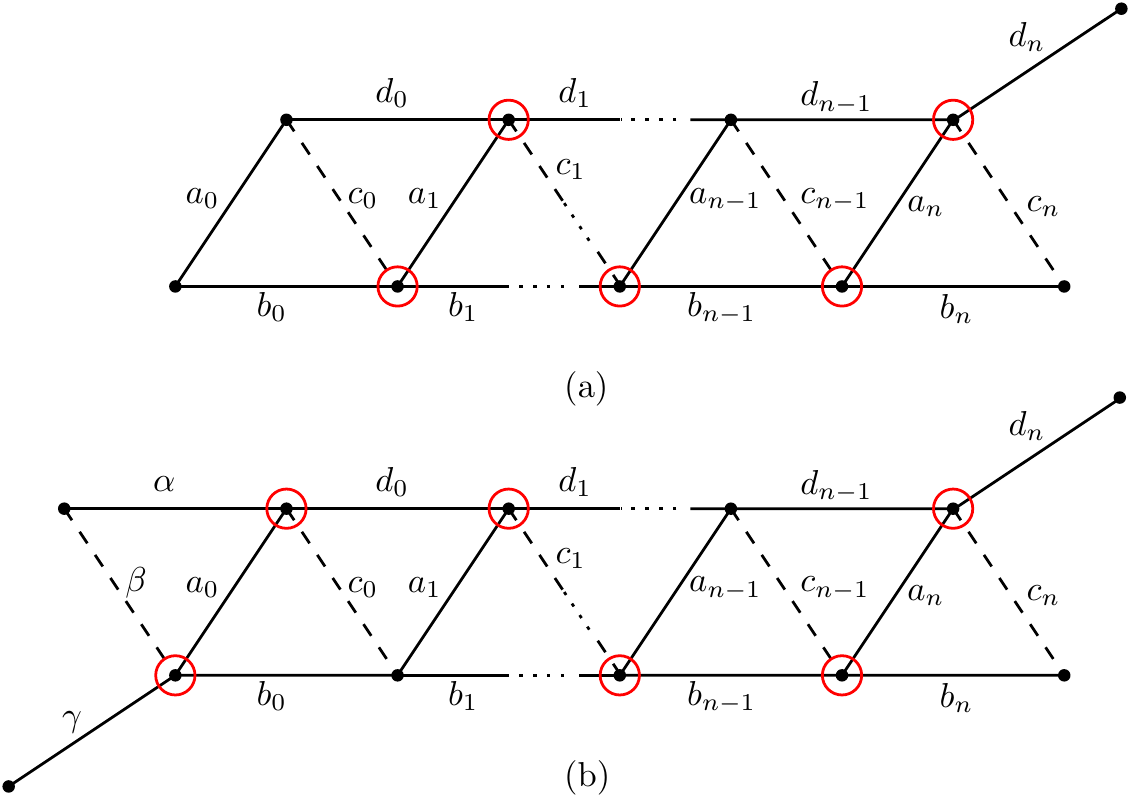}
\caption{In (a) and (b), the elements shown are distinct, with the exception that $d_n$ may be the same as $\gamma$ in (b).  Either $\{d_{n-2},a_{n-1},c_{n-1},d_{n-1}\}$ or $\{d_{n-2},a_{n-1},c_{n-1},a_n,c_n\}$ is a cocircuit in (a) and (b).  Either $\{b_0,c_0,a_1,b_1\}$ or $\{\be,a_0,c_0,a_1,b_1\}$ is also a cocircuit in (b).  Furthermore, $n\geq 2$.}
\label{laddery}
\end{figure}

\begin{theorem}
\label{killcasek}
Let $M$ and $N$ be \ifc\ binary matroids such that $|E(M)|\geq 13$ and $|E(N)|\geq 7$.  
Assume that $M$ has  a bowtie $(\{x_0,y_0,z_0\},\{x_1,y_1,z_1\},\{y_0,z_0,x_1,y_1\})$, where $M\ba z_0$ is \ffsc, $M\ba z_0,z_1$ has an $N$-minor, and $M\ba z_1$ is not \ffsc.  
Then one of the following holds.  
\begin{itemize}
\item[(i)] $M$ has a proper minor $M'$ such that $|E(M)|-|E(M')|\leq 3$ and $M'$ is \ifc\ with an $N$-minor; or 
\item[(ii)] $M$ contains an open rotor chain, a ladder structure, or a ring of bowties that can be trimmed to obtain an \ifc\ matroid with an $N$-minor; or
\item[(iii)] $M$ contains an enhanced quartic ladder from which an \ifc\ minor of $M$ with an $N$-minor can be obtained by an enhanced-ladder move.
\end{itemize}
\end{theorem}


\section{Preliminaries}
\label{prelim}

In this section, we give some basic definitions mainly relating to matroid connectivity. The subsequent section contains some straightforward properties of connectivity along with a lemma concerning bowties that distinguishes various cases whose analysis is fundamental to completing our work on the splitter theorem. The main result of this paper completely resolves what happens in one of these cases. In Section~\ref{outline}, we outline the proof of the main result,  Theorem~\ref{killcasek}.

Let $M$ and $N$ be matroids. We shall sometimes write $N\preceq M$ to indicate that $M$ has an {\it $N$-minor}, that is, a minor isomorphic to $N$. 
Now let $E$ be the ground set of $M$ and $r$ be its rank function.   
The {\it connectivity function} $\lambda_M$ of $M$ is defined on all subsets $X$ of $E$ by $\lambda_M(X) = r(X) + r(E-X) - r(M)$.  Equivalently, $\lambda_M(X) = r(X) + r^*(X) - |X|.$ We will sometimes abbreviate $\lambda_M$ as $\lambda$.
For a positive integer $k$, a subset $X$ or a partition $(X,E-X)$ of $E$ is 
{\em $k$-separating} if $\lambda_M(X)\le k-1$.  A $k$-separating partition $(X,E-X)$ of $E$ is a {\it $k$-separation} if  $|X|,|E-X|\ge k$. If $n$ is an integer exceeding one, a matroid  is {\it $n$-connected} if it has  no $k$-separations for all $k < n$.  This definition \cite{wtt} has the attractive property that a matroid is $n$-connected if and only if its dual is. Moreover, this matroid definition of $n$-connectivity is relatively compatible with the graph notion of $n$-connectivity when $n$ is $2$ or $3$. For example, when $G$ is a graph with at least four vertices and with no isolated vertices, $M(G)$ is a $3$-connected matroid if and only if $G$ is a $3$-connected simple graph. But the link between $n$-connectivity for matroids and graphs breaks down for $n \ge 4$. In particular, a $4$-connected matroid with at least six elements cannot have a triangle. Hence, for $r \ge 3$, neither $M(K_{r+1})$ nor $PG(r-1,2)$ is $4$-connected. This motivates the consideration of other types of $4$-connectivity  in which certain $3$-separations are allowed.

A matroid is \emph{internally $4$-connected} if it is $3$-connected and, whenever $(X,Y)$ is a $3$-separation, either $|X|=3$ or $|Y|=3$.  Equivalently, a $3$-connected matroid $M$ is \ifc\ if and only if, for every $3$-separation $(X,Y)$ of $M$, either $X$ or $Y$ is a triangle or a triad of $M$.  A graph $G$ without isolated vertices is {\it internally $4$-connected} if $M(G)$ is internally $4$-connected. 
 
In a matroid $M$, a subset $S$  of $E(M)$ is a {\em fan}   if $|S| \ge 3$ and there is an ordering $(s_1,s_2,\ldots,s_n)$
of $S$ such that $\{s_1,s_2,s_3\},\{s_2,s_3,s_4\},\ldots,\{s_{n-2},s_{n-1},s_n\}$  alternate between triangles and triads.   
We call $(s_1,s_2,\ldots,s_n)$   a {\it fan ordering} of 
$S$. 
For convenience, we will often refer to the fan ordering as the fan.  
We will be mainly concerned with $4$-element and $5$-element fans. By convention, we shall always view a fan ordering of a $4$-element fan as beginning with a triangle and we shall use the term {\it $4$-fan} to refer to both the $4$-element fan and such a fan ordering of it. Moreover, we shall use the terms {\it $5$-fan} and {\it $5$-cofan} to refer to the two different types of $5$-element fan where the first contains two triangles and the second two triads. Let $(s_1,s_2,\ldots,s_n)$ be a fan ordering of a fan $S$. When $M$ is $3$-connected and $n \ge 4$, every fan ordering of $S$ has its first and last elements in $\{s_1,s_n\}$. We call these elements  the {\it ends} of the fan while the elements of $S - \{s_1,s_n\}$  are called the {\it internal elements} of the fan.  When $(s_1,s_2,s_3,s_4)$ is a $4$-fan, our convention is that $\{s_1,s_2,s_3\}$ is a triangle, and we call $s_1$  the {\it guts    element} of the fan and $s_4$ the {\it coguts    element} of the fan since $s_1 \in \cl(\{s_2,s_3,s_4\})$ and $s_4 \in \cls(\{s_1,s_2,s_3\})$.

A set $U$ in a matroid $M$ is {\em fully closed} if it is closed in both $M$ and $M^*$. Let $(X,Y)$ be a partition of $E(M)$. If $(X,Y)$ is $k$-separating in $M$ for some positive integer $k$,  and $y$ is an element of  $Y$ that is also in $\cl(X)$ or $\cl ^*(X)$, then it is well known and easily checked that $(X \cup y,Y-y)$ is $k$-separating, and we say that we have {\it moved} $y$ into $X$. More generally, $(\fcl(X),Y-\fcl(X))$ is $k$-separating in $M$. Let $n$ be an   integer exceeding one. If $M$ is $n$-connected, an $n$-separation $(U,V)$ of $M$ is {\it sequential} if $\fcl(U)$ or $\fcl(V)$ is $E(M)$. In particular, when $\fcl(U) = E(M)$, there is an ordering $(v_1,v_2,\ldots,v_m)$ of the elements of $V$ such that 
$U \cup \{v_m,v_{m-1},\ldots,v_i\}$ is $n$-separating for all $i$ in $\{1,2,\ldots,m\}$. When this occurs, the set $V$ is called {\it sequential}. Moreover, if $n \le m$, then $\{v_1,v_2,\ldots,v_n\}$ is a circuit or a cocircuit of $M$. 
A \thc\ matroid is {\it \sfc} if all of its $3$-separations are sequential. It is straightforward to check that, when $M$ is  binary, a sequential set with $3, 4$, or $5$ elements is a fan. Let $(X,Y)$ be a \ths\ of a \thc\ binary matroid $M$. We shall frequently be interested in $3$-separations that indicate that $M$ is, for example, not \ifc. We call $(X,Y)$ or $X$ a {\it \ftv} if $|Y| \ge |X| \ge 4$. Similarly, $(X,Y)$ is a {\it \ffsv} if, for each $Z$ in $\{X,Y\}$, either $|Z| \ge 5$, or $Z$ is \ns.

\begin{figure}[htb]
\centering
\includegraphics{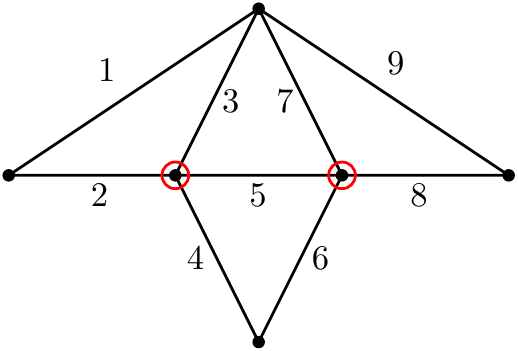}
\caption{A quasi rotor, where $\{2,3,4,5\}$ and $\{5,6,7,8\}$ are cocircuits.}
\label{rotorfig}
\end{figure}

Next we note another special structure from~\cite{zhou}, which has already arisen frequently in our work towards the desired splitter theorem. In an internally $4$-connected binary matroid $M$, we shall call 
$(\{1,2,3\},\{4,5,6\},\{7,8,9\},\{2,3,4,5\},\{5,6,7,8\},\{3,5,7\})$  a {\it quasi rotor} with {\it central triangle} $\{4,5,6\}$ and {\it central element $5$} 
if $\{1,2,3\}, \{4,5,6\},$ and $\{7,8,9\}$ are disjoint triangles in $M$ such that $\{2,3,4,5\}$ and $\{5,6,7,8\}$ are cocircuits and $\{3,5,7\}$ is a triangle.  
Section~\ref{qrsec} is dedicated to results concerning bowties and quasi rotors.  

\begin{figure}[htb]
\centering
\includegraphics[scale=0.75]{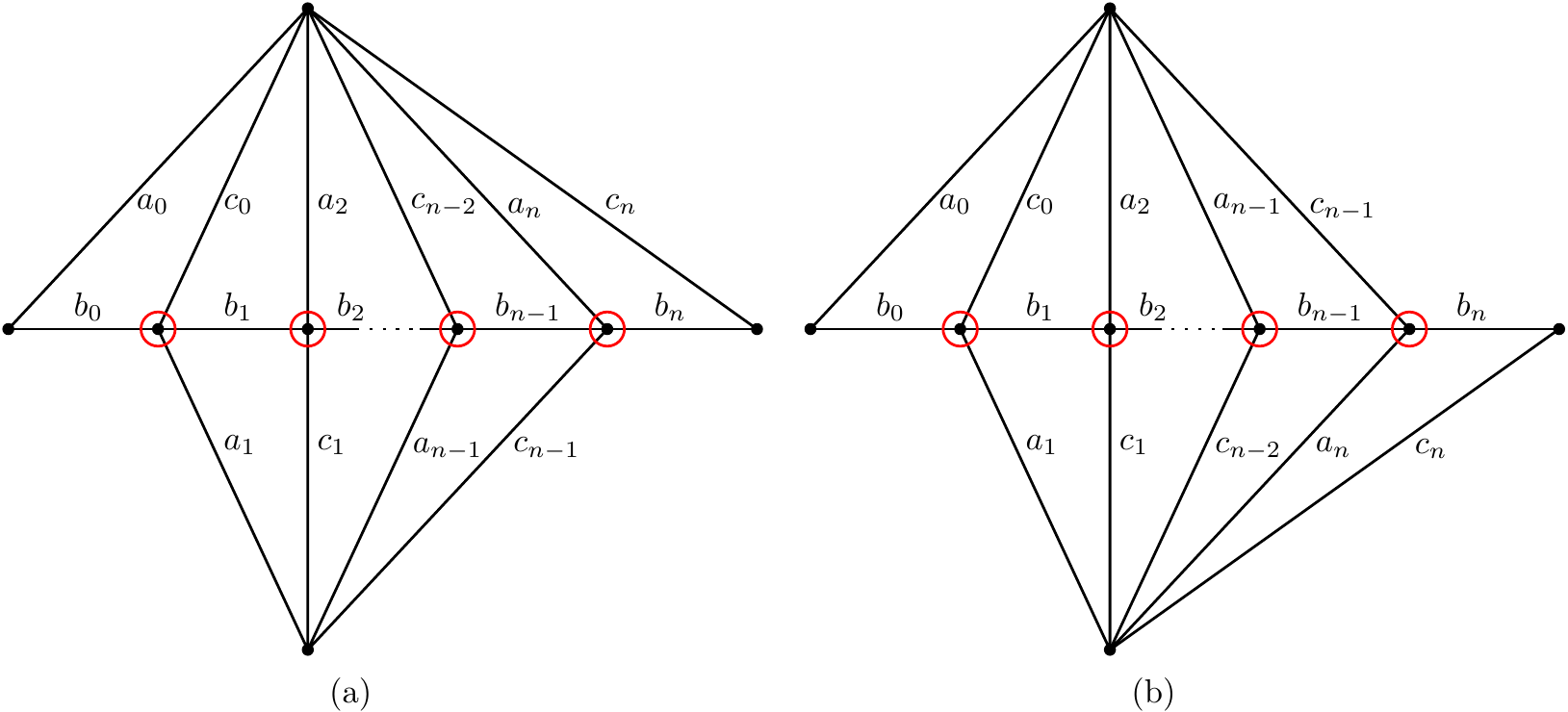}
\caption{Right-maximal rotor chain configurations.  In the case that $n$ is even, the rotor chain is depicted on the left.  If $n$ is odd, then the rotor chain has the form on the right.}
\label{rotorchainabc}
\end{figure}

For all non-negative integers $i$, it will be convenient  to adopt the convention throughout the paper of using $T_i$ and $D_i$ to denote, respectively, a triangle $\{a_i,b_i,c_i\}$ and a cocircuit $\{b_i,c_i,a_{i+1},b_{i+1}\}$.  
Let $M$ have $(T_0,T_1,T_2,D_0,D_1,\{c_0,b_1,a_2\})$ as a quasi rotor.  
Now $T_2$ may also be the central triangle of a quasi rotor.  
In fact, we may have a structure like one of the two depicted in Figure~\ref{rotorchainabc}.  
If $T_0,D_0,T_1,D_1,\dots ,T_n$ is a string of bowties in $M$, for some $n\geq 2$, and $M$ has the additional structure that $\{c_i,b_{i+1},a_{i+2}\}$ is a triangle for all $i$ in  $\{0,1,\dots ,n-2\}$, then we say that $((a_0,b_0,c_0),(a_1,b_1,c_1),\dots ,(a_n,b_n,c_n))$ is a {\it rotor chain}.  Clearly, deleting $a_0$ from a rotor chain gives an open rotor chain. 
Observe that every  three consecutive triangles within a rotor chain have the structure of a quasi rotor;   
that is, for all $i$  in $\{0,1,\dots ,n-2\}$, the sequence $(T_i,T_{i+1},T_{i+2},D_i,D_{i+1},\{c_i,b_{i+1},a_{i+2}\})$ is a quasi rotor.  
Zhou~\cite{zhou} considered a similar structure  called a {\it double fan of 
length $n-1$};  it  consists of all of the  elements in the rotor chain except for $a_0,b_0,b_n$, and $c_n$.  

If a rotor chain $((a_0,b_0,c_0),(a_1,b_1,c_1),\dots ,(a_n,b_n,c_n))$ cannot be extended to a rotor chain of the form $((a_0,b_0,c_0),(a_1,b_1,c_1),\dots ,(a_{n+1},b_{n+1},c_{n+1}))$, then we call it a {\it right-maximal rotor chain}.  

In the introduction, we defined a string of bowties. 
We say that such a string $T_0,D_0,T_1,D_1,\dots ,T_n$ is a {\it right-maximal bowtie string} in $M$ if $M$ has no triangle $\{u,v,w\}$ such that $T_0,D_0,T_1,D_1,\dots,T_n,\{x,c_n,u,v\},\{u,v,w\}$ is a bowtie string for some $x$ in $\{a_n,b_n\}$. Now let $(\{a_0,b_0,c_0\}, \{b_0,c_0,a_1,b_1\}, \linebreak \{a_1,b_1,c_1\},\dots, \{a_n,b_n,c_n\}, \{b_n,c_n,a_0,b_0\})$ be a ring of bowties. It is tempting to assume that, in such a ring, the set $\{b_0,b_1,\ldots,b_n\}$ is a circuit of $M$. Indeed, when $M$ is  internally $4$-connected, if $M = M(G)$ for some graph $G$, then it is not difficult to check that each of the cocircuits in the bowtie ring corresponds to the set of edges meeting some vertex of $G$. It follows that $\{b_0,b_1,\ldots,b_n\}$ is indeed a circuit of $M$.  However, if $M$ is not graphic, then $\{b_0,b_1,\ldots,b_n\}$ need not be a circuit of $M$.  To see this, observe that  $(\{a_0,b_0,c_0\}, \{b_0,c_0,a_1,b_1\}, \{a_1,b_1,c_1\},\dots, \{a_3,b_3,c_3\}, \{b_3,c_3,a_0,b_0\})$ is a bowtie ring in the bond matroid of the graph $G$ shown in Figure~\ref{biwheel}. However, $\{b_0,b_1,b_2,b_3\}$ is not a bond of $G$.

\begin{figure}[htb]
\center
\includegraphics[scale=0.8]{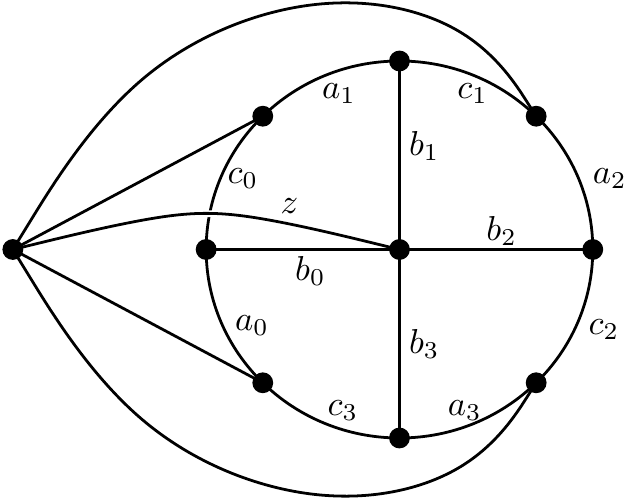}
\caption{The set $\{b_0,b_1,b_2,b_3\}$ is not a bond of this graph.}
\label{biwheel}
\end{figure}



\section{Outline of the proof}
\label{outline}

This section gives an outline of the strategy used to prove the  main theorem of the paper. The rest of the paper is concerned with implementing that strategy. Recall that a matroid $M$ is \ffsc\ if it is \thc\ and every $3$-separation in $M$ has, as one of its sides, a triad, a triangle, or a $4$-fan.  The hypotheses of the theorem 
present us with a bowtie $(\{a_0,b_0,c_0\}, \{a_1,b_1,c_1\}, \{b_0,c_0,a_1,b_1\})$ where $M\ba c_0$ is \ffsc, $M\ba c_0,c_1$ has an $N$-minor, and $M\ba c_1$ is not \ffsc. 

Because $M\ba c_1$ is not \ffsc, Lemma~\ref{6.3rsv} gives us that $\{a_1,b_1,c_1\}$ is the central triangle of a quasi rotor $(T_0,T_1,T_2,D_0,D_1, \{c_0,b_1,a_2\})$ where we recall that   $T_i$   denotes a triangle $\{a_i,b_i,c_i\}$, and $D_i$   denotes a $4$-cocircuit $\{b_i,c_i,a_{i+1},b_{i+1}\}$. Thus, for some $n \ge 2$, we have a right-maximal rotor chain $((a_0,b_0,c_0),(a_1,b_1,c_1),\dots,(a_n,b_n,c_n))$. In Lemma~\ref{betweenbts}, we prove that either  the theorem holds, or $M$ has such a right-maximal rotor chain in which 
$M\ba c_0,c_1,\dots,c_n$ is \sfc\ with an $N$-minor, $M/b_i$ has no $N$-minor for all $i$ in $\{1,2,\dots,n-1\}$, and $M\ba c_n$ is \ffsc, while $M$ has a triangle $T_{n+1}$ and a $4$-cocircuit $D_n$ such that 
$T_0,D_0,T_1,D_1,\dots,T_n,D_n,T_{n+1}$ is a bowtie string in $M$, and 
$M\ba c_0,c_1,\ldots,c_{n+1}$ has an $N$-minor.

We extend this bowtie string to a right-maximal bowtie string 
$T_0,D_0,T_1,D_1,\dots,T_n,D_n,\dots,T_{k}$. Then $k \ge n+1$. In Lemma~\ref{stringybark}, we show that $M\ba c_0,c_1,\dots,c_k$ has an $N$-minor. In Lemma~\ref{reachtheend}, we deal with the case when this bowtie string does not wrap around on itself to form a bowtie 
ring, and we show that the theorem holds in this case.

We may now assume that $(T_k,T_0,\{b_k,a_k,a_0,b_0\})$ is  a bowtie, so 
$(T_0,D_0,T_1,D_1,\dots,T_k,D_k)$ is a bowtie ring. In that case, Lemma~\ref{btring} shows that, when the theorem does not hold, $M\ba c_0,c_1,\ldots,c_k$ is \sfc\,  $M$ has a
triangle $\{e_1,f_1,g_1\}$ that is disjoint from $T_0 \cup T_1 \cup \dots \cup T_k$,  and $M$ has a cocircuit $\{e_1,f_1,c_j,z_j\}$ for some $j$ in $\{0,1,\dots,k\}$ where $z_j$ is $b_j$ or $a_j$.  Situations corresponding to the two possibilities for $z_j$ are illustrated in Figure~\ref{ibis_n_haz}. The proof of the theorem is completed when these two cases are treated in Lemma~\ref{hazwaste}.  

\begin{figure}[htb]
\centering
\includegraphics[scale=1.]{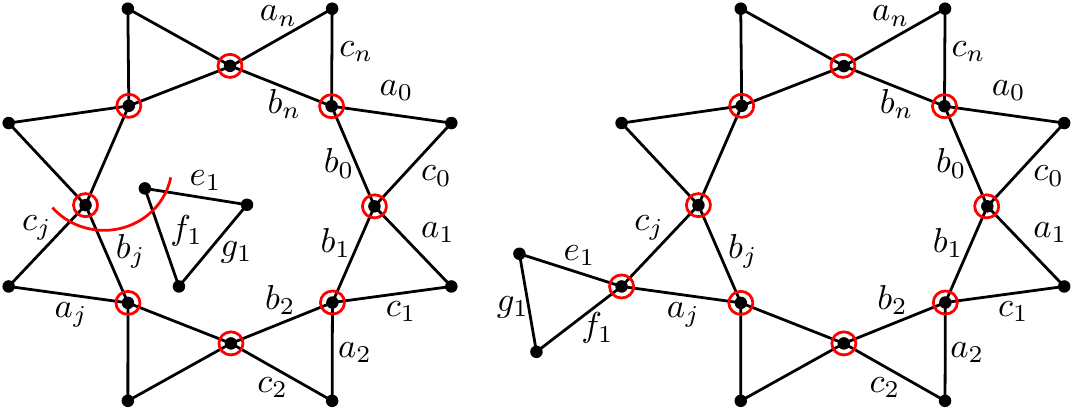}
\caption{The two quintessential examples of obstructions to $M\ba c_0,c_1,\dots, c_n$ being \ifc\ when $M$ contains a bowtie ring.  The  curved line on the left indicates that $\{e_1,f_1,b_j,c_j\}$ is a cocircuit.}
\label{ibis_n_haz}
\end{figure}

In the two graph diagrams shown in Figure~\ref{ibis_n_haz}, the diagrams suggest that the set 
$\{b_0,b_1,\ldots,b_n\}$ is a circuit and this is certainly true when $M$ is graphic. But this need not be so for an arbitrary \ifc\ binary matroid $M$. This means that the reader needs to exercise some caution in dealing with such diagrams when they wrap around. Each of the triangles shown   is certainly a circuit of the matroid $M$, and one can infer that other cycles in the graphs are circuits in the matroid when, for example, such cycles can be built by taking symmetric differences of sets of overlapping triangles.

In the next section, we note some properties of bowties and quasi rotors. In Section~\ref{bowtiestring}, we prove some results for strings of bowties, while Section~\ref{ladderseg} presents some results for quartic ladder segments. The rest of the proof of our splitter theorem both here and in the sequels to this paper will essentially amount to an analysis of the behavior of bowtie strings and ladder segments.

\section{Some results for bowties and quasi rotors}
\label{qrsec}

The following lemma will be used repeatedly throughout the paper. It extends \cite[Lemma 2.2]{cmoII} to include seven-element \ifc\ binary matroids, noting that the only such matroids are $F_7$ and its dual, and neither of these matroids has a $4$-element fan.

\begin{lemma}
\label{2.2} 
Let $N$ be an \ifc\  matroid having at least seven elements and $M$ be a binary matroid with an $N$-minor. If $(s_1,s_2,s_3,s_4)$ is a $4$-fan in $M$, then $M\ba s_1$ or $M/s_4$ has an $N$-minor. If $(s_1,s_2,s_3,s_4,s_5)$ is a $5$-fan in $M$, then either $M\ba s_1,s_5$ has an $N$-minor, or both $M\ba s_1/s_2$ and $M\ba s_5/s_4$ have $N$-minors. 
\end{lemma}

We will use the following elementary result frequently when considering bowtie structures.  

\begin{lemma}
\label{bowwow}
Let $M$ be  an \ifc\ matroid having at least ten elements. If $(\{1,2,3\}, \{4,5,6\}, \{2,3,4,5\})$ is a bowtie in $M$, then $\{2,3,4,5\}$ is the unique $4$-cocircuit of $M$ that meets both $\{1,2,3\}$ and  $\{4,5,6\}$.
\end{lemma}

\begin{proof} Suppose $M$ has a $4$-cocircuit $C^*$ other than $\{2,3,4,5\}$ that meets both $\{1,2,3\}$ and  $\{4,5,6\}$. Then, by \ort, $C^* \subseteq \{1,2,3,4,5,6\}$. Hence $\lambda(\{1,2,3,4,5,6\}) \le 2$; a \cn\ as $|E(M)| \ge 10$.
\end{proof}

Observe that the last result need not hold if $|E(M)| < 10$. Indeed, every bowtie of  $M^*(K_{3,3})$ has three distinct $4$-cocircuits meeting both of its triangles.

The following result uses Lemma~\ref{bowwow} to make a straightforward modification of  \cite[Lemma   2.6]{cmoIV}.

\begin{lemma}
\label{6.3rsv}
Let $(\{1,2,3\}, \{4,5,6\}, \{2,3,4,5\})$ be a bowtie in an \ifc\ binary matroid $M$ with $|E(M)| \ge 13$. Then $M\ba 6$ is \ffsc\ unless $\{4,5,6\}$ is the central triangle of 
a quasi rotor $(\{1,2,3\},\{4,5,6\},\{7,8,9\},\{2,3,4,5\},\{y,6,7,8\},\{x,y,7\})$ for some $x$ in $\{2,3\}$ and some $y$ in $\{4,5\}$.  
In addition, when $M\ba 6$ is \ffsc, one of the following holds. 
\begin{itemize}
\item[(i)] $M\ba 6$ is \ifc; or
\item[(ii)] $M$ has a triangle $\{7,8,9\}$ disjoint from $\{1,2,3,4,5,6\}$ such that $(\{4,5,6\}, \{7,8,9\}, \{a,6,7,8\})$ is a bowtie for some $a$ in $\{4,5\}$; or
\item[(iii)] every \ftv\ of $M\ba 6$ is a $4$-fan of the form $(u,v,w,x)$, where $M$ has a triangle $\{u,v,w\}$ and a cocircuit $\{v,w,x,6\}$ where $u$ and  $v$  are in $\{2,3\}$ and $\{4,5\}$, respectively, and $|\{1,2,3,4,5,6,w,x\}|=8$; or
\item[(iv)] $M\ba 1$ is \ifc\ and $M$ has a triangle $\{1,7,8\}$ and a cocircuit $\{a,6,7,8\}$ where $|\{1,2,3,4,5,6,7,8\}| = 8$ and  $a \in \{4,5\}$. 
\end{itemize}
\end{lemma}

In this paper, we are focussing  on  the case in which $(\{1,2,3\},\{4,5,6\},\{2,3,4,5\})$ is a bowtie and $M\ba 6$ has an $N$-minor but is not \ffsc;  
that is, we are concerned with the case when $\{4,5,6\}$ is the central triangle of a quasi rotor.  The remaining cases that arise from this lemma will be treated in~\cite{cmoVII} and~\cite{cmoVIII}. 

We proved the following result in~\cite[Theorem 6.1]{cmochain}.

\begin{theorem}
\label{rotor}
Let $M$ be an \ifc~ binary matroid having $(\{1,2,3\},\{4,5,6\},\{7,8,9\},\{2,3,4,5\},\{5,6,7,8\},\{3,5,7\})$ as a quasi rotor   
and having at least thirteen elements. Then either 
\begin{itemize}
\item[(i)]  $M\ba 1$, $M\ba 9$, $M\ba 1/2$, or $M\ba 9/8$ is \ifc; or
\item[(ii)] $M$ has triangles $\{6,8,10\}$ and $\{2,4,11\}$ such that $|\{1,2,\ldots,11\}| = 11$, and  $M\ba 3,4/5$ is \ifc.
\end{itemize}
\end{theorem}

We show next that if $M$ contains an element $e$ that is in two triangles of a quasi rotor and $M/e$ has an $N$-minor, then $M$ has an \ifc~ minor $M'$ that has an $N$-minor and satisfies $|E(M) - E(M')| \in \{1,2,3\}$.

\begin{lemma}
\label{rotorwin} Let $M$ be an \ifc~ binary matroid having $(\{1,2,3\},\{4,5,6\},\{7,8,9\},\{2,3,4,5\},\{5,6,7,8\},\{3,5,7\})$ as a quasi rotor   
and having at least thirteen elements. Let $N$ be an \ifc~ matroid containing at least seven elements such that $M/e$ has an $N$-minor for some $e$ in $\{3,5,7\}$. Then one of $M\ba 1$, $M\ba 9$, $M\ba 1/2$,  $M\ba 9/8$, or $M\ba 3,4/5$ is \ifc\ with an $N$-minor.
\end{lemma}

\begin{proof} By symmetry, it suffices to prove the theorem for $e$ in $\{5,7\}$. We prove it first for $e = 5$ and then use that case to prove it for $e = 7$.

Assume $N \preceq M/5$. As $M/5$ has $\{3,7\}$ and $\{4,6\}$ as circuits, we deduce that $N \preceq M/5\ba 3,4$. The theorem holds for $e = 5$ if $M/5\ba 3,4$ is \ifc, so assume it is not. Then, by Theorem~\ref{rotor}, one of $M\ba 1$, $M\ba 9$, $M\ba 1/2$, or $M\ba 9/8$ is \ifc. By symmetry, if we can show that $N\preceq M\ba 1/2$, then the theorem will follow for $e= 5$. Now $N \preceq M/5\ba 3,4$ and, in $M\ba 3,4$, the elements $2$ and $5$ are in series. Hence $M/5\ba 3,4 \cong M/2\ba 3,4$. But, in $M/2$, the elements $1$ and $3$ are in parallel. Hence $M/2\ba 3,4 \cong M/2\ba 1,4$. We deduce that $N\preceq M\ba 1/2$, so the theorem holds for $e= 5$.

Now suppose that $N \preceq M/7$ but $N \not \preceq M/5$. Then $N\preceq M/7 \ba 5,8$. Now 
$M/7 \ba 5,8 \cong M\ba 5,8/6 \cong M/6\ba 4,8$. Since $(1,2,3,5,7)$ is a fan of $M\ba 4,8$, 
Lemma~\ref{2.2} implies  that $N\preceq M\ba 4,8\ba 1,7$ or $N\preceq M\ba 4,8/5\ba 7$. The second possibility is excluded because $N \not \preceq M/5$. Thus $N \preceq M\ba 7,8$. As $M \ba 7,8$ has $\{5,6\}$ as cocircuit, we deduce that $N \preceq M/5$; a contradiction. 
\end{proof}




\section{Strings and rings of bowties} 
\label{bowtiestring}

Strings and rings of bowties will   feature  prominently throughout the rest of the paper. 
This section develops some properties of such structures. 
Recall that, for each natural number $i$, we are using $T_i$ and $D_i$ to denote a triangle $\{a_i,b_i,c_i\}$  and a cocircuit $\{b_i,c_i,a_{i+1},b_{i+1}\}$, respectively.  

When bowtie strings appear in our theorems, they do so embedded in more highly structured configurations, specifically, open rotor chains, ladder structures, enhanced quartic ladders, and bowtie rings. Of the last four structures, bowtie rings allow for the most variation in the surrounding structure. Suppose that we have  a bowtie string as shown in Figure~\ref{bowtiechainfig} and that, in addition, $n \ge 2$ and $\{b_n,c_n,a_0,b_0\}$ is a cocircuit, $D_n$. Then $(T_0,D_0,T_1,D_1,\dots ,T_n,D_n)$ is a ring of bowties. Suppose that this ring occurs in an \ifc\ graphic matroid $M(G)$. As observed in the introduction, for all $i$ in $\{0,1,\ldots,n\}$, there is a vertex $v_i$ of $G$ such that $\{b_i,c_i,a_{i+1},b_{i+1}\}$ is the set of edges of $G$ meeting $v_i$. Hence $\{b_0,b_1,\dots,b_n\}$ is a cycle of $G$. For each $i$, let $w_i$ be the vertex that meets $a_i$ and $c_i$. The definition of a ring of bowties does not require $w_0,w_1,\dots,w_n$ to be distinct. Indeed, subject to some constraints to ensure that $G$ stays \ifc, they need not be. As an example, take a copy of $K_{10}$ with vertex set $\{u_0,u_1,\dots,u_9\}$ and let $n= 99$, so we have a bowtie ring with $100$ triangles. Suppose that, for all $j$ in $\{0,1,\dots,99\}$, the vertex $w_j$ is identified with $u_t$ 
where $j \equiv t\mod 10$. Let $G$ be the resulting $110$-vertex graph. Then $M(G)$ is easily shown to be \ifc. For $N = M(G) \ba c_0,c_1,\dots,c_{99}$, we see that $N$ is \ifc\ while there is no \ifc\ proper minor of $M$ that has an $N$-minor. 

As another example, let $H$ be the octahedron graph. Clearly $M(H)$ contains a bowtie ring $(T_0,D_0,T_1,D_1,T_2,D_2,T_3,D_3)$ with exactly four triangles. In addition, $M(H)$ has a $4$-cocircuit $D'$ such that $(T_0,D_0,T_1,D_1,T_2,D')$ is a bowtie ring. A bowtie ring is   {\it minimal} exactly when no proper subset of its set of triangles is the set of triangles of a bowtie ring.  Next we show that when we trim a bowtie ring to produce an \ifc\ matroid, that bowtie ring must be minimal. 

\begin{lemma}
\label{minbtring}
Let $(T_0,D_0,T_1,D_2,\dots ,T_n,D_n)$ be a bowtie ring in an \ifc\ binary matroid $M$ where $|E(M)| \ge 10$.  If $M\ba c_0,c_1,\dots ,c_n$ is \ifc,   
then $(T_0,D_0,T_1,D_1,\dots ,T_n,D_n)$ is a minimal bowtie ring.  
\end{lemma}

\begin{proof}
Suppose the lemma does not hold.  
Then some proper subset of  $\{T_0,T_1,\dots ,T_n\}$ is the set of triangles of a bowtie ring.  Hence $M$ has a $4$-cocircuit that meets two triangles in this set and that is not contained in $\{D_0,D_1,\dots ,D_n\}$.  
Choose  such a $4$-cocircuit $C^*$ to maximize  $|C^* \cap \{c_0,c_1,\dots ,c_n\}|$.  

Take two distinct integers $i$ and $j$ in $\{0,1,\dots ,n\}$ such that $C^*$ meets $T_i$ and $T_j$.  
Lemma~\ref{bowwow} implies that $T_i$ and $T_j$ are not consecutive triangles in the ring. By orthogonality,   $C^*\subseteq T_i\cup T_j$.  
If $C^*\supseteq \{a_i,b_i\}$, then $C^*\btu D_{i-1}$ is a $4$-cocircuit that is not contained in $\{D_0,D_1,\dots ,D_n\}$  and that has a larger intersection with $\{c_0,c_1,\dots ,c_n\}$ than $C^*$ does; a \cn.  
Thus $c_i \in C^*$ and, by symmetry, $c_j \in C^*$.  
Then $M\ba c_0,c_1,\dots ,c_n$ has a $2$-cocircuit contained in $\{a_i,b_i,a_j,b_j\}$. This is a \cn\ as   $|E(M\ba c_0,c_1,\dots ,c_n)| \ge 6$ since $n \ge 2$. 
\end{proof} 

The following property of bowtie strings will be frequently used.


\begin{lemma}
\label{stringswitch}
Let $T_0,D_0,T_1,D_1,\dots ,T_n$ be a string of bowties in a matroid $M$. Then, 
for all $k$ in $\{1,2,\ldots,n\}$, 
\begin{align*}
M\ba c_0,c_1,\ldots,c_n/b_n &\cong M\ba c_0,c_1,\ldots c_{k-1}/b_k\ba a_k,a_{k+1},\ldots,a_n\\
&\cong M\ba c_0,c_1,\ldots c_{k-1}/b_{k-1}\ba a_k,a_{k+1},\ldots,a_n\\
&\cong M\ba a_0,a_1,\ldots,a_n/b_0.
\end{align*}

\end{lemma} 

\begin{proof}
Evidently, $M\ba c_0,c_1,\ldots,c_n/b_n \cong  M\ba c_0,c_1,\ldots,c_{n-1}/b_n\ba a_n \cong 
M\ba c_0,c_1,\ldots,c_{n-1}/b_{n-1}\ba a_n$. The lemma follows by repeatedly applying this observation.
\end{proof}

\begin{lemma}
\label{deletecs}
Let $T_0,D_0,T_1,D_1,\dots,T_n$ be a string of bowties in an \ifc\ binary matroid $M$, where 
$|E(M)| \ge 8$.  
Then either 
$M\ba c_0,c_1,\dots ,c_n$ is \thc, or it has $a_j$ or $b_j$ in a cocircuit of size at most two for some $j$ in $\{2,3,\dots ,n\}$.  
\end{lemma}

\begin{proof}
As $|E(M)| \ge 8$ and $M$ is \ifc\ having $c_0$   in a triangle, 
\cite[Lemma 2.5]{cmochain} 
 implies that $M\ba c_0$ is \thc.  Now if $M\ba c_0,c_1,\dots ,c_{n}$ is  \thc, then the lemma holds. Hence we may assume there is an element $i$ in $\{1,2,\ldots,n\}$ such that 
 $M\ba c_0,c_1,\dots ,c_{i}$ is not \thc\  but   $M\ba c_0,c_1,\dots ,c_{j}$ is \thc\ for all 
 $j$ in $\{0,1,\dots, i-1\}$.  
Then Tutte's Triangle Lemma~\cite{wtt} (or see \cite[Lemma 8.7.7]{oxrox}) implies that either $M\ba c_0,c_1,\dots ,c_{i-1}\ba x$ is \thc\ for some $x$ in $\{a_i,b_i\}$, or $c_{i}$ is in a triad of $M\ba c_0,c_1,\dots ,c_{i-1}$.  
The former does not occur because  $M\ba c_0,c_1,\dots ,c_{i-1}$ has $\{b_{i-1},a_i,b_{i}\}$ as a cocircuit. Thus $c_{i}$ is in a triad of $M\ba c_0,c_1,\dots ,c_{i-1}$.  By \ort, this triad meets $\{a_i,b_i\}$. Then $M\ba c_0,c_1,\ldots,c_i$ has a cocircuit $\{x,y\}$ where $x \in \{a_i,b_i\}$ but $y \not \in \{a_i,b_i\}$. Thus  $M$ has a cocircuit $C^*$ such that 
$\{c_i,x,y\} \subseteq C^* \subseteq \{y,a_i,b_i,c_i,c_{i-1},\dots ,c_0\}$.  
We know that $C^*$ contains at least four elements, since it meets a triangle in $M$.  
Suppose $c_i\neq a_0$. Then, by \ort\ between $C^*$ and  each of $T_0,T_1,\dots ,T_{i-1}$, and $T_i$, we deduce  that $C^*=\{x,c_i,y,c_j\}$ for some $j\neq i$, where  $y\in\{a_j,b_j\}$. Now  Lemma~\ref{bowwow} implies that $\{i,j\}\neq \{0,1\}$ so the lemma holds when $c_i \neq a_0$. We may now assume that $c_i = a_0$. Then   
 $i=n$.  Moreover, $n \ge 2$, otherwise the $4$-element set $D_0$ is both a circuit and a cocircuit of $M$; a contradiction. 
Orthogonality between $C^*$ and each of $T_0,T_1,\dots ,T_{i-1}$, and $T_i$ implies that $C^*=\{x,y,c_n,c_0\}$.  
Then $M\ba c_0,c_1,\dots ,c_n$ has $a_n$ or $b_n$ in a cocircuit of size at most two and again the lemma holds.  
\end{proof}

The following modifies the argument used to prove \cite[Lemma 11.4]{cmochain}.

\begin{lemma}
\label{ring} 
Let $M$ be an \ifc\ matroid having at least ten elements. Suppose that  $M$ has 
 $T_0,D_0,T_1,D_1,\ldots,T_n$ as a  string of bowties where $T_0$ and $T_n$ are disjoint.  Assume  that $(T_n,T_{n+1},D)$ is a bowtie of $M$ for some $4$-cocircuit $D\neq D_{n-1}$ where $\{c_n,a_{n+1},b_{n+1}\}\subseteq D$. 
If $T_{n+1}$ meets $T_i$ for some $i \le n$, then 
 \begin{itemize}
 \item[(i)]  $T_{n+1}  =  T_j$ for some $j$ with $0\le j \le {n-2}$; or
 \item[(ii)] $T_{n+1} \cap (T_0 \cup T_1 \cup \cdots \cup T_n) = T_{n+1} \cap T_0 = \{a_0\} = \{c_{n+1}\}$. 
 \end{itemize}
Furthermore, if $\{x,c_n,a_0,b_0\}$ is a cocircuit for some $x $ in $\{a_n,b_n\}$, then (i)  holds. 
 \end{lemma}

\begin{proof}  Let $j$ be the greatest integer, $i$, in $\{0,1,\ldots,n\}$ such that $T_{n+1} \cap T_i \neq \emptyset$. We show first that (i) or (ii) holds. Since $(T_n,T_{n+1},D)$ is a bowtie, $j \le n-1$. 
If $T_j$ meets $\{a_{n+1},b_{n+1}\}$, then, by \ort, $T_j$ contains $\{a_{n+1},b_{n+1}\}$, so $T_j = T_{n+1}$. Also, if $T_{n+1}$ meets $\{b_j,c_j\}$, then $T_{n+1}$ contains $\{b_j,c_j\}$, so $T_j = T_{n+1}$. Thus either
\begin{itemize}
\item[(a)] $T_j = T_{n+1}$; or 
\item[(b)]  $T_j \cap T_{n+1} = \{a_j\} = \{c_{n+1}\}$.
\end{itemize}

If (a) holds,   then Lemma~\ref{bowwow} implies that $j\leq n-2$, so (i) holds.  
Hence we may assume that (b) holds. Then $j > 0$, otherwise (ii) holds.  Thus $T_{n+1}$ meets $D_{j-1}$, so $T_{n+1}$ meets 
$\{b_{j-1},c_{j-1}\}$. Hence $a_{n+1}$ or $b_{n+1}$ is in $T_{j-1}$, that is, $D$ meets $T_{j-1}$. But $T_n \cap T_{j-1} = \emptyset$, so $\{a_{n+1},b_{n+1}\} \subseteq T_{j-1}$. Thus $T_{n+1} = T_{j-1}$. Hence $T_{n+1} \cap T_j = \emptyset$; a contradiction. We conclude that (i) or (ii) holds. 

Finally,  suppose that $\{x,c_n,a_0,b_0\}$ is a cocircuit for some $x\in\{a_n,b_n\}$ but that (i) does not hold. Then (ii) holds so \ort\ implies that $T_{n+1}$ meets $T_n$; a \cn.  
\end{proof}

\begin{lemma}
\label{btring}
Let $M$ be an \ifc\ binary matroid containing at least thirteen elements.  
Suppose that $M$ has a ring of bowties $(T_0,D_0,T_1,D_1,\dots, T_k, D_k)$.
Then 
 \begin{itemize}
 \item[(i)] $M\ba c_0,c_1,\dots ,c_k$ is \ifc; or
\item[(ii)] $M\ba c_0,c_1,\dots, c_k$ is \sfc\ and every $4$-fan of it has the form $(1,2,3,4)$ where   $\{1,2,3\}$ is disjoint from $T_0\cup T_1\cup\dots\cup T_k$, and $M$ has a cocircuit $\{2,3,4,c_i\}$ for some $i$ in $\{0,1,\dots ,k\}$ where $4\in\{a_i,b_i\}$; or 
\item[(iii)]  $M\ba c_0,c_1,\dots ,c_k$ has a $1$- or $2$-element cocircuit that meets $\{a_2,b_2,a_3,b_3,\dots ,a_k,b_k\}$.  
\end{itemize}
\end{lemma}
\begin{proof}
Assume that (iii) does not hold.  
Then Lemma~\ref{deletecs} implies that $M\ba c_0,c_1,\dots, c_i$ is \thc\ for all $i$ in $\{0,1,\dots, k\}$.  Let $M' = M\ba c_0,c_1,\dots, c_k$. We shall show next that 

\begin{sublemma}
\label{newnu0}
$M'$ is \sfc.
\end{sublemma}

Assume that this fails. First we show that 

\begin{sublemma}
\label{newnu}
$M'$ has a non-sequential $3$-separation $(X,Y)$ such that, for each $i$ in $\{0,1,\dots,k\}$, the pair $\{a_i,b_i\}$ is contained in $X$ or $Y$.
\end{sublemma}

Certainly $M'$ has a non-sequential $3$-separation $(X,Y)$ in which $\{a_0,b_0,b_k\}$ is contained in $X$ or $Y$. Take the smallest index, $i$, such that $\{a_i,b_i\}$ meets both $X$ and $Y$. We may assume that $a_i\in X$ and $b_i \in Y$. Now $\{a_{i-1},b_{i-1}\}$ is contained in $X$ or $Y$. Then $\{b_{i-1},a_i\} \subseteq X$ or $\{b_{i-1},b_i\} \subseteq Y$. 
Thus $b_i \in \cls_{M'}(X)$ or $a_i \in \cls_{M'}(Y)$, respectively. Hence 
$(X \cup b_i, Y- b_i)$ or $(X- a_i, Y \cup a_i)$ is a non-sequential $3$-separation of $M'$ in which $\{a_i,b_i\}$ is contained in one side. By repeating this process, we see that \ref{newnu} holds.

By \ref{newnu}, each $c_i$ is in the closure of $X$ or $Y$, so $M$ has a non-sequential $3$-separation; a \cn. We conclude that \ref{newnu0} holds.

We may assume that $M'$ is not \ifc\ otherwise the lemma holds. To complete the proof, we show the following.

\begin{sublemma}
\label{newnu2}
If $(1,2,3,4)$ is a $4$-fan of $M'$, then    $\{1,2,3\}$ is disjoint from $T_0\cup T_1\cup\dots\cup T_k$, and $M$ has  $\{2,3,4,c_i\}$ as a cocircuit for some $i$ in $\{0,1,\dots ,k\}$ where $4\in\{a_i,b_i\}$.
\end{sublemma}

Evidently $M$ has $\{1,2,3\}$ as a triangle and has a cocircuit $C^*$ such that 
$\{2,3,4\} \subsetneqq C^* \subseteq \{2,3,4,c_0,c_1,\dots,c_k\}$.
Suppose $c_i \in C^*$. Then $\{2,3,4\}$ meets $\{a_i,b_i\}$. We shall show next that $\{2,3\}$ avoids $\{a_i,b_i\}$. Assume the contrary. Then the cocircuit $\{b_{i-1},a_i,b_i\}$ in $M'$ implies that $\{1,2,3\}$ contains $\{b_{i-1},a_i\}$ or $\{b_{i-1},b_i\}$ where all subscripts are interpreted modulo $k+1$. The cocircuit $\{b_{i-2},a_{i-1},b_{i-1}\}$ implies that $\{1,2,3\}$ is 
$\{b_{i-2},b_{i-1},a_{i}\}$ or $\{b_{i-2},b_{i-1},b_{i}\}$. Orthogonality with the cocircuit 
$\{b_{i-3},a_{i-2},b_{i-2}\}$ implies that $b_{i-3} = b_i$, that is, $k = 2$. In that case, since 
$\{b_{i-2},b_{i-1},b_{i}\}$ is a triangle, we see that $\lambda(T_{i-2} \cup T_{i-1} \cup T_i) \le 2$. This contradicts the fact that $M$ is \ifc. We conclude that $\{2,3\} \cap \{a_i,b_i\} = \emptyset$. It follows that $4 \in \{a_i,b_i\}$. Hence $C^*_i = \{2,3,4,c_i\}$. Thus, by orthogonality, $\{1,2,3\}$ avoids $T_0 \cup T_1 \cup \dots T_k$, so \ref{newnu2} holds. Hence so does the lemma.
\end{proof}

When we trim the bowtie ring in Figure~\ref{btringfig}, we delete all of the dashed oblique edges. The next lemma shows that we obtain an isomorphic matroid by deleting, instead, all of the solid oblique edges. 

\begin{lemma}
\label{ringletsarecute}
Let $(T_0,D_0,T_1,D_1,\ldots,T_n,D_n)$ be a ring of bowties in an \ifc\ binary matroid $M$.
Then 
\begin{itemize}
\item[(i)] $\{b_0,b_1,\ldots,b_n\}$ is either a circuit or an independent set of $M$; and 
\item[(ii)] $M\ba c_0,c_1,\ldots,c_n \cong M\ba a_0,a_1,\ldots,a_n.$
\end{itemize}
\end{lemma}

\begin{proof}
If $\{b_0,b_1,\ldots,b_n\}$ is dependent, then it contains a circuit, $C$. By \ort\ between $C$ and the cocircuit $D_i$, we see that  if $b_i \in C$ for some $i$, then $b_{i+1} \in C$. Hence $C = \{b_0,b_1,\ldots,b_n\}$, and (i) holds. 

To prove (ii), let $E = E(M)$ and, for each $j$ in $\{0,1,\ldots,n\}$,  define the function 
$\varphi_j: E-\{c_0,c_1,\ldots,c_n,b_j\} \to E-\{a_0,a_1,\ldots,a_n,b_{j+1}\}$ by 
\begin{equation*}
\varphi_j(x) = 
\begin{cases} 
c_i & \text{if $x= a_i$ and $i \in \{0,1,\ldots,n\}$};\\
b_{k+1} & \text{if $x= b_k$ and $k \neq j$; and}\\
x & \text{otherwise.}
\end{cases} 
\end{equation*}
We show next that 

\begin{sublemma}
\label{phij}
$\varphi_j$ is an isomorphism between 
$M\ba c_0,c_1,\ldots,c_n/b_j$ and $M\ba a_0,a_1,\ldots,a_n/b_{j+1}$.
\end{sublemma}

By the symmetry of the ring of bowties, it suffices to prove this when $j  = n$. Here we will exploit the isomorphisms noted in Lemma~\ref{stringswitch}. For each $t$ in $\{1,2,\ldots,n\}$, the isomorphism between 
$M\ba c_0,c_1,\dots,c_{t-1},c_t/b_t\ba a_{t+1},\dots,a_n$ and 
$M\ba c_0,c_1,\dots,c_{t-1},a_t/b_{t-1}\ba a_{t+1},\dots,a_n$ can be achieved by the mapping $\psi_t$ that takes $a_t$ to $c_t$  and takes $b_{t-1}$ to $b_t$ while fixing every other element.  In addition, let $\omega_0$ be the isomorphism between  
$M\ba c_0/b_0\ba a_{1},\dots,a_n$ and 
$M\ba a_0/b_0\ba a_{1},\dots,a_n$ obtained by mapping $a_0$ to $c_0$ and fixing every other element. 
The composition $\omega_0 \circ \psi_1 \circ \psi_2 \circ \dots \circ \psi_n$ equals $\varphi_n$. Hence \ref{phij} holds. 

Now define
$\varphi: E-\{c_0,c_1,\ldots,c_n\} \to E-\{a_0,a_1,\ldots,a_n\}$ by 
\begin{equation*}
\varphi(x) = 
\begin{cases} 
c_i & \text{if $x= a_i$ and $i \in \{0,1,\ldots,n\}$};\\
b_{i+1} & \text{if $x= b_i$ and $i \in \{0,1,\ldots,n\}$; and}\\
x & \text{otherwise.}
\end{cases} 
\end{equation*}
Observe that, for all $j$ in $\{0,1,\dots,n\}$, 

\begin{sublemma}
\label{phiphi}
$\varphi(y) = \varphi_j(y)$ for all $y$ in $E-\{c_0,c_1,\ldots,c_n\} - b_j$. 
\end{sublemma}

To complete the proof of (ii), we shall show that 

\begin{sublemma}
\label{phiphi3}
$\varphi$ is an isomorphism between 
$M\ba c_0,c_1,\dots,c_n$ and $M\ba a_0,a_1,\dots,a_n$.
\end{sublemma}

By \ref{phij} and \ref{phiphi}, for all $j$ in $\{0,1,\dots,n\}$, the function $\varphi$ induces an isomorphism between 
$M\ba c_0,c_1,\dots,c_n/b_j$ and $M\ba a_0,a_1,\dots,a_n/b_{j+1}$. To establish 
\ref{phiphi3}, we suffices to prove the following.

\begin{itemize}
\item[(a)] If $X$ is a cocircuit of $M\ba c_0,c_1,\dots,c_n$, then $\varphi(X)$ contains a cocircuit of $M\ba a_0,a_1,\dots,a_n$.
\item[(b)] If $Y$ is a cocircuit of $M\ba a_0,a_1,\dots,a_n$, then $\varphi^{-1}(Y)$ contains a cocircuit of $M\ba c_0,c_1,\dots,c_n$.
\end{itemize}

We will show (a); a symmetric argument establishes (b). Suppose the cocircuit $X$ of 
$M\ba c_0,c_1,\dots,c_n$ avoids $b_s$ for some $s$. Then
$X$ is a cocircuit of $M\ba c_0,c_1,\dots,c_n/b_s$. Thus $\varphi(X)$ is a cocircuit of $M\ba a_0,a_1,\dots,a_n/b_{s+1}$ and so is a cocircuit of $M\ba a_0,a_1,\dots,a_n$. Then (a) holds unless $X$  contains $\{b_0,b_1,\ldots,b_n\}$. Consider the exceptional case. 
Since $M\ba c_0,c_1,\dots,c_n$ has $\{b_0,a_1,b_1\}$ as a disjoint union of cocircuits and $\{b_0,b_1\} \subseteq X$, it follows that $a_1 \not\in X$. 
Thus $X \btu \{b_0,a_1,b_1\}$ contains a cocircuit $D$ of $M\ba c_0,c_1,\dots,c_n$ that contains $a_1$ but avoids $\{b_0,b_1\}$. Hence $D$ is a cocircuit of $M\ba c_0,c_1,\dots,c_n/b_0$ that avoids $b_1$. 
Thus $\varphi(D)$ is  a cocircuit of $M\ba a_0,a_1,\dots,a_n/b_1$ that contains $c_1$ and avoids $b_2$. Hence $\varphi(D)$ is  a cocircuit of $M\ba a_0,a_1,\dots,a_n$ that contains $c_1$ and avoids $\{b_1,b_2\}$. But $\{b_1,c_1,b_2\}$ is a disjoint union of cocircuits of $M\ba a_0,a_1,\dots,a_n$, so $\varphi(D) \btu \{b_1,c_1,b_2\}$ contains a cocircuit of 
$M\ba a_0,a_1,\dots,a_n$ that avoids $c_1$ and so is contained in $\varphi(X)$. We conclude that (a) holds. Hence \ref{phiphi3} holds and the lemma is proved.
\end{proof}


We conclude this section with another property of strings of bowties that we will use often.

\begin{lemma}
\label{stringybark}
Let $M$ be a  binary matroid and  $N$ be an \ifc\ binary matroid having at least seven elements. 
Let $T_0,D_0,T_1,D_1,\ldots,T_n$ be a string of bowties in $M$.
Suppose $M\ba c_0,c_1$ has an $N$-minor but $M\ba c_0,c_1/b_1$ does not. Then $M\ba c_0,c_1,\dots ,c_n$ has an $N$-minor, but 
$M\ba c_0,c_1,\dots ,c_i/b_i$ has no $N$-minor for all $i$ in $\{1,2,\dots, n\}$, and    $M\ba c_0,c_1,\dots ,c_j/a_j$ has no $N$-minor for all $j$ in $\{2,3,\dots,n\}$. 
\end{lemma}

\begin{proof} We may assume that $n \ge 2$ otherwise there is nothing to prove. For  $i$ in $\{1,2,\ldots,n\}$, it follows by Lemma~\ref{stringswitch} that 
$M\ba c_0,c_1,\ldots,c_i/b_i \cong M\ba c_0, c_1/b_1\ba a_2,\ldots,a_i$. As $M\ba c_0,c_1/b_1$ has no $N$-minor, we deduce that $M\ba c_0,c_1,\ldots,c_i/b_i$ has no $N$-minor. 
If $M\ba c_0,c_1,\dots ,c_j/a_j$ has an $N$-minor for some $j\in\{2,3,\dots ,n\}$, then so do $M\ba c_0,c_1,\dots ,c_{j-1},b_j/a_j$ and $M\ba c_0,c_1,\dots ,c_{j-1},b_j/b_{j-1}$; a \cn. 
Thus the second part of the lemma holds. For the first part, 
suppose that  $M\ba c_0,c_1,\dots ,c_k$ has no $N$-minor for some $k$ in $\{2,3,\dots, n\}$ but that $M\ba c_0,c_1,\dots ,c_{k-1}$ does have an $N$-minor.  
As $M\ba c_0,c_1,\dots ,c_{k-1}$ has $(c_k,b_k,a_k,b_{k-1})$ as a $4$-fan, we know by Lemma~\ref{2.2}  that $M\ba c_0,c_1,\dots ,c_{k-1}/b_{k-1}$ has an $N$-minor; a \cn. We conclude that $M\ba c_0,c_1,\dots ,c_n$ has an $N$-minor. 
\end{proof}

\section{Results for ladder segments}
\label{ladderseg}

A string of bowties may be part of a quartic ladder segment within a binary matroid.  
In this section, we consider the ramifications of such an occurrence.  

\begin{figure}[tb]
\center
\includegraphics[scale=1.]{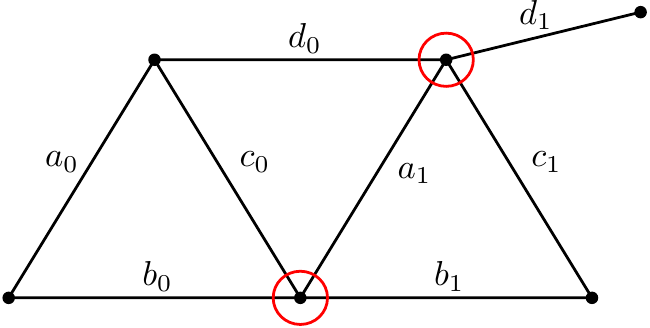}
\caption{}
\label{onehorned}
\end{figure}

\begin{lemma}
\label{minidrossrsv}
Assume that $M$ is an \ifc\ binary matroid that  contains the configuration shown in Figure~\ref{onehorned} and has at least thirteen elements. 
Then all of the elements in the figure are distinct, $M\ba c_0,c_1$ is \sfc, and $a_1$ is in no triangle of $M\ba c_0,c_1$.  Moreover, one of the following holds. 
\begin{itemize}
\item[(i)] $M\ba c_0,c_1$ is \ifc; 
\item[(ii)] $\{d_0,d_1\}$ is in a triangle of $M$; 
\item[(iii)] $\{b_0,b_1\}$ is in a triangle of $M$; 
\item[(iv)] $M$ has a triangle $\{\alpha, \beta, a_0\}$ and has $\{\beta,a_0,d_0,c_0\}$ or $\{\beta,a_0,a_1,c_0,c_1\}$ as a cocircuit, where $\alpha$ and $\beta$ are elements not shown in the figure, and $M\ba c_0,c_1$ is \ffsc\ with every $4$-fan of it having 
$\alpha$ as its guts or $b_1$ as its coguts; or
\item[(v)] $M\ba c_0,c_1$ is \ffsc\ and $b_1$ is the coguts element of all its $4$-fans.
\end{itemize}
\end{lemma}

\begin{proof} 
Recall that $T_i = \{a_i,b_i,c_i\}$ for each $i$ in $\{0,1\}$. First we show that 
\begin{sublemma}
\label{eightels} $|T_0 \cup T_1 \cup \{d_0,d_1\}| = 8$.
\end{sublemma}
Evidently $|T_0 \cup T_1| = 6$ otherwise $a_0 = c_1$ and $\lambda(T_0 \cup T_1) \le 2$; a \cn. If $\{d_0,d_1\}$ meets $T_1$, then $M$ has a $4$-fan or has a triangle contained in a cocircuit. Each possibility gives a \cn. Finally, if 
$\{d_0,d_1\}$ meets $T_0$, then, by orthogonality, $\{d_0,d_1\} \subseteq T_0$, so $\lambda(T_0 \cup T_1) \le 2$; a \cn.  Thus \ref{eightels} holds.

To see that $a_1$ is in no triangle of $M\ba c_0,c_1$, we observe that such a triangle would have to be 
 $\{b_0,a_1,d_1\}$. Then $\lambda(T_0\cup T_1\cup\{d_0,d_1\}) \le 2$; a \cn. 
 
 Next we show that 
 
 \begin{sublemma}
\label{issfc} $M\ba c_0,c_1$ is \sfc.
\end{sublemma}

Let $(U,V)$ be a \ns\ \ths\ of $M\ba c_0,c_1$. Then we may assume that $\{a_1,d_0,d_1\} \subseteq U$ and $U$ is fully closed in $M\ba c_0,c_1$. If $a_0,b_0$, or $b_1$ is in $U$, then $\{a_0,b_0,b_1\}\subseteq U$ and $(U \cup \{c_0,c_1\},V)$ is a \ns\ \ths\ of $M$; a \cn. Thus $\{a_0,b_0,b_1\} \subseteq V$. Then $a_1\in\cl _{M\ba c_0,c_1}^*(V)$, and $(U-a_1,V\cup a_1 \cup c_0\cup c_1)$ is a \ns\ \ths\ of $M$; a \cn. Thus \ref{issfc} holds. 

Now assume that (i) does not hold. 
Let $(\al,\be,\ga,\de)$ be a $4$-fan in $M\ba c_0,c_1$. Then $M$ has a cocircuit $C^*$ such that 
$\{\be,\ga, \de\} \subsetneqq C^* \subseteq \{\be,\ga,\de,c_0,c_1\}$. 

Suppose that $c_0 \in C^*$. Then $C^*$ meets each of $\{a_1,d_0\}$ and $\{a_0,b_0\}$ in exactly one element. If $d_0 \in \{\be,\ga\}$, then, by \ort, $d_1 \in \{\al,\be,\ga\}$ so (ii) holds. Hence we may assume that $\de \in \{a_1,d_0\}$ since $a_1$ is in no triangle of $M\ba c_0,c_1$. Moreover, without loss of generality, $\ga\in\{a_0,b_0\}$. 

Suppose $b_0 =\ga$. Then \ort\ implies that $\{b_0,b_1\}$ or $\{b_0,a_1\}$ is contained in a triangle. In the former case, (iii) holds; the latter does not arise since $M\ba c_0,c_1$ has no triangle containing $a_1$.

We may now assume that $a_0 = \ga$. It follows that $M$ has a triangle containing $a_0$, and $C^*$ is one of $\{\be,a_0,a_1,c_0\}, \{\be,a_0,a_1,c_0,c_1\}, \{\be,a_0,d_0,c_0\}$, or 
$\{\be,a_0,d_0,c_0,c_1\}$. If $C^*$ is $\{\be,a_0,a_1,c_0\}$ or $\{\be,a_0,d_0,c_0,c_1\}$, then, by \ort, $\beta =b_1$, so $\lambda(T_0\cup T_1\cup d_0) \le 2$; a \cn. Thus $C^*$ is $\{\be,a_0,a_1,c_0,c_1\}$ or $\{\be,a_0,d_0,c_0\}$. 
Therefore (iv) holds provided $\al$ and $\be$ are new elements and every \ftv\ of $M\ba c_0,c_1$ is a $4$-fan with $\alpha$ as its guts or $b_1$ as its coguts. 

If $\be$ is an existing element, then $\lambda(T_0\cup T_1\cup\{d_0,d_1\}) \le 2$; a \cn. Suppose $\al$ is an existing element. Then $\beta \in\cl (T_0\cup T_1\cup \{d_0,d_1\})$, so $\lambda (T_0\cup T_1\cup\{d_0,d_1,\beta\})\leq 2$; a \cn.  Thus $\al$ and $\be$ are new elements.  
Let $(u,v,w,x)$ be a $4$-fan in $M\ba c_0,c_1$ where $u\neq \al$.  
Since $a_1$ is in no triangle of $M\ba c_0,c_1$, we know that if $\{u,v,w\}$ meets $\{d_0,d_1\}$ or $\{b_0,b_1\}$, then \ort\ implies that (ii) or (iii) holds, respectively.  
Hence we may assume that every triangle in $M\ba c_0,c_1$ avoids $\{a_1,d_0,d_1,b_0,b_1\}$.  
We know that $M$ has a cocircuit $D^*$ such that $\{v,w,x\}\subsetneqq D^*\subseteq\{v,w,x,c_0,c_1\}$.
  
Suppose $c_0\in D^*$. Then \ort\ implies that $\{v,w,x\}$ meets $\{a_0,b_0\}$ and $\{a_1,d_0\}$, so $x\in\{a_1,d_0\}$ and, without loss of generality, $w=a_0$.  
Then \ort\ between the triangle $\{\alpha, \beta, a_0\}$ and the cocircuit $D^*$ implies that $v\in\{\al,\be\}$. Thus $\{u,v,w\} = \{\alpha,\beta,a_0\}$  so $v=\al$ and $u=\be$. Hence the cocircuits $C^*$ and $D^*$ imply that  $\lambda (T_0\cup T_1\cup\{\al,\be,d_0\})\leq 2$; a \cn.  
We deduce that $c_0\notin D^*$, so $D^*=\{v,w,x,c_1\}$. Thus Lemma~\ref{bowwow} implies that $\{v,w,x\}$ avoids $T_0$.  
Orthogonality implies that $\{v,w,x\}$ meets $\{a_1,b_1\}$, so $x\in\{a_1,b_1\}$.  
If $x=a_1$, then \ort\ implies that $\{v,w\}$ meets $\{a_0,b_0,d_0\}$, a \cn.  
We conclude that every $4$-fan of $M\ba c_0,c_1$ has $\al$ as its guts or $b_1$ as its coguts, so (iv) holds provided $M\ba c_0,c_1$ has no $5$-fans and no $5$-cofans.  
If $M\ba c_0,c_1$ has $(1,2,3,4,5)$ as a $5$-fan, then, up to reversing the order of the fan, $1=\al$ and $2=b_1$, so $b_1$ is in a triangle of $M\ba c_0,c_1$; a \cn.  
Suppose then that $(0,1,2,3,4)$ is a $5$-cofan of $M\ba c_0,c_1$.  
Then, up to reversing the order, $0=b_1$ and $1=\al$, so \ort\ implies that $2\in\{\be ,a_0\}$.  Thus $M$ has a cocircuit containing $\{b_1,\al\}$ and contained in  
$\{b_1,\al,\be,a_0,c_0,c_1\}$. Using this cocircuit together with the cocircuits $C^*$ and $\{b_0,c_0,a_1,b_1\}$, we see that  $\lambda (T_0\cup T_1\cup\{\al,\be,d_0\})\leq 2$; a \cn.  
We deduce that (iv) holds.  

We may now assume that $c_0 \not \in C^*$. Then $c_1 \in C^*$, so $C^* = \{\be,\ga,\de,c_1\}$. By \ort, $a_1$ or $b_1$ is in $C^*$. Suppose $a_1 \in C^*$. Then, as $a_1$ is not in a triangle of $M\ba c_0,c_1$, we have $a_1 = \de$, so $C^* = \{\be,\ga,a_1,c_1\}$. Orthogonality with $\{c_0,d_0,a_1\}$ implies that $d_0$ is in $\{\be,\ga\}$, so $C^* = \{d_0,d_1,a_1,c_1\}$. Hence $\{d_0,d_1\}$ is contained in a triangle and (ii) holds. We may now assume  that $a_1 \not \in C^*$. Thus $b_1 \in C^*$.   If $b_1 \in \{\be, \ga\}$, 
then it follows that $\{b_0,b_1\}$ is in a triangle so (iii) holds. Hence we may assume that $b_1 = \de$. Then $C^* = \{\be,\ga,b_1,c_1\}$. Thus $b_1$ is the coguts element of the $4$-fan 
$(\al,\be,\ga,\de)$,  so $M\ba c_0,c_1$ is \ffsc\ and (v) holds. 
\end{proof}

\begin{figure}[tb]
\center
\includegraphics[scale=1.]{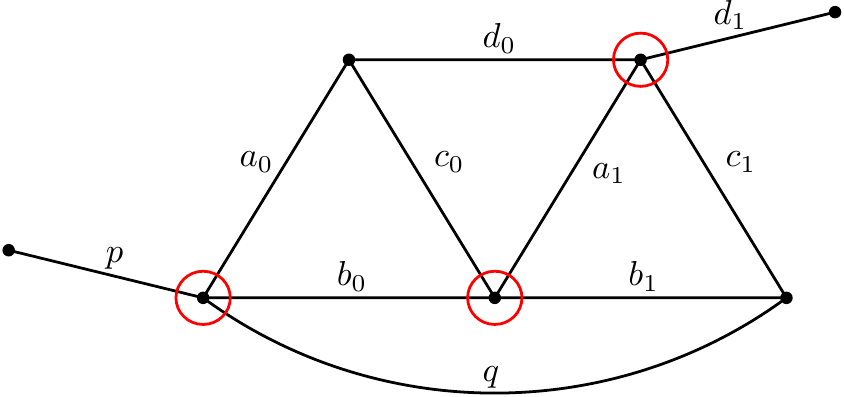}
\caption{}
\label{onehornedplus}
\end{figure}

\begin{lemma}
\label{minidross2} 
Let $M$ be an \ifc\ binary matroid. 
Assume that $M$ contains the configuration shown in Figure~\ref{onehornedplus} and that $|E(M)| \ge 13$.  
Then the elements in the figure are distinct except that $p$ may be $d_1$. Moreover,  one of the following holds.
\begin{itemize}
\item[(i)] $M\ba c_0,c_1,q$ is \ifc; or
\item[(ii)] $\{d_0,d_1\}$ is in a triangle of $M$; or
\item[(iii)] $M$ has a triangle $\{s_1,s_2,s_3\}$ and a cocircuit $\{q,c_1,b_1,s_2,s_3\}$ 
where $\{s_1,s_2,s_3\}$ avoids $\{b_0,c_0,q,a_1,b_1,c_1\}$; or 
\item[(iv)]  $M$ has a triangle containing $\{a_0,p\}$ and some element that is not shown in the configuration; or
\item[(v)] $M$ has a $4$-cocircuit that contains $\{q,b_1,c_1\}$.  
\end{itemize}
\end{lemma}

\begin{proof} First note that Lemma~\ref{minidrossrsv} implies that the elements in Figure~\ref{onehorned} are distinct.  
Orthogonality between $\{b_0,b_1,q\}$ and $\{d_0,a_1,c_1,d_1\}$ implies that $q\notin \{d_0,a_1,c_1,d_1\}$. Moreover,  $q \neq c_0$.  
If $q=a_0$, then $c_0=b_1$; a \cn.  
Thus $q\notin T_0\cup T_1\cup \{d_0,d_1\}$, and so \ort\ implies that $p\notin T_0\cup T_1\cup \{d_0,q\}$.  

Next we show that 

\begin{sublemma}
\label{minidrosssub}
$M\ba c_0,c_1,q$ is \thc\ or (v) holds.  
\end{sublemma}

By Lemma~\ref{deletecs}, $M\ba c_0,c_1$ is \thc. Since $\si(M\ba c_0,c_1/q)$ has a $2$-element cocircuit, it is not \thc. Thus, by Bixby's Lemma \cite{bixby}, $\co(M\ba c_0,c_1,q)$ is \thc. Suppose that $M\ba c_0,c_1,q$ is not \thc. Then $M\ba c_0,c_1$ has a triad $C^*$ containing $q$. By \ort, $C^*$ meets $\{b_0,b_1\}$ and $\{b_1,a_1,d_0,a_0\}$. Thus $b_1 \in C^*$ otherwise $\lambda(T_0 \cup T_1 \cup \{q,d_0\}) \le 2$; a \cn.

Now $M$ has a cocircuit $D^*$ such that $C^*\subsetneqq D^* \subseteq C^* \cup \{c_0,c_1\}$. Assume $c_0 \in D^*$. Then, by \ort, $C^*$ is $\{q,b_1,a_0\}$ or $\{q,b_1,b_0\}$. Thus $\lambda(T_0 \cup T_1 \cup \{q,d_0\}) \le 2$; a \cn. We conclude that $c_0 \not \in D^*$, so $D^* = C^* \cup c_1$ and (v) holds. Thus \ref{minidrosssub} holds. 

Now we show that 

\begin{sublemma}
\label{getiv}
if $\{a_0,p\}$ is contained in a triangle of $M$, then (iv) holds.  
\end{sublemma}

Assume that $\{a_0,p\}$ is contained in a triangle whose third element is already in the configuration. By \ort, this third element is not $d_1$, so $p \in \cl(T_0 \cup T_1 \cup \{d_0,q\})$. Thus 
$\lambda(T_0 \cup T_1 \cup \{d_0,q,p\}) \le 2$. This is a \cn\ as $|E(M)| \ge 13$. Hence \ref{getiv} holds.

Next, suppose that $(U,V)$ is a \ns\ \ths\ of $M\ba c_0,c_1,q$. Then we may assume that $\{b_0,b_1,a_1\}\subseteq U$. Thus $(U \cup q,V)$ is a \ns\ \ths\ of $M\ba c_0,c_1$; a \cn\ to Lemma~\ref{minidrossrsv}.
We deduce that $M\ba c_0,c_1,q$ is \sfc. 

Now let $(s_1,s_2,s_3,s_4)$ be a $4$-fan of $M \ba c_0,c_1,q$. Then $M$ has a cocircuit $C^*$ such that 
$\{s_2,s_3,s_4\} \subsetneqq C^* \subseteq \{s_2,s_3,s_4,c_0,c_1,q\}$, and $M$ has $\{s_1,s_2,s_3\}$ as a triangle. 

By Lemma~\ref{minidrossrsv}, $a_1$ is not in a triangle of $M\ba c_0,c_1,q$. It follows that 
\begin{sublemma}
\label{triavoids}
none of $a_1$, $b_0$, or $b_1$ is in $\{s_1,s_2,s_3\}$.  
\end{sublemma}

Suppose that $a_1 \in C^*$. Then $a_1 = s_4$. Thus $q \not \in C^*$ otherwise, by \ort, $\{s_2,s_3\}$ meets $\{b_0,b_1\}$. Moreover, by \ort\ again, $c_1 \in C^*$ and exactly one of $d_0$ and $c_0$ is in $C^*$. If $d_0 \in C^*$, then $\{d_0,d_1\} \subseteq \{s_1,s_2,s_3\}$, so (ii) holds. Hence we may assume that $c_0 \in C^*$. Thus $C^* = \{s_2,s_3,a_1,c_0,c_1\}$. As $b_0 \not \in \{s_1,s_2,s_3\}$, we may assume, by \ort\ and symmetry, that 
$a_0 = s_3$. By \ort, $p \in \{s_1,s_2\}$, so $\{a_0,p\}$ is contained in a triangle.  Thus, by \ref{getiv}, (iv) holds.

We may now assume that $a_1 \not \in C^*$. Then $C^*$ contains $\{b_1,c_1\}$ or avoids $\{b_1,c_1\}$. Consider  the first case. Then~\ref{triavoids} implies that $b_1 = s_4$ and $b_0 \not \in C^*$.  Then $q \in C^*$. Moreover, $\{a_0,c_0\} \subseteq C^*$ or $\{a_0,c_0\}$ avoids $C^*$. Suppose $\{a_0,c_0\} \subseteq C^*$. Then $C^* = \{s_2,s_3,b_1,c_1,c_0,q\}$, where $a_0 \in \{s_2,s_3\}$. Orthogonality implies that $d_0\in\{s_2,s_3\}$, so $\lambda(T_0\cup T_1\cup\{d_0,q\}) \le 2$; a \cn. Now suppose $C^*$ avoids $\{a_0,c_0\}$. Then $C^* = \{s_2,s_3,b_1,q,c_1\}$, so (iii) holds.

Finally, assume that $C^*$ avoids $\{b_1,c_1\}$. Suppose first that $q \not \in C^*$. Then, as  $C^*$ meets $\{c_0,c_1,q\}$, we deduce that  $C^* = \{s_2,s_3,s_4,c_0\}$. 
By~\ref{triavoids} and \ort\ with the circuit $\{b_0,b_1,q\}$, we deduce that  $b_0 \not \in C^*$. Hence, by \ort\ with the triangle $\{a_0,b_0,c_0\}$, we see that $a_0 \in C^*$. Hence $d_0 \in C^*$. Then $d_0 = s_4$ otherwise $\{d_0,d_1\}$ is contained in a triangle and (ii) holds. Thus $a_0 \in \{s_2,s_3\}$, so $p \in \{s_1,s_2,s_3\}$ and~\ref{getiv} implies that (iv) holds.  

It remains to consider when $C^*$ avoids $\{b_1,c_1\}$ but contains $q$.  
Then $b_0 \in C^*$ so~\ref{triavoids} implies that $b_0 = s_4$. If $c_0 \in C^*$, then $d_0 \in \{s_1,s_2,s_3\}$, so $\{d_0,d_1\} \subseteq \{s_1,s_2,s_3\}$, and the lemma holds. Thus we may assume that $c_0 \not \in C^*$. Then $C^* = \{s_2,s_3,b_0,q\}$, so $a_0 \in \{s_2,s_3\}$. As $M$ has  $\{a_0,p,b_0,q\}$ as a cocircuit, we must have that $C^* = \{a_0,p,b_0,q_0\}$. Thus $M$ has a triangle containing $\{a_0,p\}$ and~\ref{getiv} implies that (iv) holds. 
\end{proof}

\begin{figure}
\center
\includegraphics{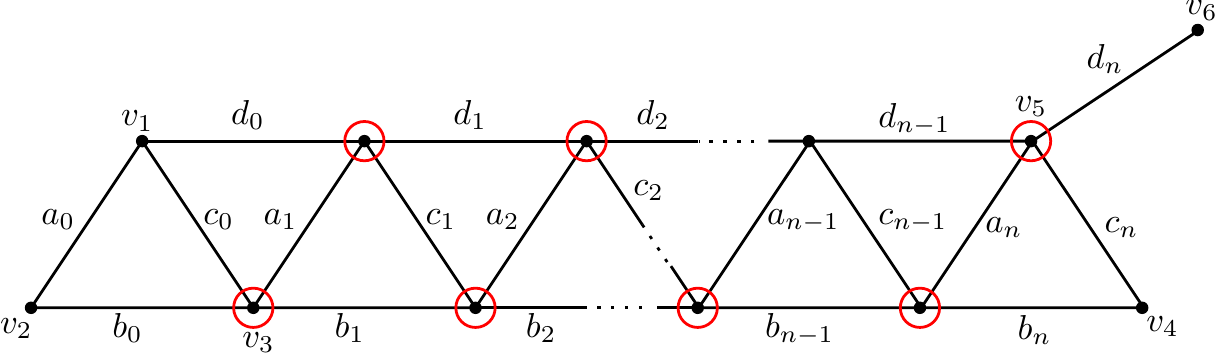}
\caption{ $\{d_{n-2},a_{n-1},c_{n-1},d_{n-1}\}$ or $\{d_{n-2},a_{n-1},c_{n-1},a_n,c_n\}$ is a cocircuit.}
\label{drossfigi}
\end{figure}

The next two lemmas concern the configuration shown in Figure~\ref{drossfigi}. 

\begin{lemma}
\label{trywhere}
Let $M$ be an \ifc\ binary matroid that has at least thirteen elements and contains the configuration shown in Figure~\ref{drossfigi} where $n\ge 2$,  all of the elements shown are distinct, and, in addition to the cocircuits shown, exactly one of 
$\{d_{n-2},a_{n-1},c_{n-1},d_{n-1}\}$ or $\{d_{n-2},a_{n-1},c_{n-1},a_n,c_n\}$ is a cocircuit. Then a triangle  $T$ of $M\ba c_0,c_1,\ldots,c_n$ that meets $\{a_0,b_0,d_0,a_1,b_1,d_1,\ldots,a_n,b_n,d_n\}$ does so in $\{a_0\}, \{d_{n-1},d_n\}$, or $\{a_0,d_{n-1},d_n\}$. 
\end{lemma}

\begin{proof} 
Assume first that $T$ meets $\{a_i,b_i\}$ for some $i$ with $1 \le i \le n-1$. Then, as $\{b_{i-1},a_i,b_i\}$ is a cocircuit of $M\ba c_0,c_1,\ldots,c_n$ and $T$ does not contain $\{a_i,b_i\}$, we see, by \ort, that $T$ contains $\{b_{i-1},b_i\}$ or $\{b_{i-1},a_i\}$.  

Suppose $\{d_{i-1},a_i,c_i,d_i\}$ is a cocircuit.  Orthogonality implies that $T$ is $\{b_{i-1},b_i,a_{i+1}\},\{b_{i-1},b_i,b_{i+1}\}$, or $\{b_{i-1},a_i,d_i\}$.  
If $T=\{b_{i-1},b_i,b_{i+1}\}$, then $\lambda (\{b_{i-1},c_{i-1},d_{i-1},a_i,b_i,c_i,d_i,a_{i+1},b_{i+1}\})\leq 2$, so $|E(M)|\leq 12$; a \cn.
If $T=\{b_{i-1},a_i,d_i\}$, then $\lambda (\{b_{i-1},c_{i-1},d_{i-1},a_i,b_i,c_i,d_i\})\leq 2$; a \cn.  Thus $T=\{b_{i-1},b_i,a_{i+1}\}$. But  $\{d_{i},a_{i+1},c_{i+1},d_{i+1}\}$ or $\{d_i,a_{i+1},c_{i+1},a_{i+2},c_{i+2}\}$ is a cocircuit, and so \ort\ with $T$ gives a \cn.  
We conclude that $\{d_{i-1},a_{i},c_i,d_i\}$ is not a cocircuit. Hence $i = n-1$ and 
$\{d_{n-2},a_{n-1},c_{n-1},a_{n},c_{n}\}$ is a cocircuit.  
As $T$ contains $\{b_{n-2}, b_{n-1}\}$ or $\{b_{n-2},a_{n-1}\}$, \ort\ now implies that $T$ is $\{b_{n-2},b_{n-1},b_n\}$.  Then $\lambda (\{b_{n-2},c_{n-2},d_{n-2}\}\cup T_{n-1}\cup T_n)\leq 2$, so $|E(M)|\leq 12$; a \cn.  We conclude that $T$ avoids $\{a_i,b_i\}$ for all $i$ with $1 \le i \le n-1$. It follows that $b_0 \not \in T$ otherwise $\{a_1,b_1\}$ meets $T$. Moreover, $T$ avoids $\{a_n,b_n\}$ otherwise $T$ contains $b_{n-1}$; a \cn. We conclude that 

\begin{sublemma} 
\label{trywhere2}
 $T \cap \{a_0,b_0,a_1,b_1,\ldots,a_n,b_n\} \subseteq \{a_0\}$.
\end{sublemma}
 
Next we note the following. 

\begin{sublemma} 
\label{trywhere3}
If $d_0 \in T$, then $d_1 \in T$.
\end{sublemma}

This follows by \ort\ and \ref{trywhere2} since $T$ must meet $\{a_1,c_1,d_1\}$ or $\{a_1,c_1,a_2,c_2\}$.

\begin{sublemma} 
\label{trywhere4}
If $d_i \in T$ for some $i$ with $1 \le i \le n$, then $i \in \{n-1,n\}$ and $\{d_{n-1},d_n\} \subseteq T$.
\end{sublemma}

To see this, observe that, by \ort, $T$ meets $\{d_{i-1},a_i\}$ if $i \neq n-1$, and $T$ meets $\{a_n,d_n\}$ if $i = n-1$. Thus $\{d_{i-1},d_i\} \subseteq T$ if $i \neq n-1$, and $\{d_{n-1},d_n\} \subseteq T$ if $i = n-1$. Assume $i \le n-2$. Then $\{d_{i-1},d_i\} \subseteq T$ and, by \ort, either $d_{i+1} \in T$ and  $\{d_i, a_{i+1},c_{i+1},d_{i+1}\}$ is a cocircuit, or $i = n-2$ and 
$\{d_{n-2}, a_{n-1},c_{n-1},a_n,c_n\}$ is a cocircuit. In the latter case, 
$\{a_{n-1},c_{n-1},a_n,c_n\}$ meets $T$; a \cn\ to \ref{trywhere2}. Hence $d_{i+1} \in T$, so 
$T = \{d_{i-1},d_i, d_{i+1}\}$. But orthogonality implies that either $T$ meets 
$\{a_{i+2},c_{i+2},d_{i+2}\}$; or $i+1 = n-2$ and $T$ meets $\{a_{n-1},c_{n-1},a_n,c_n\}$. Both possibilities yield a \cn\ since all the elements in the figure are distinct. We conclude that $i > n-2$, so \ref{trywhere4} holds. 

By combining \ref{trywhere3} and \ref{trywhere4}, we deduce that the lemma holds unless $n = 2$ and $T = \{d_0,d_1,d_2\}$. In the exceptional case, 
$\lambda(T_0 \cup T_1 \cup \{d_0,d_1,d_2\}) \le 2$; a \cn.
\end{proof}



\begin{figure}
\center
\includegraphics{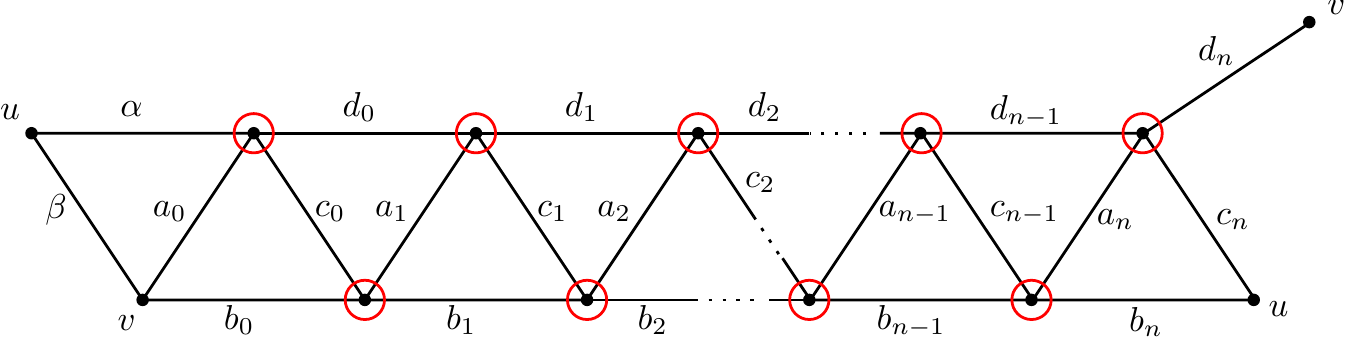}
\caption{A quartic M\"obius ladder.}
\label{drossfigii}
\end{figure}

\begin{lemma}
\label{drossdistinct}
Let $M$ be an \ifc\ binary matroid such that $|E(M)|\geq 13$.  
Suppose that 
 $M$ contains the structure in Figure~\ref{drossfigi}, where   $T_0,D_0,T_1,D_1,\dots ,T_n$ is a string of bowties and, when $n \ge 2$, either $\{d_{n-2},a_{n-1},c_{n-1},d_{n-1}\}$ or $\{d_{n-2},a_{n-1},c_{n-1},a_n,c_n\}$ is a cocircuit. Then 
\begin{itemize}
\item[(i)] all of the elements shown in Figure~\ref{drossfigi} are distinct; or
\item[(ii)] $(a_0,b_0) = (c_n,d_n)$ but all the other elements in the figure are distinct,  $M$ has both $\{d_{n-2},a_{n-1},c_{n-1},d_{n-1}\}$ and $\{b_n,a_0,c_0,d_0\}$ as cocircuits and either 
	\begin{itemize} 
	\item[(a)] all of the elements of  $M$ are shown in Figure~\ref{drossfigi},  and $M$ is the cycle matroid of the quartic M\"obius ladder that is obtained from the structure in Figure~\ref{drossfigi} by identifying the vertices $v_1,v_2$, and $v_3$ with the vertices  $v_4,v_5$, and $v_6$, respectively;  or 
	\item[(b)] $M$ has exactly one element, $\gamma$, that is not shown in Figure~\ref{drossfigi}, and $M$ is the matroid for which $M\ba \gamma$ is  a wheel whose spokes, in cyclic order, are $c_0,a_1,c_1,a_2,c_2,\dots,a_n,c_n$ such that the fundamental circuit of $\gamma$  with respect to the basis, $B$, consisting of this set of spokes is $B \cup \gamma$; and the rim of $M\ba \gamma$, in cyclic order, is 
	$d_0,b_1,d_1,b_2,d_2,\dots,b_n,d_n$; 
	\end{itemize} or 
 \item[(iii)] $M$ has $\{d_{n-2},a_{n-1},c_{n-1},a_n,c_n\}$  as a cocircuit, $(a_0,b_0) = (d_{n-1},d_n)$ but all the other elements in Figure~\ref{drossfigi} are distinct, and $M$ has at most one element that is absent from that figure. 
\end{itemize}
\end{lemma}

\begin{proof}
First we note that, by Lemma~\ref{minidrossrsv}, if $n = 1$, then the elements in Figure~\ref{drossfigi} are distinct. Thus 

\begin{sublemma}
\label{bign}
$n \ge 2$ or (i) holds.
\end{sublemma}

As $T_0,D_0,T_1,D_1,\dots ,T_n$ is a bowtie string, we know that the elements in $T_0\cup T_1\cup\dots\cup T_n$ are all distinct except that $a_0$ may be $c_n$.  
We will begin by treating the case when  $a_0 \neq c_n$, first showing the 
following. 
\begin{sublemma}
\label{newcocct7}
When $a_0 \neq c_n$, if some $d_i$ is in $T_0\cup T_1\cup \dots \cup T_n$, then $\{d_i,a_{i+1},c_{i+1},a_{i+2},c_{i+2}\}$ is not a cocircuit.
\end{sublemma} 

Suppose $\{d_i,a_{i+1},c_{i+1},a_{i+2},c_{i+2}\}$ is a cocircuit. Then, since $M$ is binary, \ort\ implies that $d_i\in\{a_{i+1},c_{i+1},a_{i+2},c_{i+2}\}$; a \cn\ to Lemma~\ref{bowwow}.  
Thus \ref{newcocct7} holds. 

We now show:
\begin{sublemma}
\label{newtoo}
When $a_0 \neq c_n$, either $\{d_0,d_1,\dots ,d_n\}$ avoids $T_0\cup T_1\cup \dots \cup T_n$; or $M$ has at most one element that is not shown in Figure~\ref{drossfigi}, $M$ has $\{d_{n-2},a_{n-1},c_{n-1},a_n,c_n\}$  as a cocircuit, $(a_0,b_0) = (d_{n-1},d_n)$, and $\{d_0,d_1,\ldots,d_{n-2}\}$ avoids $T_0\cup T_1\cup \dots \cup T_n$. 
\end{sublemma}
 
This is certainly true if $n = 1$. Thus we may assume that $n \ge 2$. Suppose that some $d_i$ is in   $T_0\cup T_1\cup \dots \cup T_n$,  choosing the least such $i$. Then $d_i \in T_j$, say.
 If $\{d_{i-1},a_i,c_i,d_i\}$ is a cocircuit, then \ort\ implies that $T_j$ meets $\{d_{i-1},a_i,c_i\}$. We deduce that   
 either $i=0$; or $i=n-1$ and $\{d_{n-2},a_{n-1},c_{n-1},d_{n-1}\}$ is not a cocircuit.  
It follows using \ref{newcocct7} that, in both cases, $\{d_i,a_{i+1},c_{i+1},d_{i+1}\}$ is a cocircuit. Now \ort\ implies that either $d_{i+1}\in T_j$, or $j = i+1$.  The latter implies that  $T_j$ is contained in a cocircuit. Hence the former holds. 
Then \ref{newcocct7} implies that $\{d_{i+1},a_{i+2},c_{i+2},a_{i+3},c_{i+3}\}$ is not a 
cocircuit. Thus if $i=0$, then $\{d_{i+1},a_{i+2},c_{i+2},d_{i+2}\}$ is a cocircuit and \ort\ implies that $T_j=\{d_0,d_{1},d_{2}\}$. Hence $\lambda (T_1\cup T_2\cup \{d_0,d_1,d_2\})\leq 2$; a \cn.  We deduce that $i \neq 0$, so 
 $i=n-1$ and $\{d_{n-2},a_{n-1},c_{n-1},d_{n-1}\}$ is not a cocircuit. 
 Thus $\{d_{n-2},a_{n-1},c_{n-1},a_n,c_n\}$ is a cocircuit. 
 Moreover, since $\{d_i,d_{i+1}\} \subseteq T_j$, we see that 
$\{d_{n-1},d_n\}\subseteq T_j$.  
 Thus $\{d_0,d_1,\dots ,d_n\}$ avoids $T_0 \cup T_1 \cup \cdots \cup T_{j-1}$. 
 Now $j \le n-2$ otherwise $\lambda(T_{n-1} \cup T_n \cup \{d_{n-1},d_n\}) \le 2$; a \cn. 
 
 Let $Z = T_j \cup T_{j+1} \cup \cdots \cup T_n \cup \{d_j,d_{j+1},\ldots,d_n\}$. Since $\{d_{n-1},d_n\} \subseteq T_j$, we deduce that $Z$ is spanned by 
 $\{d_n\} \cup \{a_k,b_k: j+1 \le k \le n\}$. Thus $r(Z) \le 2(n-j) + 1$. Since $Z$ contains at least $2(n-j)$ cocircuits of $M$ none a symmetric difference of the others, we deduce that $\lambda(Z) \le 1$. Thus $Z$ contains at least $|E(M)| - 1$ elements, so $j = 0$. 
Furthermore, we know that $T_0$ is the only triangle in the bowtie string to meet $\{d_0,d_1,\dots ,d_n\}$ and, by the minimality of $i$, it follows that $\{d_0,d_1,\dots,d_{n-2}\}$ avoids $T_0\cup T_1\cup \dots \cup T_n$.  
By \ort\ between the triangle $\{c_{n-1},d_{n-1},a_n\}$ and the cocircuit $D_0$, we deduce that $d_{n-1}=a_0$.  
Since $T_0$ also contains $d_n$, we see that $d_n \in \{b_0,c_0\}$. 
If $d_n=c_0$, then \ort\ between the triangle $\{c_0,d_0,a_1\}$ and the cocircuit $\{d_{n-1},a_n,c_n,d_n\}$ implies that $d_0\in\{a_n,c_n\}$; a \cn.  
Thus  $d_n=b_0$ and  \ref{newtoo} holds. 

Next, we complete the proof of the lemma when $a_0 \neq c_n$ by proving the following.

\begin{sublemma}
\label{newtutu}
When $a_0 \neq c_n$, if $\{d_0,d_1,\dots ,d_n\}$ avoids $T_0\cup T_1\cup \dots \cup T_n$, then 
 the elements in $\{d_0,d_{1},\dots ,d_n\}$ are distinct.  
\end{sublemma}

Assume that this is not so, and choose $j$ to be the maximum member of $\{0,1,\ldots,n\}$ such that $d_j \in \{d_0,d_{1},\dots ,d_n\} - \{d_j\}$. 
Then $d_j = d_i$ for some $i$ with $0\leq i <j\leq n$.  
By \ort\ with the triangle $\{c_i,d_i,a_{i+1}\}$ and the maximality of $j$, we see that neither  $\{d_j,a_{j+1},c_{j+1},a_{j+2},c_{j+2}\}$ nor $\{d_j,a_{j+1},c_{j+1},d_{j+1}\}$ is a cocircuit.  
Hence $j=n$, and  $\{d_{j-1},a_j,c_j,d_j\}$ is a cocircuit.  Then $i \neq j-1$. 
But  \ort\ implies that $\{c_i,a_{i+1}\}$ meets $\{d_{j-1},a_j,c_j\}$; a \cn.  Hence \ref{newtutu} holds.

We may now assume that $a_0=c_n$. Moreover, by \ref{bign}, we may assume that $n \ge 2$. 
Next we show  the following. 

\begin{sublemma}
\label{tootiny}
The elements in $T_1\cup T_2\cup \dots \cup T_n\cup\{d_1,d_2,\dots ,d_n\}$ are  distinct.
\end{sublemma}

Clearly $a_1 \neq c_n$. Thus, by applying Lemma~\ref{minidrossrsv}, \ref{newtoo},  and \ref{newtutu} to the structure we get from Figure~\ref{drossfigi} by deleting $T_0 \cup d_0$, we deduce that 
the elements of $T_1\cup T_2\cup \dots \cup T_n\cup\{d_1,d_2,\dots ,d_n\}$ are distinct unless $a_1=d_{n-1}$. In the exceptional case,  
 \ort\ between $\{c_{n-1},a_1,a_n\}$ and $D_0$ implies that $\{c_{n-1},a_n\}$ meets 
 $\{b_0,c_0,b_1\}$; a \cn.  Thus \ref{tootiny} holds. 
 
Since $a_0=c_n$, \ort\ between $T_0$ and $\{d_{n-1},a_n,c_n,d_n\}$ implies that $\{b_0,c_0\}$ meets $\{d_{n-1},d_n\}$.  
If $d_{n-1}\in\{b_0,c_0\}$, then \ort\ between $\{c_{n-1},d_{n-1},a_n\}$ and $D_0$ implies that $\{c_{n-1},a_n\}$ meets $D_0$; a \cn.  
Therefore $d_n\in\{b_0,c_0\}$.  

\begin{sublemma}
\label{tootiny2}
$d_n = b_0$.
\end{sublemma}

Assume that this fails. Then $d_n= c_0$.  
Now \ort\ between $\{c_0,d_0,a_1\}$ and $\{d_{n-1},a_n,c_n,d_n\}$ implies, using \ref{tootiny},  that $d_0\in\{d_{n-1},a_n\}$.  
If $d_0=a_n$, then \ort\ between $\{c_0,d_0,a_1\}$ and $D_{n-1}$ gives a \cn. 
Thus $d_0=d_{n-1}$ so $n > 2$ otherwise $\lambda(T_1 \cup T_2  \cup \{d_1,d_2\}) \le 2$; a \cn.  Hence $\{d_0,a_1,c_1,d_1\}$ is a cocircuit. By \ref{tootiny}, this cocircuit meets the triangle $\{d_{n-1},c_{n-1},a_n\}$ in a single element; a \cn.   
We   conclude that \ref{tootiny2} holds.

If $\{d_{n-2},a_{n-1},c_{n-1},a_n,a_0\}$ is a cocircuit, then \ort\ with $T_0$ implies that $c_0=d_{n-2}$, and so $\lambda (T_{n-1}\cup T_n \cup  \{d_{n-1},b_0,c_0\})\leq 2$; a \cn.  
Thus $\{d_{n-2},a_{n-1},c_{n-1},d_{n-1}\}$ is a cocircuit.  
By \ort\ between $D_0$ and the triangles in Figure~\ref{drossfigi}, we can easily  check  that $c_0\notin T_1\cup T_2\cup\dots\cup T_n\cup\{d_1,d_2,\dots ,d_n\}$.  
Certainly $c_0\neq d_0$.  
Now, by \ort\ between $\{c_0,d_0,a_1\}$ and each of the indicated cocircuits in Figure~\ref{drossfigi} as well as $\{d_{n-2},a_{n-1},c_{n-1},d_{n-1}\}$, we deduce that $d_0\notin T_1\cup T_2\cup\dots \cup T_n\cup \{d_1,d_2,\dots ,d_n\}$.  Moreover, by taking the symmetric difference of this same set of cocircuits, we get the set $\{a_0,c_0,d_0,b_n\}$, so this set must also be a cocircuit of $M$.

Letting $Z = T_1\cup T_2\cup\dots\cup T_n\cup\{d_1,d_2,\dots ,d_n\} \cup \{c_0,d_0\}$, we see that $|Z| = 4n+2$. Moreover, we can easily check that $Z$ is spanned by 
$\{a_1,b_1,a_2,b_2,\ldots,a_n,b_n,d_n\}$ in $M$ and by 
$\{d_0,d_1,\ldots,d_n,b_1,b_2,\ldots,b_n,c_n\}$ in $M^*$. 
Thus $r(Z) \le 2n+1$ and $r^*(Z) \le 2n+2$.   

Suppose that $r(Z)\leq 2n$. Then $\lambda (Z) =0$  and $Z = E(M)$. Thus the elements in Figure~\ref{drossfigi} are all of the elements in $M$, where $(a_0,b_0)=(c_n,d_n)$.
Moreover,  $r(M) = 2n$ and $r^*(M) = 2n+2$. 
Since $\{a_1,b_1,a_2,b_2,\ldots,a_n,b_n,d_n\}$ spans $M$, this set must contain a circuit $C$. By \ort\ with the known $4$-cocircuits, including $\{a_0,c_0,d_0,b_n\}$, we deduce that $C$ avoids $\{a_1,a_2,\ldots,a_{n-1},b_n\}$. Hence $C \subseteq \{b_1,b_2,\ldots,b_{n-1},a_n,d_n\}$. Again \ort\ with the known $4$-cocircuits implies that $C =  \{b_1,b_2,\ldots,b_{n-1},a_n,d_n\}$. Hence $\{a_1,b_1,a_2,b_2,\ldots,a_n,b_n\}$ is a basis of $M$. 
Since $M$ is binary, 
we deduce that $M$ must be the  cycle matroid of the quartic M\"obius ladder that is obtained from Figure~\ref{drossfigi} by identifying the vertices $v_1,v_2,$ and $v_3$ with $v_4,v_5,$ and $v_6$, respectively; that is, (ii)(a) holds.  

We may now assume that  $r(Z)=2n+1$. Then $\lambda(Z) \le 1$. Suppose $\lambda(Z) = 0$.  Then $E(M)  = Z$ and $r^*(M) = 2n+1$. Hence $\{d_0,d_1,\ldots,d_n,b_1,b_2,\ldots,b_n,c_n\}$ contains a cocircuit $C^*$.  By \ort\ with the triangles in $M$, we deduce that $C^*$ avoids $\{d_0,b_1,d_1,b_2,\ldots,b_{n-1},d_{n-1}\}$. Hence $C^* \subseteq \{d_n,b_n,c_n\}$. This contradicts the fact that $M$ is \ifc. It follows that $\lambda(Z) = 1$, so $E(M) - Z$ has a unique element, say $\gamma$.  Now $\cl(T_1 \cup T_2 \cup \dots \cup T_n)$ has rank $2n$ and avoids $\{b_0,c_0,d_0,\gamma\}$. Hence the last set is a cocircuit of $M$. In addition, by symmetry, $\{b_i,c_i,d_i,\gamma\}$ and $\{d_{i-1},a_i,b_i,\gamma\}$ are cocircuits of $M$ for all $i$ in $\{1,2,\ldots,n\}$. It follows that every element of $M\ba \gamma$ is in both a triangle and a triad of that matroid. But $\gamma$ itself cannot be in a triad of $M$ since every element of $E(M) - \gamma$ is in a triangle of $M\ba \gamma$. Thus $M\ba \gamma$ is \thc\ and so  is a wheel. The rank of this wheel is $\tfrac{1}{2}(|E(M)| - 1) = 2n+1$. From the set of triangles of $M\ba \gamma$, we see that the spokes of this wheel, in cyclic order are $c_0,a_1,c_1,a_2,c_2,\dots,a_n,c_n$. These spokes form a basis, $B$, of $M$. From the $4$-cocircuits of $M$ containing $\gamma$, we see that  the fundamental circuit of $\gamma$ with respect to $B$ is $B \cup \gamma$. Finally, the cyclic order on the spokes determines that on the rim, namely, 	$d_0,b_1,d_1,b_2,d_2,\dots,b_n,d_n$. Hence (ii)(b) holds.
\end{proof}

\begin{lemma}
\label{dross} 
Let $M$ be an \ifc\ binary matroid that has at least thirteen elements. Assume that $M$ 
contains the configuration shown in Figure~\ref{drossfigi} where $n\ge 2$,  all of the elements shown are distinct, and, in addition to the cocircuits shown, exactly one of 
$\{d_{n-2},d_{n-1},a_{n-1},c_{n-1}\}$ or $\{d_{n-2},a_{n-1},c_{n-1},a_n,c_n\}$ is a cocircuit. Then either 
\begin{itemize}
\item[(i)] $M\ba c_0,c_1,\ldots, c_n$ is \ifc; or 
\item[(ii)] $M$ has a triangle containing $\{d_{n-1},d_n\}$, the matroid $M\ba c_n$ is not \ffsc, and $M$ has $\{d_{n-2},a_{n-1},c_{n-1},a_n,c_n\}$ as a cocircuit; or
\item[(iii)] $M\ba c_0,c_1,\ldots, c_n$ is \ffsc\ but not \ifc,  and   one side of 
every \ftv\ of $M\ba c_0,c_1,\ldots, c_n$ is a $4$-fan 
$F=(u_1,u_2,u_3,u_4)$ where either $u_4 = d_0$ and $a_0 \in \{u_2,u_3\}$, and 
$F$ is a $4$-fan of $M\ba c_0$; or
$u_4 = b_n$ and 
$F$ is a $4$-fan of $M\ba c_n$; or
 \item[(iv)] 
$M$ is the cycle matroid of a quartic M\"obius ladder labelled as in Figure~\ref{drossfigii} where the two vertices labelled $v$ are identified and the two vertices labelled $u$ are identified.
\end{itemize}
\end{lemma}

\begin{proof} 
It follows easily from Lemma~\ref{trywhere} that if $M\ba c_0,c_1,\ldots,c_n$ has a triangle meeting $\{d_{n-1},d_n\}$, then \ort\ implies that $\{d_{n-2},a_{n-1},c_{n-1},d_{n-1}\}$ is not a cocircuit, so $\{d_{n-2},a_{n-1},c_{n-1},a_n,c_n\}$ is a cocircuit and (ii) holds since $M\ba c_n$ has a $5$-fan. Thus we shall assume that $M$ has no triangle meeting $\{d_{n-1},d_n\}$. Hence a triangle of $M\ba c_0,c_1,\ldots ,c_n$ that meets 
$\{a_0,b_0,d_0,a_1,b_1,d_1,\ldots,a_n,b_n,d_n\}$ does so in $\{a_0\}$.

Let $S = \{c_0,c_1,\ldots,c_n\}$. 
Lemma~\ref{deletecs} implies that either $M\ba S$ is \thc, or $M\ba S$ has $a_i$ or $b_i$ in a $1$- or $2$-element cocircuit for some $i$ in $\{2,3,\dots ,n\}$.  
Suppose the latter.  
Then $\{x,y\}$ is a cocircuit of $M\ba S$ for $x\in\{a_i,b_i\}$, so $M$ has a cocircuit $C^*$ such that $\{x,y\}\subseteq C^*\subseteq \{x,y\}\cup S$.  
By \ort, if $c_j \in C^*$,  then $C^*$ meets $\{a_j,b_j\}$. Hence $C^*$ contains at most two elements in $S$.  
Furthermore, if $C^*$ contains only one element in $S$, then $C^*$ is a triad that meets a triangle of $M$; a \cn.  
Thus $C^*=\{x,y,c_j,c_k\}$ for some $k\neq j$.  
Then \ort\ between $C^*$ and the triangles $T_j$ and $T_k$ implies, without loss of generality, that $y\in\{a_j,b_j\}$, that $i=k$, and that $j<k$.  
Orthogonality with the triangle $\{c_j,d_j,a_{j+1}\}$ implies that $a_{j+1}=a_k=x$.  
If $\{c_k,d_k,a_{k+1}\}$ is a triangle, then it contains exactly one element of $C^*$; a \cn.  
Thus $j+1=n$.  
Orthogonality with $\{c_{j-1},d_{j-1},a_{j}\}$ implies that $a_j\notin C^*$. Hence $b_j\in C^*$, and $C^*\btu \{b_j,c_j,a_k,b_k\}$, which equals $\{b_k,c_k\}$, is a cocircuit in $M$; a \cn.  
We conclude that $M\ba S$ is \thc.  

Next we show that 

\begin{sublemma}
\label{mdelssfc}
$M\ba S$ is \sfc. 
\end{sublemma}

Let $(U,V)$ be a \ns\ $3$-separation of $M\ba S$. Then we may assume that $\{b_0,a_1,b_1\} \subseteq U$. If $\{a_0,d_0,d_1,a_2\}$ meets $U$, then $\fcl_{M\ba S}(U) \supseteq \{a_0,d_0,d_1,a_2\}$ and it is straightforward to check that $\fcl_{M\ba S}(U)$ contains $\{a_0,b_0,d_0,a_1,b_1,d_1,\ldots,a_n,b_n,d_n\}$. Hence $(\fcl _{M\ba S}(U) \cup S,V-\fcl (U))$ is a \ns\ $3$-separation of $M$; a \cn. Thus we may assume that $\{a_0,d_0,d_1,a_2\} \subseteq V$. Then $\fcl_{M\ba S}(V)$ contains $\{a_0,b_0,d_0,a_1,b_1,d_1,\ldots,a_n,b_n,d_n\}$, so $(U-\fcl (V), \fcl (V) \cup S)$ is a \ns\ $3$-separation of $M$; a \cn.  Hence \ref{mdelssfc} holds.

Next we show the following. 

\begin{sublemma}
\label{mdelsfan2}
Each $4$-fan of $M\ba S$  is either  a fan of $M\ba c_0$ and has $d_0$ as its coguts element and $a_0$ as an interior element, or is a fan of $M\ba c_n$ and has $b_n$ as its coguts element.
\end{sublemma}

Let $(u_1,u_2,u_3,u_4)$ be a $4$-fan in $M\ba S$. Then $M$ has a cocircuit $C^*$ such that 
$\{u_2,u_3,u_4\} \subsetneqq C^* \subseteq \{u_2,u_3,u_4\} \cup S$. We now  show that 

\begin{sublemma}
\label{mdelsfan}
$|S \cap C^*| = 1$. 
\end{sublemma}

Assume that $|S \cap C^*| \ge 2$. Then $c_i \in C^*$ for some $i > 0$. Hence $C^*$ meets $\{a_i,b_i\}$, and 
 it follows, by Lemma~\ref{trywhere}, that $u_4 \in \{a_i,b_i\}$. Thus $c_i$ is the unique element of $S - c_0$ that is in $C^*$, so $c_0 \in C^*$. Therefore, by \ort,
 $\{u_2,u_3\}$ meets both $\{a_0,b_0\}$ and $\{d_0,a_1\}$; a \cn\ to  Lemma~\ref{trywhere}.  
We conclude that \ref{mdelsfan} holds.

Using \ref{mdelsfan}, suppose first that $c_0 \in C^*$. Then $C^* = \{u_2,u_3,u_4,c_0\}$ and $u_4 \in \{d_0,a_1\}$. Moreover, $a_0 \in \{u_2,u_3\}$. If $u_4 = a_1$, then $b_1 \in \{u_2,u_3\}$; a \cn. Thus $u_4 = d_0$. We deduce that the $4$-fan $(u_1,u_2,u_3,u_4)$ has $d_0$ as its coguts element and has $a_0$ as an interior element.

Next suppose that $c_n \in C^*$. Then $u_4 \in \{a_n,b_n\}$. If $u_4 = a_n$, then $\{u_2,u_3\}$ meets $\{d_{n-1},c_{n-1}\}$. Thus $d_{n-1} \in \{u_2,u_3\}$ so, by \ort, $\{d_{n-1},d_n\} \subseteq \{u_1,u_2,u_3\}$; a \cn. Thus $u_4 = b_n$.

Finally, suppose that $c_i \in C^*$ for some $i$ with $0 < i < n$. Then $u_4 \in \{a_i,b_i\}\cap \{d_i,a_{i+1}\}$; a \cn. 
Thus \ref{mdelsfan2} holds.  

We shall now assume that $M\ba S$ is not \ffsc, otherwise (i) or (iii) holds.  Next we show the following. 

\begin{sublemma}
\label{mdelsfan3}
If  $(U,V)$ is a \ffsv\ of $M\ba S$, then $U$ or $V$ is a $5$-cofan of the form $(d_0,a_0,\alpha,\beta,b_n)$. Moreover, 
$\{\al,\be\} \cap \{a_i,b_i,c_i,d_i: 0 \le i \le n\} = \emptyset.$
\end{sublemma}

To see this, first suppose that $M\ba S$ has a $5$-fan $(w_1,w_2,w_3,w_4,w_5)$. Then, by \ref{mdelsfan2}, we may assume that $(w_2,w_4) = (b_n,d_0)$. Thus $M\ba S$ has a triangle containing $d_0$; a \cn\ to Lemma~\ref{trywhere}. Therefore $M$ has no $5$-fans. 

Next let $(w_1,w_2,w_3,w_4,w_5)$ be a $5$-cofan in $M\ba S$. Then, by~\ref{mdelsfan2}, we may assume that 
$(w_1,w_5) = (d_0,b_n)$ and $a_0 \in \{w_2,w_3\}$, and $\{d_0,w_2,w_3,c_0\}$ and $\{w_3,w_4,b_n,c_n\}$ are cocircuits of $M$.  
If $a_0 = w_3$, then $M$ has $\{a_0,w_4,b_n,c_n\}$ as a cocircuit and has $\{a_0,b_0,c_0\}$ as a circuit. Using Lemma~\ref{trywhere}, we see that  $w_4 \not \in \{b_0,c_0\}$. Thus we have a \cn\ to \ort. We deduce that $a_0 = w_2$. 

Now, by \cite[Lemma 2.2(iv)]{cmoIV}, if the first sentence of \ref{mdelsfan3} is false, then the 
$5$-cofan $(w_1,w_2,w_3,w_4,w_5)$ has an element in its coguts or its guts. 
It is clear that the $5$-cofan has no element in its coguts since adjoining such an element to the $5$-cofan gives 
 a $6$-element set that contains three elements that are ends of $5$-cofans. Yet each such end must be in $\{d_0,b_n\}$.  
Now assume that  $(w_1,w_2,w_3,w_4,w_5)$  has an element $w_6$ in its guts. Then $M\ba S$ has $\{d_0,w_3,b_n,w_6\}$ as a circuit $C$. By \ort, $C$ meets $\{d_1,a_1\}$ and $\{b_{n-1},a_n\}$ so $w_3 \in \{d_1,a_1,b_{n-1},a_n\}$ and the triangle $\{w_2,w_3,w_4\}$ gives a \cn\ to Lemma~\ref{trywhere}. 

We conclude that if $M\ba S$ has a \ffsv, then it is the $5$-cofan $(d_0,a_0,w_3,w_4,b_n)$. Writing $\alpha$ and $\beta$ for $w_3$ and $w_4$, respectively, we get that the first sentence of \ref{mdelsfan3} holds. Moreover, 
 $M$ has $\{\alpha,\beta,a_0\}$ as a circuit. 


Since $\{\alpha, \beta,a_0\}$ is a triangle of $M\ba S$, Lemma~\ref{trywhere} implies that either $\{\alpha, \beta\}$ avoids $\{a_i,b_i,c_i,d_i: 0 \le i \le n\}$ or $\{\alpha, \beta\} = \{d_{n-1},d_n\}$. The latter contradicts our assumption. We deduce that the former must hold. Hence \ref{mdelsfan3} holds.

We now know that $M$ has $\{\alpha,\beta,a_0\}$ as a circuit 
  and has $\{b_n,c_n,\alpha,\beta\}$ and $\{d_0,c_0,a_0,\alpha\}$ as cocircuits. Next we 
aim to show that (iv) of Lemma~\ref{dross} holds. Let $Y= \{\beta,c_0,\ldots,c_n,a_0,a_1,\ldots,a_n\}$ and $Z = Y \cup \{\alpha,d_0,\ldots,d_{n-1}\} \cup \{b_0,b_1,\ldots,b_n\}$. 
Evidently $Z$ is spanned by $Y$. Thus $r(Z) \le 2n+3$. Since we know of $2n+1$ cocircuits that are contained in $Z$, none of which is the symmetric difference of any others, we deduce that 

\begin{sublemma}
\label{banddp} 
$\lambda(Z) \le 2$. 
\end{sublemma}

Next we show that 

\begin{sublemma}
\label{bandd0} 
$M$ has cocircuits that are contained in $Z \cup d_n$ and meet $Y$ in each of $\{a_0,c_0\}, \{c_0,a_1\}, \{a_1,c_1\},\ldots, \{a_{n-1},c_{n-1}\},\{c_{n-1},a_n\}, \{a_n,c_n\}$, and $\{c_n,\be\}$. Thus a circuit of $M$ whose intersection with $Z \cup d_n$ is a non-empty subset of $Y$ must contain $Y$.
\end{sublemma}

This is immediate for all the indicated pairs except $\{a_{n-1},c_{n-1}\}$. It is also true for the last  pair unless $\{d_{n-2},a_{n-1},c_{n-1},d_{n-1}\}$ is not a cocircuit of $M$. In the exceptional case, 
$\{d_{n-2},a_{n-1},c_{n-1},a_n,c_n\}$ is a cocircuit. Since $\{d_{n-1},a_n,c_n,d_n\}$ is also a cocircuit, the symmetric difference of the last two sets, which equals 
$\{d_{n-2},d_{n-1},a_{n-1},c_{n-1},d_n\}$, is also a cocircuit. In this case, we again get that $M$ has a cocircuit meeting $Y$ in $\{a_{n-1},c_{n-1}\}$. Hence \ref{bandd0} holds.

Now consider the sets $\{b_0,b_1,\ldots,b_n,\beta\}$ and $\{\alpha,d_0,d_1,\ldots,d_{n-1},c_n\}$, which we denote by $B$ and $D$, respectively. Next we show the following. 

\begin{sublemma}
\label{bandd} 
Either both $B$ and $D$ are circuits of $M|(B \cup D)$; or $B \cup D$ is a circuit of $M|(B \cup D)$. 
\end{sublemma}

Since $B \cup D$ is the symmetric difference of a set of triangles of $M$,  it is a disjoint union of circuits.  Assume that $B\cup D$ is not a circuit. For each $i$ in $\{0,1,\ldots,n-1\}$, since $M$ has a cocircuit that meets 
$B \cup D$ in $\{b_i,b_{i+1}\}$, it follows that a circuit of $M|(B\cup D)$ that meets $B- \beta$ must contain $B-\beta$. Similarly, the cocircuits of $M$ shown in Figure~\ref{drossfigii} imply that a circuit of $M|(B\cup D)$ that meets $\{\alpha,d_0,\ldots,d_{n-2}\}$   must contain $\{\alpha,d_0,d_1,\ldots,d_{n-2}\}$. Moreover, $M$ has a cocircuit that meets $B \cup D$ in $\{d_{n-1},c_n\}$ and has a cocircuit that meets $B\cup D$ in either $\{d_{n-2},d_{n-1}\}$ or $\{d_{n-2},c_n\}$. Hence every circuit of $M|(B\cup D)$ that meets $\{\alpha,d_0,d_1,\ldots,d_{n-2}\}$   must contain $D$. Let $C$ be a circuit of $M|(B\cup D)$ that contains $D$. If \ref{bandd} fails, then $\beta \in C$ and $B - \be$ is a circuit of $M$. But the last circuit contradicts \ort\ with  the cocircuit $\{b_n,c_n,\al,\be\}$. Hence \ref{bandd} holds.

Next we show that 

\begin{sublemma}
\label{bandd2} 
$B \cup D$ is not a circuit of $M$. 
\end{sublemma}

Assume that $B \cup D$ is  a circuit of $M$. Then, as $|B \cup D| = 2n+4$, we deduce that  $r(Z) = 2n+3$ so $Y$ is a basis of $Z$. Suppose first that $\lambda(Z) = 2$. 
Then $|E(M)-Z|\in\{2,3\}$.  Suppose $ E(M)-Z = \{x_1,x_2\}$. Then $Y$ is a basis of $M$. By \ref{bandd0}, the fundamental circuits $C(x_1,Y)$ and $C(x_2,Y)$ are  $\{x_1\}\cup Y$ and $\{x_2\}\cup Y$, respectively.  
Thus  $M$ has $\{x_1,x_2\}$ as a circuit; a \cn.  
It follows that $|E(M) - Z| = 3$ so $E(M) - Z$ is a triangle or a triad containing $d_n$.   Since $M$ has a cocircuit that contains $d_n$ and is contained in $Z \cup d_n$, it follows that $E(M) - Z$ is a triad, say $\{d_n,x,y\}$. Now $\{d_n,a_0,c_0,a_1,c_1,\ldots,a_n,c_n\}$ spans a hyperplane of $M$ whose complementary cocircuit, $C^*$, is contained in $\{\beta,\alpha,x,y\}$. As $M$ has no $4$-fans, it follows that $C^* = \{\beta,\alpha,x,y\}$. But the symmetric difference of $C^*$ and $\{d_n,x,y\}$ is $\{\beta,\alpha,d_n\}$; a \cn. 

We may now assume that $\lambda(Z) \le 1$. Since $E(M) - Z$ contains $d_n$, it follows that  $E(M) - Z = \{d_n\}$ 
and $Y$ is a basis of $M$. Then $Y-c_n$ spans a hyperplane of $M$ whose complementary cocircuit is contained in $\{b_n,c_n,d_n\}$. Thus $M$ has a $4$-fan; a \cn. We conclude that \ref{bandd2} holds.

By \ref{bandd} and \ref{bandd2}, both $B$ and $D$ are circuits of $M$, so $r(Z) \le 2n+2$. Thus $\lambda(Z) \le 1$, 
so $E(M) - Z = \{d_n\}$. Since $E(M) - Z$ contains $d_n$, it follows that $\lambda(Z) = 1$ and $r(Z) = 2n+2 = r(M)$. Now $Y$ spans $M$ and has $r(M) + 3$ elements, so it contains a circuit. 
By \ref{bandd0}, $Y$ is a circuit of $M$. Thus $Y - \beta$ is a basis of $M$. 

To complete the proof that $M$ is the cycle matroid of the quartic M\"{o}bius ladder labelled as in Figure~\ref{drossfigii}, we first observe that both matroids have $Y - \beta$ as a basis. Since both matroids are binary, it suffices to show that they have the same fundamental circuits with respect to this basis. Evidently the fundamental circuits of each of 
$\be,b_0,b_1,\ldots,b_n,d_0,d_1,\ldots,d_{n-1}$ are the same. Moreover, the cocircuit $\{\al,\be,b_n,c_n\}$ implies that the fundamental circuit $C_M(\al, Y - \be)$ must contain $c_n$ and so, by \ref{bandd0}, $C_M(\al, Y - \be)$ contains $a_n,c_{n-1},a_{n-1},\ldots,a_1,c_0$. Since $M$ is binary, $C_M(\al,Y -\be)$ does not contain $a_0$ and so is $\al \cup (Y - \{\be,a_0\})$, which is also a circuit in the cycle matroid of the quartic M\"{o}bius ladder.

Finally, let $C' = C_M(d_n,Y- \be)$.  By \ort, exactly one of $a_n$ and $c_n$ is in $C'$. As $\beta \not \in C'$, it follows by \ort\ that $c_n \not \in C'$. Thus $a_n \in C'$, so $c_{n-1} \in C'$.  
Suppose that $a_{n-1} \not \in C'$. Then, by \ort, none of $c_{n-2},a_{n-2},c_{n-3},a_{n-3},\ldots,c_0,a_0$ is in $C'$, so $C' = \{d_n,a_n,c_{n-1}\}$. This is a \cn\ as $M$ has $\{d_{n-1},a_n,c_{n-1}\}$ as a circuit. We deduce that 
$a_{n-1}  \in C'$, so all of $c_{n-2},a_{n-2},c_{n-3},a_{n-3},\ldots,c_0,a_0$ are in $C'$. Thus $C' = 
\{a_0,c_0,a_1,c_1,\ldots,a_{n-1},c_{n-1},a_n,d_n\}$. 
It now follows that $M$ is indeed the cycle matroid of the quartic M\"{o}bius ladder in Figure~\ref{drossfigii}. 
This completes the proof of the lemma. 
\end{proof}

\section{A quick wrap}
\label{qwrap}

In this section, we deal with a situation when a short string of bowties wraps around on itself as in Figure~\ref{snake0}. Observe that  $T_0,D_0,T_1,D_1,T_2$ is not contained in a ring of bowties since $\{b_2,c_2,a_0,c_0\}$ is properly contained in a cocircuit, while $\{a_2,c_2,a_0,b_0\}$ is not a cocircuit otherwise we obtain the \cn\ that $\lambda(T_0 \cup T_1 \cup T_2) \le 2$. Lemma~\ref{killthesnake} deals with this situation but the argument is long and technical. The following preliminary lemma will be used several times in this  proof. 

\begin{figure}[htb]
\center
\includegraphics{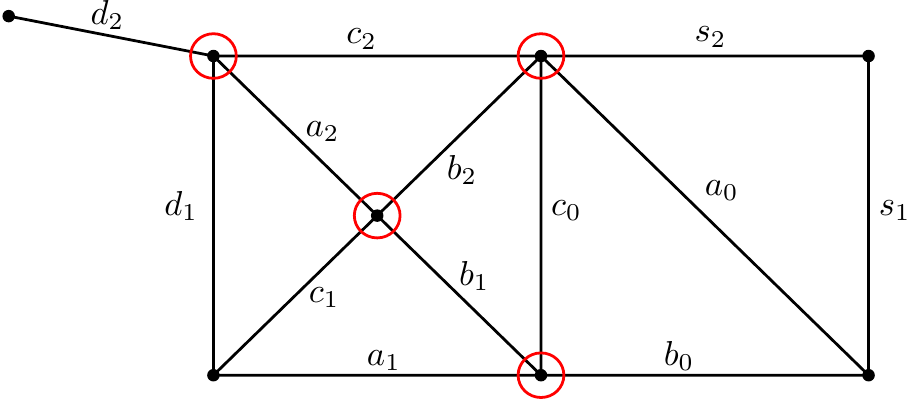}
\caption{}
\label{snake0}
\end{figure}

\begin{lemma}
\label{deletechainsfc}
Let $M$ be a binary \ifc\ matroid having at least thirteen elements and let $N$ be an \ifc\ proper minor of $M$ that has at least seven elements.  
Suppose $M$ has a rotor chain $((a_0,b_0,c_0),(a_1,b_1,c_1),\dots,(a_n,b_n,c_n))$ 
and that $M\ba c_0,c_1$ has an $N$-minor.  
Then either $M\ba c_0,c_1,\dots ,c_n$ is \sfc\ with an $N$-minor,  or $M$ has a minor $M'$ such that $M'$ is \ifc\ with an $N$-minor and $1\leq |E(M)|-|E(M')|\leq 3$.  
\end{lemma}
\begin{proof}
Let $S=\{c_0,c_1,\dots ,c_n\}$.  
We may assume that $M/b_1$ has no $N$-minor, otherwise, by Lemma~\ref{rotorwin},   the lemma holds.  
Thus $M\ba c_0,c_1/b_1$ has no $N$-minor, but $M\ba c_0,c_1$ has an $N$-minor.  
By Lemma~\ref{stringybark}, we know that $M\ba c_0,c_1,\dots ,c_i/b_i$ has no $N$-minor for all $i\in\{1,2,\dots ,n\}$, that $M\ba c_0,c_1,\dots ,c_j/a_j$ has no $N$-minor for all $j\in\{2,3,\dots, n\}$, and that 
$M\ba S$ has an $N$-minor.  
Lemma~\ref{deletecs} implies that either $M\ba S$ is \thc,  or $M\ba S$ has $a_i$ or $b_i$ in a cocircuit of size at most two for some $i$ in $\{2,3,\dots ,n\}$.  
The latter implies that $M\ba S/a_i$ or $M\ba S/b_i$ has an $N$-minor; a \cn.  
Hence  $M\ba S$ is \thc.



Let $(X,Y)$ be a \ns\ \ths\ of $M\ba S$.  
Without loss of generality, we may assume that the triad $\{b_0,a_1,b_1\}\subseteq X$.  
If $a_0,a_2$, or $b_2$ is in $X$, then all of them are in the full closure of $X$, and we may assume that $\{a_0,a_2,b_2\}\subseteq X$.  
Then the full closure of $X$ in $M\ba S$ contains $\{a_0,b_0,a_1,b_1,\dots ,a_n,b_n\}$, and we see that 
$(\fcl_{M\ba S} (X)\cup S,Y-\fcl_{M\ba S} (X))$ is a \ns\ \ths\ of $M$; a \cn.  
We may now assume that $\{a_0,a_2,b_2\}\subseteq Y$.  
Then the full closure of $Y$ in $M\ba S$ contains $b_0,a_1$, and $b_1$,  and we obtain the same \cn\ as before. 
\end{proof}

Beginning with the next lemma and for the rest of the paper, we shall start abbreviating how we refer to the following three outcomes in the main theorem.

\begin{itemize}
\item[(i)] $M$ has a proper minor $M'$ such that $|E(M)|-|E(M')|\leq 3$ and $M'$ is \ifc\ with an $N$-minor; 
\item[(ii)] $M$ contains an open rotor chain, a ladder structure, or a ring of bowties that can be trimmed to obtain an \ifc\ matroid with an $N$-minor; 
\item[(iii)] $M$ contains an enhanced quartic ladder from which an \ifc\ minor of $M$ with an $N$-minor can be obtained by an enhanced-ladder move.
\end{itemize}
When (i) or (iii)  holds, we say, respectively, that $M$ has a {\it quick win} or an {\it enhanced-ladder win}.  
When trimming an open rotor chain, a ladder structure, or a ring of bowties in $M$ produces an \ifc\ matroid with an $N$-minor,   we say, respectively, that $M$ has an {\it open-rotor-chain win}, a {\it ladder win}, or a {\it bowtie-ring win}.  

\begin{lemma}
\label{killthesnake}
Let $M$ and $N$ be   \ifc\ binary matroids with   $|E(M)|\geq 13$ and $|E(N)|\geq 7$. Assume that  $M$ contains the structure in Figure~\ref{snake0} and that $M\ba c_0,c_1,c_2,s_1$ has an $N$-minor.  
Then 
\begin{itemize}
\item[(i)] $M$ has a quick win; or 
\item[(ii)] $\{d_1,d_2\}$ is contained in a triangle of $M$; or
\item[(iii)] $\{b_0,a_1\}$ is contained in a triangle of $M$; or
\item[(iv)] $M$ has an open-rotor-chain win or a ladder win; or 
\item[(v)] $M$ has an enhanced-ladder win.
\end{itemize}
\end{lemma}

\begin{proof}
We assume that none of (i)--(v) holds.   We show first that

\begin{sublemma}
\label{dist12} all the elements in Figure~\ref{snake0} are distinct.
\end{sublemma}

With $T_i = \{a_i,b_i,c_i\}$ for all $i$, we see that 
 $\lambda (T_{0}\cup T_1\cup T_{2})\leq 3$. Since $|E(M)| \ge 13$, we have that $|T_0 \cup T_1 \cup T_2| = 9$ otherwise   
  $\lambda (T_{0}\cup T_1\cup T_{2})\leq 2$. 
  Clearly $s_2 \not\in T_0 \cup T_1 \cup T_2$ otherwise $\lambda(T_0 \cup T_1 \cup T_2) \le 2$. Thus $s_1 \neq a_0$, so $s_1 \not\in T_0 \cup s_2$. Moreover, $s_1 \not\in \{b_2,c_2\}$ otherwise a triangle is contained in a cocircuit. By \ort\ between the triangle $\{s_1,s_2,a_0\}$ and the cocircuits $\{a_1,b_1,b_0,c_0\}$ and $\{b_1,c_1,a_2,b_2\}$, we see that $s_1 \not \in \{a_1,b_1\}$ and $s_1 \not \in \{a_2,b_2\}$. Thus $s_1 \not \in T_0 \cup T_1 \cup T_2$. Hence $\{s_1,s_2\}$ avoids $T_0 \cup T_1 \cup T_2$. If $\{d_1,d_2\}$ meets $T_0 \cup T_1 \cup T_2$, then, since $M$ is binary, $\{d_1,d_2\}$ is contained in $T_0$ or $T_1$ and again we obtain the \cn\ that $\lambda (T_{0}\cup T_1\cup T_{2})\leq 2$. Finally, if $\{d_1,d_2\}$ meets $\{s_1,s_2\}$, then, by \ort\ between the cocircuit $\{d_1,a_2,c_2,d_2\}$ and the triangle $\{a_0,s_1,s_2\}$, we deduce that $\{d_1,d_2\} = \{s_1,s_2\}$. Thus $\{d_1,d_2\}$ is contained in a triangle; a \cn. We conclude that \ref{dist12} holds.


\begin{figure}[htb]
\center
\includegraphics[scale=0.8]{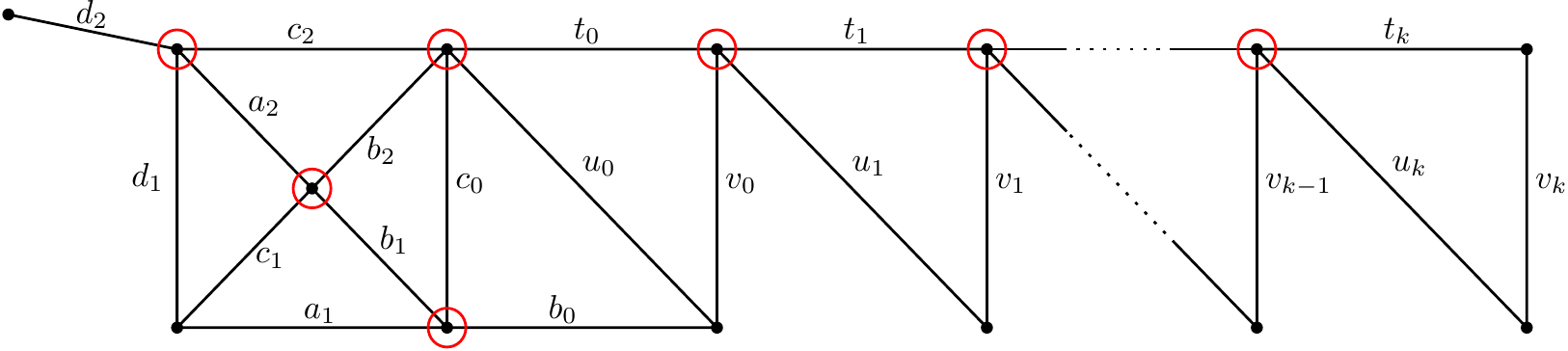}
\caption{The elements are distinct.}
\label{bonesaw0}
\end{figure}

We relabel $(s_2,a_{0},s_1)$ as $(t_0,u_0,v_0)$.  Then $T_0$ becomes $\{u_0,b_0,c_0\}$. 
Take $k$ to be maximal such that $\{t_0,u_0,v_0\},\{t_0,v_0,t_1,u_1\},\{t_1,u_1,v_1\},\linebreak \{t_1,v_1,t_2,u_2\},\dots ,\{t_k,u_k,v_k\}$ is a  string of bowties and all of the elements in Figure~\ref{bonesaw0} are distinct.  As $v_0 = s_1$, it follows by assumption that 
$M\ba c_{0},c_1,c_{2},v_0$ has an $N$-minor.  

We now show that 

\begin{sublemma}
\label{hasN}
$M\ba c_{0},c_1,c_{2},v_0,v_1,\dots ,v_i/t_i$ has no $N$-minor for every $i$ in $\{0,1,\dots ,k\}$ and
$M\ba c_{0},c_1,c_{2},v_0,v_1,\dots ,v_k$ has an $N$-minor.  
\end{sublemma}

Observe that $M$ has $(T_0,T_1,T_2,D_0,D_1,\{c_0,b_1,b_2\})$ as a quasi rotor. Thus, by 
Lemma~\ref{rotorwin}, 
$M/b_2$ has no $N$-minor.  
Note that $T_1,D_1,T_2,\{b_2,c_2,u_0,t_0\},\{u_0,t_0,v_0\},\{t_0,v_0,u_1,t_1\},\{u_1,t_1,v_1\},
\{t_1,v_1,u_2,t_2\},\dots ,\linebreak 
\{u_k,t_k,v_k\}$ is a string of bowties in $M\ba c_0$, and $M\ba c_0\ba c_1,c_2$ has an $N$-minor but $M\ba c_0\ba c_1,c_2/b_2$ has no $N$-minor.  
Lemma~\ref{stringybark} implies that~\ref{hasN} holds.

Next we show that 
\begin{sublemma}
\label{sfcgotit}
$M\ba c_{0},c_1,c_{2},v_0,\dots ,v_i$ is \sfc\ for all $i$ in $\{0,1,\dots ,k\}$.  
\end{sublemma}

Lemma~\ref{deletechainsfc} implies that $M\ba c_{0},c_1,c_{2}$ is \sfc, so it is \thc. By Tutte's Triangle Lemma \cite{wtt} (or see \cite[Lemma 8.7.7]{oxrox}), $M\ba c_{0},c_1,c_{2},v_0$ is \thc\ unless $M\ba c_{0},c_1,c_{2}$ has $v_0$ in a triad with an element $x$ in $\{t_0,u_0\}$.  
In the exceptional case, $M\ba c_0,c_1,c_2,v_0/x$ has an $N$-minor, so, by~\ref{hasN}, we know that $x=u_0$.  
Then $M\ba c_0,c_1,c_2,t_0/u_0$ and $M\ba c_0,c_1,c_2,t_0/b_2$ also have $N$-minors, and Lemma~\ref{rotorwin} gives a \cn.  
We conclude that $M\ba c_{0},c_1,c_{2},v_0$ is \thc.  

Suppose $(X,Y)$ is a \ns\ \ths\ of $M\ba c_{0},c_1,c_{2},v_0$.  
Without loss of generality, the triad $\{b_{2},t_0,u_0\}$ is contained in $X$. Hence $(X\cup v_0,Y)$ is a \ns\ \ths\ of $M\ba c_{0},c_1,c_{2}$; a \cn.  
We conclude that $M\ba c_{0},c_1,c_{2},v_0$ is \sfc, so \ref{sfcgotit} holds for $i = 0$. 

Now, for some $i$ in $\{1,2,\dots, k\}$, suppose  that $M\ba c_{0},c_1,c_{2},v_0,\dots ,v_j$ is \sfc\ for all $j<i$ but that $M\ba c_{0},c_1,c_{2},v_0,\dots ,v_i$ is not \thc. Then, by Tutte's Triangle Lemma again, 
  $M\ba c_{0},c_1,c_{2},v_0,\dots ,v_{i-1}$ has $v_i$ in a triad with $u_i$ or $t_i$. By~\ref{hasN}, this triad contains $u_i$, and $M\ba c_{0},c_1,c_{2},v_0,\dots ,v_i/u_i$ has an $N$-minor.  
Now $M\ba c_{0},c_1,c_{2},v_0,\dots ,v_i/u_i\cong M\ba c_{0},c_1,c_{2},v_0,\dots ,v_{i-1}/u_i\ba t_i\cong M\ba c_{0},c_1,c_{2},v_0,\dots ,v_{i-1},t_i/t_{i-1}$, so we obtain a \cn\ to~\ref{hasN}.  
We conclude  that $M\ba c_{0},c_1,c_{2},v_0,\dots ,v_i$ is \thc.  
Now suppose that $(X,Y)$ is a \ns\ \ths\ of $M\ba c_{0},c_1,c_{2},v_0,\dots ,v_i$.  
Without loss of generality, the triad $\{t_{i-1},u_i,t_i\}$ is contained in $X$, and $(X\cup v_i,Y)$ is a \ns\ \ths\ of $M\ba c_{0},c_1,c_{2},v_0,\dots ,v_{i-1}$; a \cn.  
We conclude, by induction, that \ref{sfcgotit} holds.

We show next that 
\begin{sublemma}
\label{notbt}
$M$ has no triangle $\{t_{k+1},u_{k+1},v_{k+1}\}$ such that $(\{t_k,u_k,v_k\},\{t_{k+1},u_{k+1},v_{k+1}\},\{z,v_k,t_{k+1},u_{k+1}\})$ is a bowtie in $M$ for some $z$ in $\{u_k,t_k\}$.  
\end{sublemma}

Suppose instead that $M$ has such a triangle $T=\{t_{k+1},u_{k+1},v_{k+1}\}$. Clearly, $T$ avoids $\{t_k,u_k,v_k\}$. 
By the choice of $k$, it follows that $T$ must contain some element in Figure~\ref{bonesaw0}.

As a step towards proving \ref{notbt},  we show next that 
\begin{sublemma}
\label{notbt2}
$T$ avoids $T_0\cup T_1\cup T_2\cup \{d_1,d_2,t_0\}$.
 \end{sublemma}
 
Suppose that $T$ meets $T_0\cup T_1\cup T_2\cup \{d_1,d_2,t_0\}$. Since the last set is a union of vertex cocircuits in Figure~\ref{bonesaw0}, $T$ must contain at least two elements of $T_0\cup T_1\cup T_2\cup \{d_1,d_2,t_0\}$. 
Orthogonality  with these vertex cocircuits implies that $T$ is contained in $T_0\cup T_1\cup T_2\cup \{d_1,d_2,t_0,v_0\}$;  otherwise $T$ contains $\{d_1,d_2\}$ or $\{b_0,a_1\}$, a \cn. 
Thus one easily checks that either $T$ is one of the  triangles shown in Figure~\ref{bonesaw0}, or $T$ meets $\{t_0,u_0\}$.

Assume that the first possibility does not hold. 
Since $T \subseteq T_0 \cup T_1 \cup T_2 \cup \{d_1,d_2,t_0,v_0\}$, we see that  $t_0 \not \in T$, otherwise, by \ort, $k = 0$ and $T$ meets $T_k$; a \cn. If $u_0 \in T$, then orthogonality implies that $T$ is $\{u_0,c_2,d_2\}$ or $\{u_0,b_2,c_1\}$. The latter possibility is excluded because $\{u_0,b_2,c_1,a_1,b_0\}$ is a circuit. 
If $T=\{u_0,c_2,d_2\}$, then the cocircuit $\{z,v_k,t_{k+1},u_{k+1}\}$ meets either $T_2$ or $T_0$ in a single element; a \cn. We conclude that $T$ is one of the triangles shown in Figure~\ref{bonesaw0}.

Suppose that $T$ meets $\{b_1,c_1,a_2,b_2\}$. Then $\{b_1,c_1,a_2,b_2\}$ meets the cocircuit $\{z,v_k,t_{k+1},u_{k+1}\}$ in a subset of $\{t_{k+1},u_{k+1}\}$. But each of $b_1, c_1, a_2$, and $b_2$ is in two triangles, and orthogonality implies that a second element of each of these triangles must also be in $\{t_{k+1},u_{k+1}\}$; a \cn. We conclude that $T$ avoids $\{b_1,c_1,a_2,b_2\}$, so $T$ is $\{u_0,b_0,c_0\}$ or $\{u_0,t_0,v_0\}$. But $T$ avoids $T_k$ so $k > 0$. If $\{u_0,c_0\}$ meets the cocircuit $\{z,v_k,t_{k+1},u_{k+1}\}$, then, because each of $u_0$ and $c_0$ is in two triangles, we again obtain the contradiction that $\{t_{k+1},u_{k+1}\}$ contains at least three elements. Thus $\{u_0,c_0\}$ avoids $\{z,v_k,t_{k+1},u_{k+1}\}$ so $T = \{u_0,t_0,v_0\}$ and $\{t_{k+1},u_{k+1}\} = \{t_0,v_0\}$. Hence $\{z,v_k,t_{k+1},u_{k+1}\} = \{z,v_k,t_0,v_0\}$, so  $M\ba c_0,c_1,c_2,v_0,v_1,\dots ,v_k$ has $\{t_0,z\}$ as a  cocircuit. Thus $M\ba c_0,c_1,c_2,v_0/t_0$ has an $N$-minor; a \cn\ to~\ref{hasN}. We conclude that  \ref{notbt2} holds. 

We now know that $T$ meets $\{v_0,t_1,u_1,v_1,\dots ,t_{k-1},u_{k-1},v_{k-1}\}$. 
If $k = 0$, then $v_0 \in \{t_1,u_1,v_1\}$; a \cn.  
Hence $k \ge 1$. Thus, by Lemma~\ref{ring}, either $T = \{u_i,t_i,v_i\}$ for some $i$ with $0 \le i \le k-2$, or $T$ meets $\{u_0,t_0,v_0\}$ in $\{u_0\}$. By \ref{notbt2}, the latter does not occur and so the former occurs with $i \ge 1$. 
Hence  $k\geq 3$.  
If $v_i\in\{t_{k+1},u_{k+1}\}$, then $M\ba c_0,c_1,c_2,v_0,v_1,\dots ,v_k$ has $z$ in a $2$-element cocircuit, so $M\ba c_0,c_1,c_2,v_0,v_1,\dots ,v_k/z$ has an $N$-minor. 
Let $y$ be the element in $\{t_k,u_k\}$ that is not $z$. 
Then $M\ba c_0,c_1,c_2,v_0,v_1,\dots ,v_k/z\cong M\ba c_0,c_1,c_2,v_0,v_1,\dots ,v_{k-1},y/z\cong M\ba c_0,c_1,c_2,v_0,v_1,\dots ,v_{k-1},y/t_{k-1}$; a \cn\ to~\ref{hasN}. 
We may assume then that $\{t_i,u_i\}=\{t_{k+1},u_{k+1}\}$. It follows that $\{z,v_k,t_{k+1},u_{k+1}\}\btu \{t_{i-1},v_{i-1},t_i,u_i\}$, which equals $\{z,v_k,t_{i-1},v_{i-1}\}$, is a cocircuit of $M$. Thus $M\ba c_0,c_1,c_2,v_0,v_1,\dots ,v_k/t_{i-1}$ has an $N$-minor; a \cn\ to~\ref{hasN}. 
We conclude that \ref{notbt} holds. 

Next we  show that 
\begin{sublemma}
\label{deletev0}
$M\ba v_0$ is \sfc.  
\end{sublemma}
By~\cite[Lemma~3.1]{cmoV},  $M\ba v_0$ is \thc.  
Suppose $(U,V)$ is a \ns\ \ths\ of $M\ba v_0$.  
By~\cite[Lemma~3.3]{cmoV}, we may assume that $\{u_0,b_0,c_0\} \cup T_1\cup b_2 \subseteq U$. It follows that we may assume that $\{u_0,b_0,c_0\} \cup T_1\cup b_2 \cup a_2 \cup c_2 \cup t_0\subseteq U$. Then 
 $(U \cup v_0,V)$ is a \ns\ \ths\ of $M$; a \cn.  
We conclude that \ref{deletev0} holds.

We now show that 
\begin{sublemma}
\label{kge1}
$k \ge 1$.  
\end{sublemma}

Suppose $k=0$.  
We know by~\ref{sfcgotit} that $M\ba c_{0},c_1,c_{2},v_0$ is \sfc. 
We show next that 
\begin{sublemma}
\label{0120ifc} 
$M\ba c_0,c_1,c_2,v_0$ is not  \ifc.  
\end{sublemma}

Assume the contrary. By~\ref{deletev0}, $M\ba v_0$ is \sfc.  
Since (i) does not hold, $M\ba v_0$ has a $4$-fan, say $(\al,\be,\ga,\de)$. 
Then $\{\be,\ga,\de,v_0\}$ is a cocircuit so \ort\ implies that $\{\be,\ga,\de\}$ meets $\{t_0,u_0\}$.  Observe that $c_0\notin \{\be,\ga,\de\}$,   otherwise we get a \cn\ to \ort\ with one of the circuits $\{c_0,b_1,b_2\}$ or  $\{c_0,a_1,d_1,c_2\}$. 

Suppose  $u_0\in\{\be,\ga,\de\}$. Then, since $c_0 \not \in \{\be,\ga,\de\}$,  \ort\ implies that  $b_0\in\{\be,\ga,\de\}$. Thus we may relabel $\{\be,\ga,\de,v_0\}$ as $\{b_0,u_0,v_0,w_0\}$, and $M$ contains the  configuration   in Figure~\ref{bonesaws0}(a).  
Orthogonality between the cocircuit $\{b_0,u_0,v_0,w_0\}$, and the circuits in Figure~\ref{bonesaw0} implies that the elements are distinct except that $d_2$ may be $w_0$, so (v) holds;  a \cn.  Hence $u_0\not \in\{\be,\ga,\de\}$.

Since  $\{t_0,u_0\}$ meets $\{\be,\ga,\de\}$, we  may now assume that $t_0\in\{\be,\ga,\de\}$.  
Then~\ref{notbt} implies that $t_0\neq \de$. Hence, by symmetry, we may assume that $t_0 = \gamma$.  
By \ort\ between $\{\al,\be,t_0\}$ and the vertex cocircuits in Figure~\ref{bonesaw0}, we deduce that $\{\al,\be\}$ is 
$\{c_0,a_1\},\{b_2,c_1\},\{c_2,d_1\}$, or $\{c_2,d_2\}$.  
By \ort\ between the cocircuit $\{\be,t_0,\de,v_0\}$ and the triangles in Figure~\ref{bonesaw0} other than $\{t_0,v_0,u_0\}$, we deduce that if $\{\be,\de\}$ meets one of these triangles, then 
$\{\be,\de\}$ is contained in that triangle. But $\{\beta, \delta\}$ avoids $\{c_0,u_0\}$ and cannot meet $\{a_2,b_2,b_1,c_1\}$. 
We conclude that $\{\be,\de\}$
avoids $T_0\cup T_1\cup T_2\cup d_1$. Thus $(\al,\be) =(c_2,d_2)$ and 
$\de \not\in T_0 \cup T_1 \cup T_2 \cup \{d_1,d_2,t_0,v_0\}$.  
Then $\lambda (T_0\cup T_1\cup T_2\cup \{d_1,d_2,t_0\})\leq 2$.  
Now $T_0\cup T_1\cup T_2\cup \{d_1,d_2,t_0\}$ contains twelve elements, and avoids the two-element set $\{\de ,v_0\}$.  Thus, as $M$ is \ifc, we deduce that $|E(M)| \in \{14,15\}$ and 
$\lambda (T_0\cup T_1\cup T_2\cup \{d_1,d_2,t_0\})=2$. Hence  $r(T_0\cup T_1\cup T_2\cup \{d_1,d_2,t_0\})=6$.  Thus $r(M) = 6$ and $|E(M)| = 14$ otherwise $M$ has a triad containing $\{\de,v_0\}$, which gives the \cn\ that $M$ has a $4$-fan since $v_0$ is in a triangle.  
It follows that $\{b_2,u_0,b_0,a_1,d_1,d_2\}$ is a basis $B$ of $M$.  Now consider the fundamental circuits $C(v_0 ,B)$ and $C(\de,B)$. 
By using \ort\ between these circuits and the   cocircuit $\{v_0,d_2,t_0,\delta\}$ as well as the vertex cocircuits in Figure~\ref{bonesaw0}, we deduce that  $C(v_0 ,B)$ and $C(\de,B)$ are $\{v_0,a_1,b_0,d_1,d_2\}$ and $\{\de ,a_1,b_0,d_1,d_2\}$.  
The symmetric difference of $C(v_0 ,B)$ and $C(\de,B)$ is   $\{v_0,\de\}$, which must contain a circuit of $M$; a \cn. We conclude that  \ref{0120ifc} holds.

By \ref{sfcgotit} and \ref{0120ifc}, we deduce that  $M\ba c_{0},c_1,c_{2},v_0$ has a $4$-fan 
$(y_1,y_2,y_3,y_4)$.  
Suppose $\{y_1,y_2,y_3\}$ meets an element in Figure~\ref{bonesaw0} where $k=0$.  
Orthogonality and the fact that $M$ contains no parallel pairs implies that either $\{y_1,y_2,y_3\}$ contains $\{d_1,d_{2}\}$, and (ii) holds; or $\{y_1,y_2,y_3\}$ contains $\{b_{0},a_1\}$, and (iii) holds.  
Since neither (ii) nor (iii) holds by assumption, we deduce that $\{y_1,y_2,y_3\}$ avoids the elements in Figure~\ref{bonesaw0}.  
Now $M$ has a cocircuit $C^*$ such that $\{y_2,y_3,y_4\}\subsetneqq C^*\subseteq \{y_2,y_3,y_4,c_{0},c_1,c_{2},v_0\}$.  
Orthogonality implies that $C^*$ avoids $\{c_{0},c_1\}$, since each of these elements is in two triangles in the figure but $y_4$ is the only element of $C^*$ that can be in such a triangle.  
If $c_{2}\in C^*$, then $\{y_2,y_3,y_4\}$ meets $\{a_2,b_2\}$ and $\{a_1,d_1,c_1\}$; a \cn.  
Thus $C^*=\{v_0,y_2,y_3,y_4\}$ and $y_4\in\{t_0,u_0\}$.  
Hence $(\{t_0,u_0,v_0\},\{y_1,y_2,y_3\},\{x,v_0,y_2,y_3\})$ is a bowtie in $M$ for some $x$ in $\{t_0,u_0\}$; a \cn\ to~\ref{notbt}.  
We conclude that \ref{kge1} holds.  

\begin{figure}[htb]
\center
\includegraphics[scale=0.8]{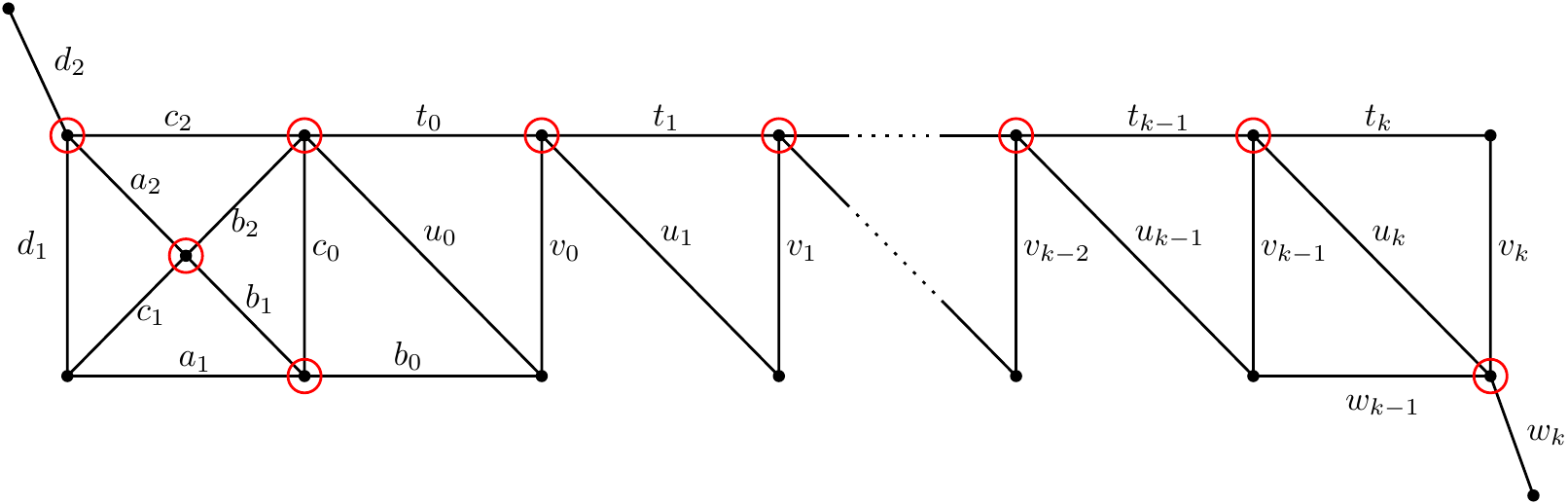}
\caption{$k\geq 1$}
\label{bonesawcap0}
\end{figure}

Next we show that 
\begin{sublemma}
\label{bonesawtime}
after possibly interchanging the labels on $t_k$ and $u_k$, the matroid $M$ contains the configuration in  Figure~\ref{bonesawcap0}, where $k \ge 1$.
\end{sublemma}

Consider $M\ba v_k$.  It certainly has an $N$-minor so it is not \ifc. Thus, 
by~\ref{notbt}, Lemma~\ref{6.3rsv} implies that $M\ba v_k$ is \ffsc\ and that either 
$M\ba v_k$ has a $4$-fan $(e,f,w_{k-1},w_k)$ where $\{w_{k-1},w_k\}$ avoids $\{u_{k-1},v_{k-1},t_{k-1},u_k,v_k,t_k\}$, while $e \in \{t_{k-1},v_{k-1}\}$ and $f \in \{t_k,u_k\}$; or $M$ has a triangle $\{u_{k-1},p,q\}$ and a cocircuit $\{a,v_k,p,q\}$ where $a \in \{u_k,t_k\}$, and $\{p,q\}$ avoids $\{u_{k-1},v_{k-1},t_{k-1},u_k,v_k,t_k\}$.  
 The latter implies that $(\{t_k,u_k,v_k\},\{u_{k-1},p,q\},\{a,p,q,v_k\})$ is a bowtie, a \cn\ to~\ref{notbt}.  
Thus we may assume that the former holds. 
Suppose $e = t_{k-1}$. Then \ort\ between the triangle  $\{w_{k-1},t_{k-1},f\}$ and the vertex cocircuit containing $\{t_{k-1}, u_{k-1}\}$ implies that $w_{k-1}$ is in a triangle $T'$ in Figure~\ref{bonesaw0} that is disjoint from $\{u_{k-1},v_{k-1},t_{k-1},u_k,v_k,t_k\}$. By \ort\ between $T'$ and the cocircuit $\{v_k,f,w_{k-1},w_k\}$ of $M$, we deduce that $T'$ also contains $w_k$, so we get a \cn\ to \ref{notbt}. 
We conclude that $e \neq t_{k-1}$. 
Thus $e=v_{k-1}$, so $M$ contains the configuration shown in Figure~\ref{bonesawcap0}, that is, \ref{bonesawtime} holds.   

As the elements in Figure~\ref{bonesaw0} are all distinct, the elements in Figure~\ref{bonesawcap0} are also distinct unless  $w_{k-1}$ or $w_k$ is equal to another element. Suppose $w_{k-1}$ is an element in Figure~\ref{bonesaw0}.  
As every element in that figure except $v_k$ is in a vertex cocircuit, $w_{k-1}$ is in such a cocircuit.  By \ort, $\{v_{k-1},u_k\}$ contains an element in this cocircuit. Hence $w_{k-1}\in\{t_{k-1},v_{k-1},u_{k},t_k\}$; a \cn.  
Clearly $w_k\neq w_{k-1}$.  
Suppose $w_k$ is an element in Figure~\ref{bonesaw0}.  
Then \ort\ implies that $w_k$ is not in a circuit in Figure~\ref{bonesaw0}, so $w_k=d_{2}$.  
We conclude that the following holds.
\begin{sublemma}
\label{bonesawcapdistinct}
The elements in Figure~\ref{bonesawcap0} are distinct except  that $w_k$ may equal $d_{2}$.  
\end{sublemma}

Next we observe that 

\begin{sublemma}
\label{newsky}
$M$ has no triangle containing $\{u_0,b_2\}$ or $\{u_0,c_2\}$.  
\end{sublemma}

Suppose $M$ has a triangle $T'$ containing $\{u_0,b_2\}$. Then, by \ort\ with the cocircuit $\{b_2,b_1,c_1,a_2\}$, we deduce that $T' = \{u_0,b_2,c_1\}$. Thus $\{u_0,b_0\} \subseteq \cl(T_1 \cup T_2)$, so $\lambda(T_0 \cup T_1 \cup T_2) \le 2$; a \cn. Hence $M$ has no triangle containing $\{u_0,b_2\}$.

Now assume $M$ has a triangle $T''$ containing $\{u_0,c_2\}$. Then, by \ort\ with the cocircuit $\{c_2,a_2,d_1,d_2\}$, we deduce that $T''$ is $\{u_0,c_2,d_1\}$ or   $\{u_0,c_2,d_2\}$. In the first case, we again get the \cn\ that $\{u_0,b_0\} \subseteq \cl(T_1 \cup T_2)$. In the second case, $\lambda(T_0 \cup T_1 \cup T_2 \cup \{d_1,d_2\}) \le 2$. This is  a \cn\ because $|E(M)| \ge 16$ since $k \ge 1$ and the elements in Figure~\ref{bonesaw0} are distinct. Thus \ref{newsky} holds.

The following will be useful not only in the proof of the subsequent assertion but also later in the proof of 
Lemma~\ref{killthesnake}.

\begin{sublemma}
\label{isffscz}
If  $(y_1,y_2,y_3,y_4)$ is a $4$-fan in $M\ba v_0$, then either $\{y_1,y_2,y_3\} = T_0$, or $y_4 \in \{t_0,u_0\}$. 
\end{sublemma}

Assume that this fails. As $\{y_2,y_3,y_4,v_0\}$ is a cocircuit, $t_0$ or $u_0$ is in $\{y_2,y_3,y_4\}$.  
By \ort\ between $\{y_1,y_2,y_3\}$ and 
the cocircuits displayed in Figure~\ref{bonesawcap0}, we see that $t_0 \not\in \{y_2,y_3\}$. Thus $u_0 \in \{y_2,y_3\}$. Then \ort\ between $\{y_1,y_2,y_3\}$ and the   cocircuits in Figure~\ref{bonesawcap0} implies that $\{y_1,y_2,y_3\}$  contains $\{u_0,b_2\}$ or $\{u_0,c_2\}$; a \cn\ to \ref{newsky}. We conclude that \ref{isffscz} holds.


We will now show that 
\begin{sublemma}
\label{isffsc}
$M\ba v_0$ is \ffsc.  
\end{sublemma}

By~\ref{deletev0}, $M\ba v_0$ is \sfc.  
Assume that \ref{isffsc} fails. Then $M\ba v_0$ has $(x_1,x_2,x_3,x_4,x_5)$ as a $5$-fan or a $5$-cofan.  Now $\{x_1,x_2,x_3,x_4,x_5\}$ contains at most one element of $\{t_0,u_0\}$ otherwise $v_0$ is in the closure of $\{x_1,x_2,x_3,x_4,x_5\}$ and so 
$\{x_1,x_2,x_3,x_4,x_5,v_0\}$ is $3$-separating in $M$; a \cn. Suppose that $(x_1,x_2,x_3,x_4,x_5)$ is a $5$-fan. 
Then $(x_1,x_2,x_3,x_4)$ and $(x_5,x_4,x_3,x_2)$ are $4$-fans. By \ref{isffscz}, we may assume that 
$T_0 = \{x_1,x_2,x_3\}$ and $x_2 \in \{t_0,u_0\}$. 
As $T_0 = \{u_0,b_0,c_0\}$, we deduce that 
 $x_2 = u_0$,  so $\{x_1,x_3\}=\{b_0,c_0\}$. Now $M$ has $\{x_2,x_3,x_4,v_0\}$ as a cocircuit.  
 If $c_0=x_3$, then, by \ort, $x_4\in\{b_1,b_2\}$. Hence the cocircuit $\{u_0,c_0,x_4,v_0\}$ meets the circuit $\{b_1,b_2,a_1,d_1,c_2\}$ in a single element; a \cn.  
Thus $(x_1,x_3)=(c_0,b_0)$. 
By \ort\ between the triangle $\{b_0,x_4,x_5\}$ and the cocircuits $D_0$ and $D_1$, we deduce that this   triangle also contains $a_1$.   
This contradicts the assumption that (ii) does not hold.

We may now assume that $(x_1,x_2,x_3,x_4,x_5)$ is a $5$-cofan. Then the $4$-fans $(x_2,x_3,x_4,x_5)$  and $(x_4,x_3,x_2,x_1)$ imply, by   \ref{isffscz}, that $T_0 = \{x_2,x_3,x_4\}$. Since both $\{x_3,x_2,x_1,v_0\}$ and $\{x_3,x_4,x_5,v_0\}$ are cocircuits of $M$, it follows that 
 $\{c_{0},v_0\}$ is contained in one of these $4$-cocircuit of $M$. But this cocircuit must meet each of  $\{b_1,b_{2}\},\{a_1,d_1,c_{2}\}$, and $\{u_0,t_0\}$; a \cn.  
 We conclude that~\ref{isffsc} holds.

\begin{figure}[htb]
\center
\includegraphics[scale=0.72]{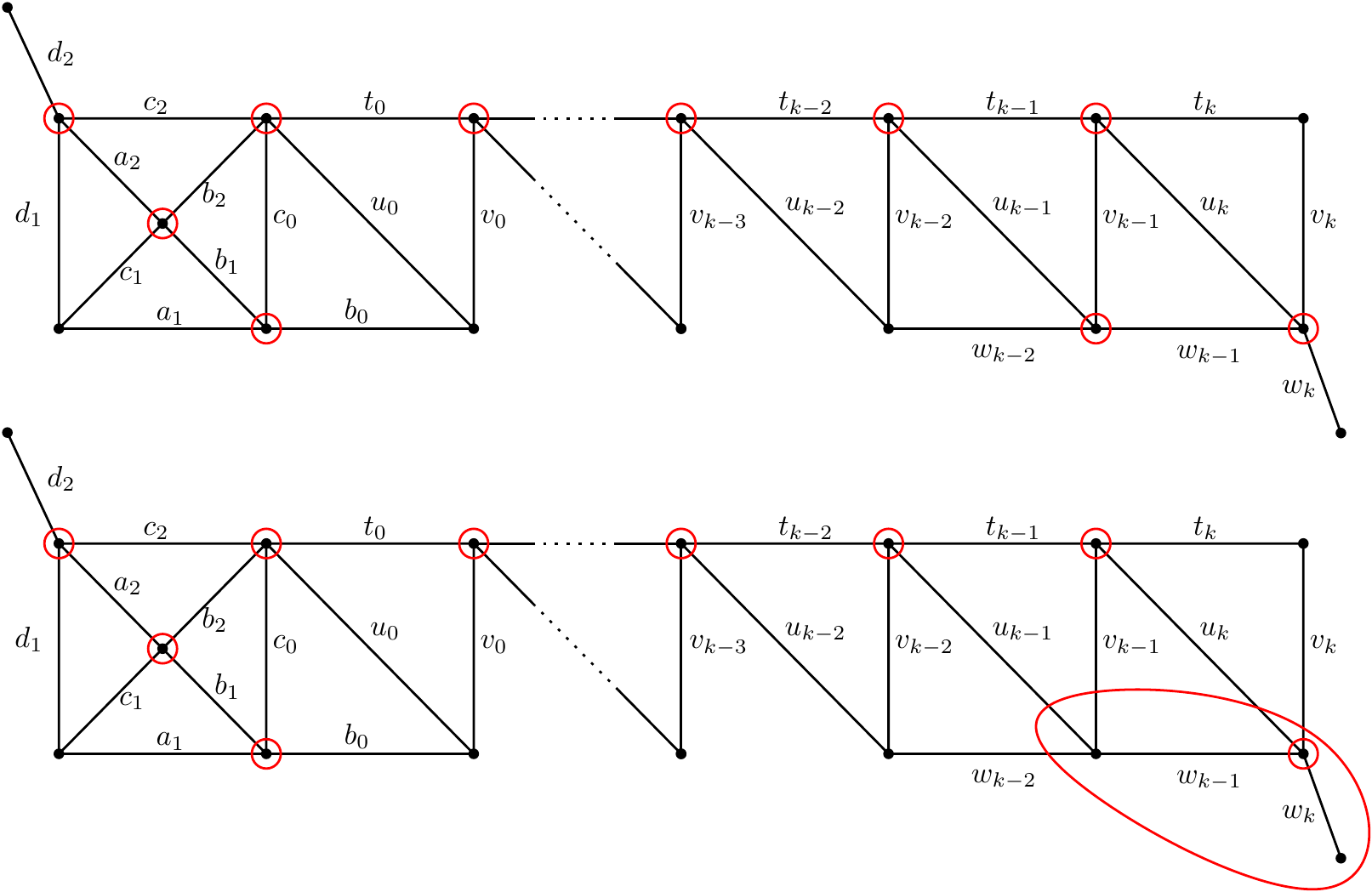}
\caption{$k\geq 2$}
\label{bonesawtails}
\end{figure}

Next we show that 
\begin{sublemma}
\label{nocoguts}
$M\ba v_{k-1},v_k$ has no $4$-fan having $t_k$ as its coguts element. 
\end{sublemma}

Assume that $M\ba v_{k-1},v_k$ has $(z_1,z_2,z_3,t_k)$ as a $4$-fan.     Then $M$ has one of $\{z_2,z_3,t_k,v_{k}\}, 
\{z_2,z_3,t_k,v_{k-1}\},$ or $\{z_2,z_3,t_k,v_{k-1},v_{k}\}$ as a cocircuit. The first possibility is excluded by \ref{notbt}. The second and third possibilities imply, by \ort, that $\{z_2,z_3\}$ meets both $\{t_{k-1},u_{k-1}\}$ and $\{u_k,w_{k-1}\}$. Then \ort\ using the triangle $\{z_1,z_2,z_3\}$ gives a \cn. Hence \ref{nocoguts} holds.

The next assertion will take some time to prove.
\begin{sublemma}
\label{tsintriangle}
$\{t_{k-1},t_k\}$ is contained in a triangle of $M$.  
\end{sublemma}

Suppose that $\{t_{k-1},t_k\}$ is not contained in a triangle of $M$.  
We can now apply Lemma~\ref{minidrossrsv} to the configuration induced by $\{t_{k-1},u_{k-1},v_{k-1},t_k,u_k,v_k,w_{k-1},w_k\}$ where $(v_{k-1},v_k)$ corresponds to $(c_0,c_1)$ in the lemma. Then neither (i) nor (iii) of  Lemma~\ref{minidrossrsv} holds. Moreover, by \ref{notbt}, (ii) of the lemma also does not hold. 

Next we eliminate the possibility that (v) of Lemma~\ref{minidrossrsv} holds by showing that $M\ba v_{k-1},v_k$ has no $4$-fan with $t_k$ as its coguts element. Assume instead that 
  $M\ba v_{k-1},v_k$   has such a $4$-fan, $(y_1,y_2,y_3,t_k)$.  
Then $M$ has a cocircuit $C^*$ such that $\{y_2,y_3,t_k\}\subsetneqq C^*\subseteq \{y_2,y_3,t_k,v_{k-1},v_k\}$.  
By~\ref{notbt}, we know that $v_{k-1}\in C^*$.  
Furthermore, Lemma~\ref{bowwow} implies that $v_k\in C^*$ unless $\{y_2,y_3,t_k,v_{k-1}\} = \{t_{k-1},u_k,t_k,v_{k-1}\}$. In the exceptional case, by \ort, $\{y_1,y_2,y_3\} = \{t_{k-1},u_k,w_k\}$, so $\lambda(\{t_{k-1},u_{k-1},v_{k-1},t_k,u_k,v_k,w_{k-1},w_k\}) \le 2$; a \cn. We deduce that 
 $C^*=\{y_2,y_3,t_k,v_{k-1},v_k\}$.  As $M$ is binary and $C^*$ contains three elements of the circuit $\{t_k,v_{k-1},v_k,w_{k-1}\}$, it follows that $w_{k-1} \in C^*$.
Orthogonality  between the triangle $\{y_1,y_2,y_3\}$ and the cocircuit $\{w_{k-1},u_k,v_k,w_k\}$ implies that $\{y_1,y_2,y_3\}$ meets $\{u_k,w_k\}$.  
As $\{v_{k-1},w_{k-1},u_{k}\}$ is a triangle, we deduce that $\{y_1,y_2,y_3\}$ does not contain $\{w_{k-1},u_k\}$, so $\{w_{k-1},w_k\}\subseteq \{y_1,y_2,y_3\}$; a \cn\ to~\ref{notbt}.  We conclude that $M\ba v_{k-1},v_k$ has no $4$-fan with $t_k$ as its coguts element, so (v) of Lemma~\ref{minidrossrsv} does not hold.

We now know that (iv) of Lemma~\ref{minidrossrsv} holds. Combining this with \ref{nocoguts}, we see that 

\begin{sublemma}
\label{fourex}
$M$ has elements  $\al$ and $\be$ not in 
$\{t_{k-1},u_{k-1},v_{k-1},t_k,u_k,v_k,w_{k-1},w_k\}$ such that   $\{\alpha ,\be, u_{k-1}\}$ is a triangle,  $\{\be ,u_{k-1},v_{k-1},w_{k-1}\}$ or $\{\be ,u_{k-1},v_{k-1},u_k,v_k\}$ is a 
cocircuit, $D^*$, and every \ftv\ of $M\ba v_{k-1},v_k$ is a $4$-fan with $\al$ as its guts. 
\end{sublemma}

We show next that 

\begin{sublemma}
\label{kge2}
$k \ge 2$.
\end{sublemma}

Suppose  that $k = 1$. By \ort\ between $\{\al,\be,u_0\}$ and the cocircuit 
$\{c_0,b_2,c_2,t_0,u_0\}$, we deduce that $\{\al,\be\}$ meets $\{c_0,b_2,c_2\}$. But, by \ref{newsky}, $\{\al,\be\}$ avoids $\{b_2,c_2\}$. Thus $\{\al,\be,u_0\} = \{c_0,b_0,u_0\}$. Orthogonality between $D^*$ and the triangle $\{b_2,b_1,c_0\}$ implies that $\beta \neq c_0$. Thus $(\al,\be) = (c_0,b_0)$. 

To complete the proof of \ref{kge2}, we will show that (ii) holds thereby obtaining  a \cn.  
By \ref{hasN} and \ref{sfcgotit}, 
 $M\ba c_0,c_1,c_2,v_0,v_1$ has an $N$-minor and is \sfc.  
As  (v) does not hold, $M\ba c_0,c_1,c_2,v_0,v_1$ is not \ifc, so it has  a $4$-fan 
$(z_1,z_2,z_3,z_4)$. Thus 
 $M$ has a cocircuit $D$ such that $\{z_2,z_3,z_4\}\subsetneqq D\subseteq \{z_2,z_3,z_4,c_0,c_1,c_2,v_0,v_1\}$.  
   
Suppose $(z_1,z_2,z_3,z_4)$ is a $4$-fan of $M\ba v_0,v_1$. Then, by \ref{fourex}, since $(\al,\be) = (c_0,b_0)$, we get the \cn\  that $z_1 = c_0$.   Hence $(z_1,z_2,z_3,z_4)$ is not a $4$-fan of $M\ba v_0,v_1$, so 
  $D$ meets $\{c_0,c_1,c_2\}$.  
Then \ort\ between $D$ and $T_0,T_1$, and $T_2$ implies that $\{z_2,z_3,z_4\}$ meets one of these triangles, specifically it meets $\{u_0,b_0,a_1,b_1,a_2,b_2\}$.  
Suppose $\{z_2,z_3\}$ meets the last set. Then \ort\ with the cocircuits shown in 
Figure~\ref{bonesawcap0} implies that $\{z_1,z_2,z_3\}$ is $\{b_0,w_0,a_1\}$, a \cn\ to \ort\ with $\{w_0,u_1,v_1,w_1\}$.  
Thus $z_4\in\{u_0,b_0,a_1,b_1,a_2,b_2\}$. But \ort\ with the circuits $\{u_0,b_0,b_1,b_2\}$ and $\{a_1,b_1,d_1,a_2\}$ in $M\ba c_0,c_1,c_2,v_0,v_1$ implies that $\{z_2,z_3\}$ also meets one of these circuits.  As $\{z_2,z_3\}$ avoids $\{u_0,b_0,a_1,b_1,a_2,b_2\}$, we deduce that 
$d_1\in\{z_2,z_3\}$.  
Since $\{z_1,z_2,z_3\}\neq \{c_1,d_1,a_2\}$, \ort\ between $\{z_1,z_2,z_3\}$ and the cocircuit $\{d_1,a_2,d_2\}$ in $M\ba c_0,c_1,c_2,v_0,v_1$ implies that $d_2\in\{z_1,z_2,z_3\}$. Hence (ii) holds. This \cn\ completes the proof of \ref{kge2}.  

Now we show that 
\begin{sublemma}
\label{hasatail}
$M$ contains one of the configurations shown in Figure~\ref{bonesawtails} where  the elements in each part are distinct except that $d_2$ may equal $w_k$.   
\end{sublemma}

As $M$ has $\{\al,\be,u_{k-1}\}$ as a triangle, by \ort, $\{\al,\be\}$ meets $\{t_{k-2},v_{k-2}\}$ in a single element. Orthogonality between $D^*$ and the triangle $\{t_{k-2},u_{k-2},v_{k-2}\}$ implies that $\be$ avoids this triangle. Thus   $\al \in \{t_{k-2},v_{k-2}\}$.

Suppose  $\al = t_{k-2}$. We know that $M$ has a $4$- or $5$-cocircuit containing $\{t_{k-2},u_{k-2}\}$. Thus $\be$ also meets this cocircuit, so $\be$ is in a triangle that violates \ort\ with $D^*$. Hence 
 $\al = v_{k-2}$. Relabelling $\be$ as $w_{k-2}$, we get that 
 $M$ contains one of the configurations in Figure~\ref{bonesaws0}. Moreover, by \ort, $\be$ is not an existing element of Figure~\ref{bonesawcap0}. Thus \ref{hasatail} holds. 
 
Next we show that 

\begin{sublemma} 
\label{bonesawdistinct}
$M$ contains one of the configurations shown in Figure~\ref{bonesaws0} where all the elements in the figure are distinct except that $d_2$ may equal $w_k$.
\end{sublemma}

To prove this, 
 we shall apply Lemma~\ref{dross} to the configuration in $M$ induced by $\{t_{k-2},u_{k-2},v_{k-2},t_{k-1},u_{k-1},v_{k-1},t_k,u_k,v_k,w_{k-2},w_{k-1},w_k\}$. 
By \ref{notbt}, $M$ is not a quartic M\"{o}bius ladder and $\{w_{k-1},w_k\}$ is not contained in a triangle of $M$, so neither (iv) nor (ii) of Lemma~\ref{dross} holds. Also (i) of that lemma does not hold as $N \preceq M\ba v_{k-2},v_{k-1},v_k$, and $M$ has no ladder win. 
Thus (iii) of   Lemma~\ref{dross}  holds.  
But, by \ref{nocoguts}, $M\ba v_k$ does not have a $4$-fan that avoids $v_{k-1}$ but has $t_k$ as its coguts element.  Thus 
$M\ba v_{k-2},v_{k-1},v_k$ has a $4$-fan $(y_1,y_2,u_{k-2},w_{k-2})$ that is also a $4$-fan of $M\ba v_{k-2}$. Then $\{y_2,u_{k-2},w_{k-2},v_{k-2}\}$ is a cocircuit of $M$.  
Suppose first that $k = 2$. Then \ort\ between the last cocircuit and the circuit $\{u_0,b_0,c_0\}$ implies that $y_2 \in \{c_0,b_0\}$. Moreover, \ort\ between the specified cocircuit and the circuit $\{c_0,b_1,b_2\}$ implies that $y_2 \neq c_0$, so $y_2 = b_0$. Thus, when $k = 2$, we conclude using \ref{hasatail} that \ref{bonesawdistinct} holds.

Now suppose that $k \ge 3$. By \ort\ between the cocircuit $\{y_2,u_{k-2},w_{k-2},v_{k-2}\}$ and the triangle $\{u_{k-3},v_{k-3},t_{k-3}\}$, we deduce that $y_2 \not\in \{v_{k-3},t_{k-3}\}$. 
 Moreover, by Lemma~\ref{trywhere}, $\{y_1,y_2\}$ avoids 
$\{t_{k-2},u_{k-2},\linebreak v_{k-2},t_{k-1},u_{k-1},v_{k-1}, t_k,u_k,v_k,w_{k-2}, w_{k-1},w_k\}$. By \ort\ between the triangle $\{y_1,y_2,u_{k-2}\}$ and the cocircuit $\{u_{k-2},t_{k-2},v_{k-3},t_{k-3}\}$, we deduce that $y_1 \in \{v_{k-3},t_{k-3}\}$. Suppose $y_1 = t_{k-3}$. 
Then the triangle $\{t_{k-3},y_2,u_{k-2}\}$ gives a \cn\ to \ort\ unless $k \ge 4$ and $y_2 = v_{k-4}$. In the exceptional case, we get a \cn\ to \ort\ between the triangle 
$\{t_{k-4},u_{k-4},v_{k-4}\}$ and the cocircuit $\{y_2,u_{k-2},w_{k-2},v_{k-2}\}$. 
 We deduce that $y_1 = v_{k-3}$. We now relabel $y_2$ as $w_{k-3}$ noting that, by \ort, 
it differs from the existing elements in Figure~\ref{bonesawtails}. 

We now  see that we have a configuration of the same form as the one to which we just applied Lemma~\ref{dross}, this new configuration being induced by 
$\{t_{i},u_{i},v_{i},w_{i}: k-3 \le i \le k\}$.  We apply Lemma~\ref{dross} again and repeat this process until we find that $M$ contains  one of the configurations shown in Figure~\ref{bonesaws0} where all of the elements $w_{k-3}, w_{k-4},\ldots,w_0$ added to Figure~\ref{bonesawtails} are distinct and differ from the existing elements in the figure. We conclude that \ref{bonesawdistinct} holds.

By~\ref{sfcgotit} and \ref{hasN}, we know that $M\ba c_{0},c_1,c_{2},v_0,v_1,\dots ,v_k$ is \sfc\ with an $N$-minor.  
If $M\ba c_{0},c_1,c_{2},v_0,v_1,\dots ,v_k$ is \ifc, then (v) of the lemma holds; a \cn. Thus we may assume that 
$M\ba c_{0},c_1,c_{2},v_0,v_1,\dots ,v_k$ has a $4$-fan $(y_1,y_2,y_3,y_4)$.

Suppose first that $(y_1,y_2,y_3,y_4)$ is a $4$-fan of $M\ba v_0,v_1,\dots ,v_k$. Then it follows by Lemma~\ref{dross} and symmetry, either $(y_3,y_4) = (u_0,w_0)$ and $(y_1,y_2,y_3,y_4)$ is a $4$-fan of $M\ba v_0$; or 
$y_4 = t_k$ and $(y_1,y_2,y_3,y_4)$ is a $4$-fan of $M\ba v_k$. In the latter case, $M$ has $\{y_2,y_3,t_k,v_k\}$ as a cocircuit, so $M$ has $(\{t_k,u_k,v_k\}, \{y_1,y_2,y_3\}, \{y_2,y_3,t_k,v_k\})$ as a bowtie, a \cn\ to \ref{notbt}. 
In the latter case, by \ref{isffscz}, $\{y_1,y_2,y_3\} = \{c_0,b_0,u_0\}$, so $(y_1,y_2,y_3,y_4)$ is not a $4$-fan of $M\ba c_0,c_1,c_2,v_0,v_1,\dots ,v_k$; a \cn. 
 We conclude that $(y_1,y_2,y_3,y_4)$ is not a $4$-fan of $M\ba v_0,v_1,\dots ,v_k$.
 
Next we observe that 

\begin{sublemma} 
\label{distinction}
$\{y_1,y_2,y_3\}$ avoids $\{b_0,a_1,b_1,a_2,b_2,d_1,u_0\}$. 
\end{sublemma}

To see this, we observe that $b_0 \not\in \{y_1,y_2,y_3\}$ otherwise, by \ort\ with the vertex cocircuits  in $M\ba c_{0},c_1,c_{2},v_0,v_1,\dots ,v_k$ shown in  Figure~\ref{bonesaws0},  we get the contradiction that $\{y_1,y_2,y_3\}$ has at least four elements. Next, \ort\ with the cocircuit $\{a_1,b_1,b_0\}$ implies that $\{a_1,b_1\}$ avoids $\{y_1,y_2,y_3\}$. Similarly, \ort\ with the cocircuit $\{b_1,b_2,a_2\}$ implies that $\{a_2,b_2\}$ avoids $\{y_1,y_2,y_3\}$. Moreover, as (ii) of the lemma does not hold, $d_1$ avoids $\{y_1,y_2,y_3\}$. Finally, if $u_0 \in \{y_1,y_2,y_3\}$, then \ort\ with the cocircuit $\{b_2,u_0,t_0\}$ implies that $\{y_1,y_2,y_3\} = \{u_0,t_0,v_0\}$; a \cn. Hence \ref{distinction} holds.

Now $M$ has a cocircuit $C^*$ such that $\{y_2,y_3,y_4\}\subsetneqq C^*\subseteq \{y_2,y_3,y_4,\linebreak c_{0},c_1,c_{2},v_0,v_1,\dots ,v_k\}$.  
As $(y_1,y_2,y_3,y_4)$ is not a $4$-fan in $M\ba v_0,v_1,\dots ,v_k$, it follows that $C^*$ meets $\{c_{0},c_1,c_{2}\}$.  
Suppose $c_{0}$ is in $C^*$. Then, by \ort, $\{y_2,y_3,y_4\}$ meets $\{b_{1},b_{2}\}$ and $\{u_0,b_0\}$, so $\{y_2,y_3\}$ meets $\{b_0,b_1,b_2,u_0\}$; a \cn\ to \ref{distinction}.  
Likewise, if $c_{1}$ is in $C^*$, then $\{y_2,y_3,y_4\}$ meets  $\{a_{1},b_{1}\}$ and $\{d_1,a_{2}\}$, so  $\{y_2,y_3\}$ meets $\{a_1,b_1,a_2,d_1\}$; 
a \cn.  
Thus $c_{2}\in C^*$ and $\{y_2,y_3,y_4\}$ meets $\{a_2,b_2\}$ and $\{u_0,b_0,a_1,d_1\}$. This final \cn\ to \ref{distinction} completes the proof of \ref{tsintriangle}.  

\begin{figure}[htb]
\center
\includegraphics[scale=0.8]{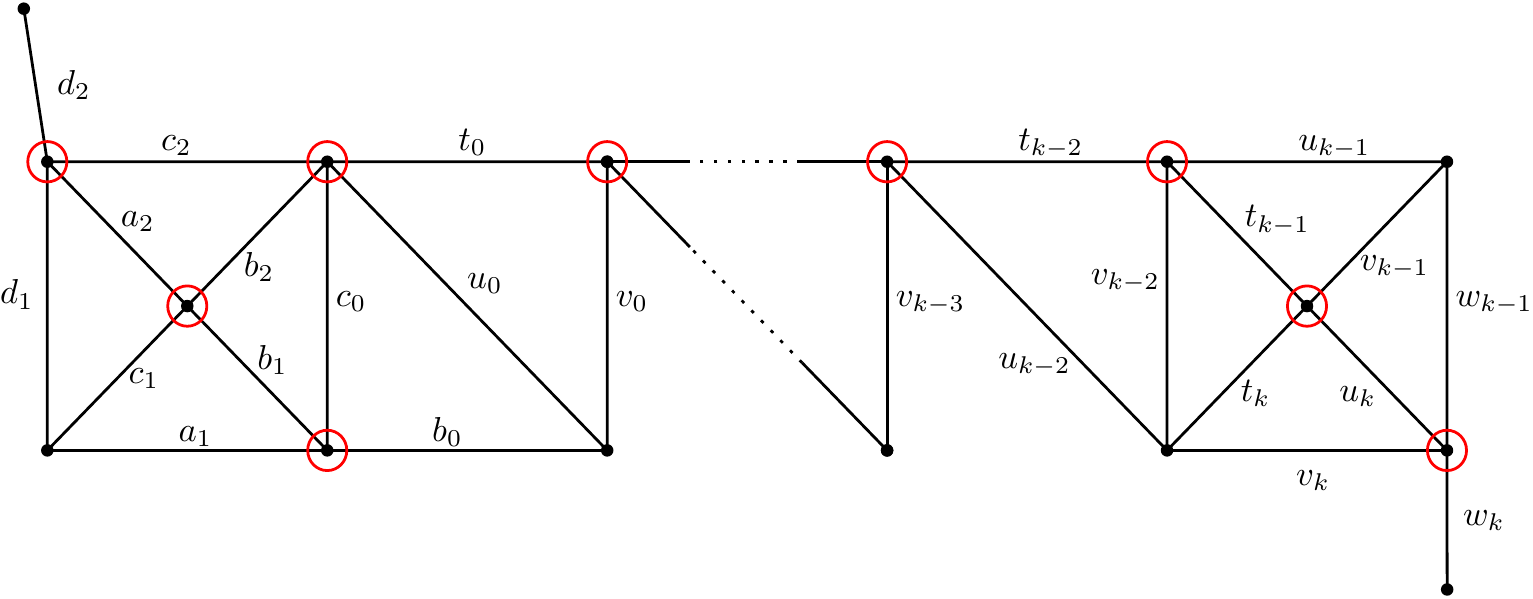}
\caption{$k\geq 2$}
\label{caterpillarfeet0}
\end{figure}


Although \ref{kge2} showed that $k \ge 2$, that proof was embedded in the proof-by-contradiction of \ref{tsintriangle}. Temporarily, all we know is that, by \ref{kge1}, $k \ge 1$. 
By~\ref{tsintriangle}, we may assume that we have the configuration in Figure~\ref{bonesawcap0} and that $\{t_{k-1},t_k\}$ is contained in a triangle of $M$.  
Then $M\ba v_{k-1}$ is not \ffsc, so~\ref{isffsc} implies that $k\geq 2$ and, by \ort, we have the configuration shown in Figure~\ref{caterpillarfeet0}.  
Now~\ref{bonesawcapdistinct} implies that the elements in Figure~\ref{caterpillarfeet0} are all distinct except that possibly $w_k = d_2$.  

\begin{figure}[htb]
\center
\includegraphics[scale=0.68]{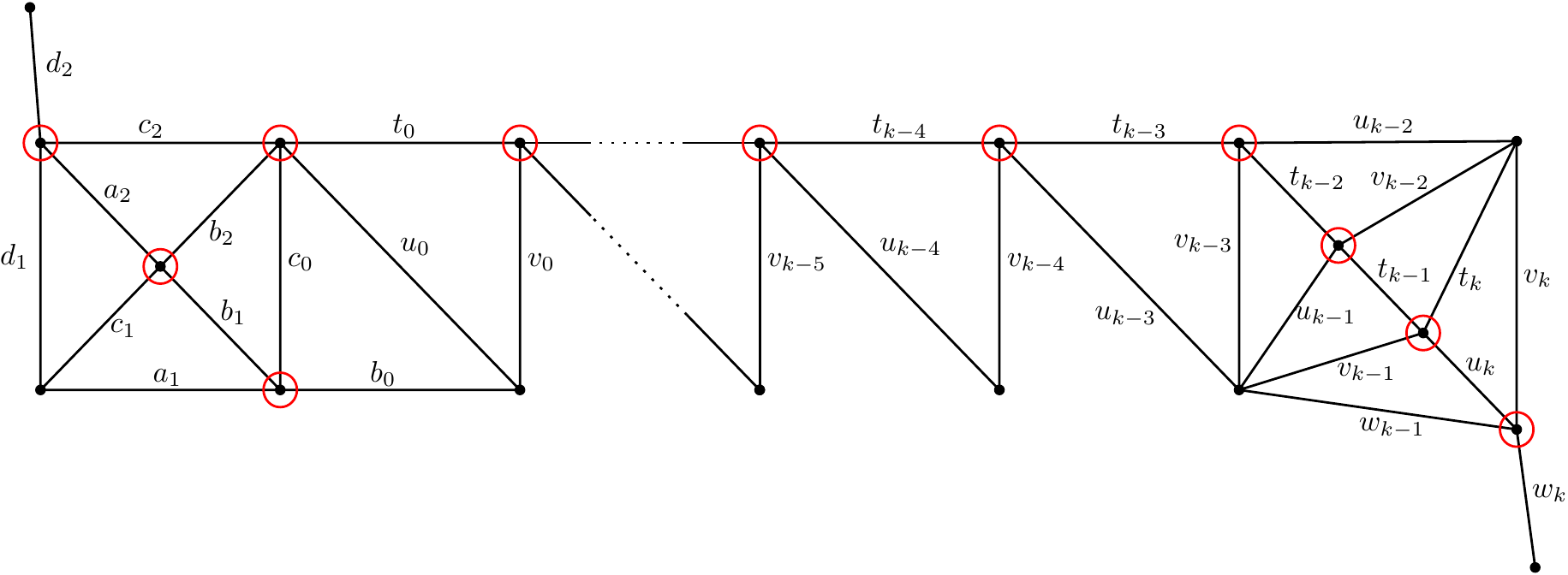}
\caption{The structure that arises if $\{t_{k-2},u_{k-1}\}$ is contained in a triangle.}
\label{caterpillartwirl0}
\end{figure}

Next we show that 

\begin{sublemma}
\label{not22}
$\{t_{k-2},u_{k-1}\}$ is not contained in a triangle of $M$.  
\end{sublemma}

Suppose that $\{t_{k-2},u_{k-1}\}$ is   contained in a triangle. 
Then  \ort\ implies that $k\geq 3$ and that 
this triangle contains $v_{k-3}$, so $M$ contains the configuration shown in Figure~\ref{caterpillartwirl0}.  
This configuration is contained in a right-maximal rotor chain of the form $((v_{k},u_{k},t_{k}),(v_{k-1},t_{k-1},u_{k-1}),(v_{k-2},t_{k-2},u_{k-2}),\dots ,(v_\ell,t_\ell,u_\ell))$, where $\ell \le k - 2$. Moreover, $\ell \ge 0$ since \ort\ implies that $\{t_0,u_1\}$ is not in a triangle of $M$.  
Lemma~\ref{deletechainsfc} implies that $M\ba v_\ell, v_{\ell +1},\dots v_k$ is \sfc.  
As $M\ba v_\ell, v_{\ell +1},\dots v_k$ has an $N$-minor, it is not \ifc\ otherwise $M$ has an  open-rotor-chain win; a \cn. 
Let $(y_1,y_2,y_3,y_4)$ be a $4$-fan in $M\ba v_\ell, v_{\ell +1},\dots v_k$. 

Suppose $\{y_1,y_2,y_3\}$ meets $\{t_\ell,u_{\ell+1},t_{\ell+1},\dots ,u_{k},t_{k},w_{k-1},w_{k}\}$.  
Every element in the last set is in a triad of $M\ba v_\ell,v_{\ell+1},\dots ,v_{k}$.  
Orthogonality implies that $\{y_1,y_2,y_3\}$ contains $\{w_{k-1},w_{k}\}$ or $\{t_\ell,u_{\ell+1}\}$.  
The former gives a \cn\ to~\ref{notbt}, so we assume the latter.  
Then $M\ba v_\ell$ is not \ffsc, so, by \ref{isffsc},  $\ell\geq 1$.  
Orthogonality implies that $\{y_1,y_2,y_3\}$ is $\{v_{\ell -1},t_\ell ,u_{\ell +1}\}$.  
By the maximality of the rotor chain, we know that $((v_{k},u_{k},t_k),(v_{k-1},t_{k-1},u_{k-1}),\dots,(v_\ell,t_\ell,u_\ell),(v_{\ell-1},t_{\ell-1},u_{\ell-1}))$ is not a rotor chain.  
Hence these elements are not distinct; a \cn\ to~\ref{bonesawcapdistinct}.  
We conclude that $\{y_1,y_2,y_3\}$ avoids $\{t_\ell,u_{\ell+1},t_{\ell+1},\dots ,u_{k},t_{k},w_{k-1},w_{k}\}$.  

Now $M$ has a cocircuit $C^*$ with  $\{y_2,y_3,y_4\}\subsetneqq C^*\subseteq \{y_2,y_3,y_4,v_\ell,v_{\ell+1},\dots ,v_{k}\}$.  
For each $i$ in $\{\ell + 1, \ell +2,\ldots,k-1\}$, the element $v_i$ is in two circuits whose other elements are contained in $\{u_{\ell+1},t_{\ell+1},\dots ,u_{k},t_{k},w_{k-1}\}$. Thus, by \ort, $v_i \not\in C^*$. Hence 
$C^* \subseteq  \{y_2,y_3,y_4,v_\ell,v_{k}\}$. 

Suppose $v_{\ell} \in C^*$. Then, by \ort, $u_{\ell} \in \{y_2,y_3\}$ and either $\ell \le k-3$ and $y_4 \in \{t_{\ell + 1},u_{\ell + 2}\}$, or $\ell = k-2$ and $y_4 \in \{t_{k-1},t_k\}$. But 
$y_4 \notin \{t_{\ell + 1},u_{\ell + 2}\}$ otherwise, by \ort, $u_{\ell+1}$ or $t_{\ell+2}$ is in 
$\{y_2,y_3\}$; a \cn. Thus $\ell = k-2$. Now 
$y_4 \neq t_{k-1}$ otherwise, by \ort, $t_k \in \{y_2,y_3\}$; a \cn. We conclude that $y_4 = t_k$, so, as $u_k \notin \{y_2,y_3\}$, we deduce that $v_k \in C^*$. Now, without loss of generality, $u_{k-2} = y_3$. The triangle $\{y_1,y_2,u_{k-2}\}$ meets the cocircuit 
$\{t_{k-3},v_{k-3},t_{k-2},u_{k-2}\}$ so $\{y_1,y_2\}$ meets $\{t_{k-3},v_{k-3},t_{k-2}\}$. But 
$t_{k-2} \not\in \{y_1,y_2\}$ and, using \ort, we see that $t_{k-3} \not\in \{y_1,y_2\}$. Thus $v_{k-3} \in \{y_1,y_2\}$. By \ort\ between $\{t_{k-3},u_{k-3},v_{k-3}\}$ and $C^*$, we deduce that $v_{k-3} \neq y_2$. Thus $v_{k-3} =y_1$. Now let $Z = \{y_2,v_{k-3},w_{k-1}\} \cup \{u_i,t_i,v_i: k-2 \le i \le k\}$. 
Then  one easily checks that $\lambda(Z) \le 2$ and $|E- Z| \ge 4$. This \cn\ implies that $v_{\ell} \not\in C^*$.

We now conclude that $C^*=\{y_2,y_3,y_4,v_{k}\}$.  Then, by \ort, $C^*$ meets $\{t_k,u_k\}$. But $\{y_1,y_2,y_3\}$ avoids the last set, so $y_4 \in \{t_k,u_k\}$. Then \ort\ implies that  $\{w_{k-1},t_{k-1}\}$ meets $\{y_2,y_3\}$; a \cn.  
We conclude that  \ref{not22} holds.

\begin{figure}[htb]
\center
\includegraphics[scale=0.7]{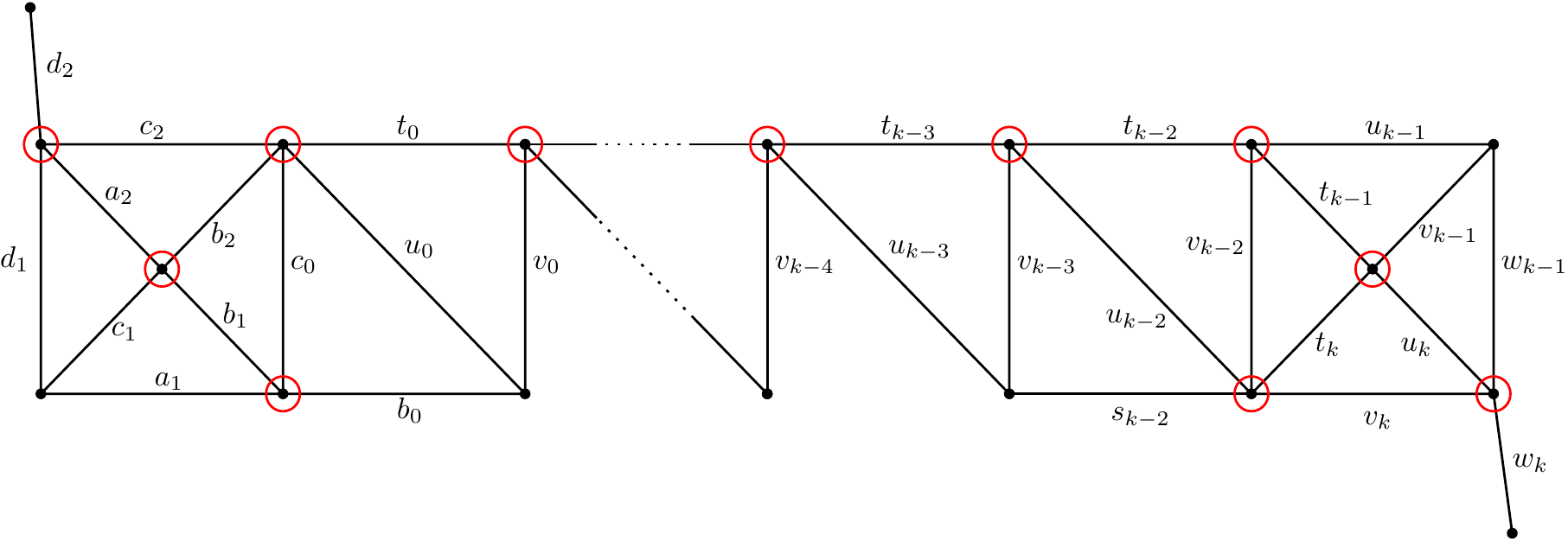}
\caption{$k\geq 2$}
\label{caterpillartail0}
\end{figure}

We now know that $M$ contains the configuration in Figure~\ref{caterpillarfeet0} but not that in Figure~\ref{caterpillartwirl0}.  We apply 
 Lemma~\ref{minidross2} to the structure in Figure~\ref{caterpillarfeet0} induced by 
 $\{t_{k-2},v_{k-2},t_{k-1},u_{k-1},v_{k-1},t_k,u_k,v_k,w_{k-1},w_k\}.$ 
Part (i) of that lemma does not hold  by assumption, and parts (ii) and (iv) do not hold by~\ref{notbt} and~\ref{not22}.  
If $M$ has a $4$-cocircuit  containing $\{v_{k-2},t_k,v_k\}$, then \ort\ implies that the fourth element of this cocircuit is in $\{t_{k-2},u_{k-2}\}$, so $\lambda (\{t_{k-2},u_{k-2},v_{k-2},t_{k-1},u_{k-1},v_{k-1},t_k,u_k,v_k\})\leq 2$; a \cn.  
We deduce that part (v) of Lemma~\ref{minidross2} does not hold. Thus part (iv) of that lemma holds, that is,  $M$ has a triangle $\{y_1,y_2,y_3\}$ such that $\{v_{k-2},t_k,v_k,y_2,y_3\}$ is a cocircuit.  
Orthogonality with the triangle $\{t_{k-2},u_{k-2},v_{k-2}\}$ implies that $\{y_2,y_3\}$ meets $\{t_{k-2},u_{k-2}\}$. But~\ref{not22} and orthogonality with the cocircuits containing $t_{k-2}$ imply that $u_{k-2}\in\{y_2,y_3\}$.  
Then \ort\ implies, without loss of generality, that $(y_1,y_2,y_3)$ is either $(c_0,b_0,u_0)$ when $k = 2$, or is     
$(v_{k-3},s_{k-2},u_{k-2})$  for some element $s_{k-2}$ when $k \ge 3$.  
We now have the configuration shown in Figure~\ref{caterpillartail0}.  

It may be that $M$ has an element $s_{k-3}$ such that $\{v_{k-4},u_{k-3},s_{k-3}\}$ is a triangle and $\{s_{k-3},u_{k-3},v_{k-3},s_{k-2}\}$ is a cocircuit.  
If there is no such triangle and cocircuit, then let $\ell= k-2$; otherwise 
let $\ell$ be the smallest positive integer such that 
$(\{u_{k-2},s_{k-2},v_{k-3}\}, \{s_{k-2},v_{k-3},u_{k-3},s_{k-3}\},\{u_{k-3},s_{k-3},v_{k-4}\},\ldots,\{u_{\ell},s_{\ell},v_{\ell -1}\})$ is a string of bowties. Hence $M$ contains  the structure shown in Figure~\ref{caterend}. 

Next we show the following.

\begin{sublemma}
\label{alldistinct}
The elements in Figure~\ref{bonesawcap0} and in $\{s_\ell ,s_{\ell +1},\dots ,s_{k-2}\}$ are all distinct, with the possible exception that $w_k$ may be the same as $d_{2}$.  
\end{sublemma}

By the definition of a string of bowties, the elements in $\{s_\ell,s_{\ell +1},\dots ,s_{k-2}\}$ are distinct. Moreover, as each of these elements is in a triangle, it follows by \ort\ and the fact that $M$ is binary that none of these elements is  in Figure~\ref{bonesawcap0}.
We conclude that \ref{alldistinct} holds. 

\begin{figure}[htb]
\center
\includegraphics[scale=0.8]{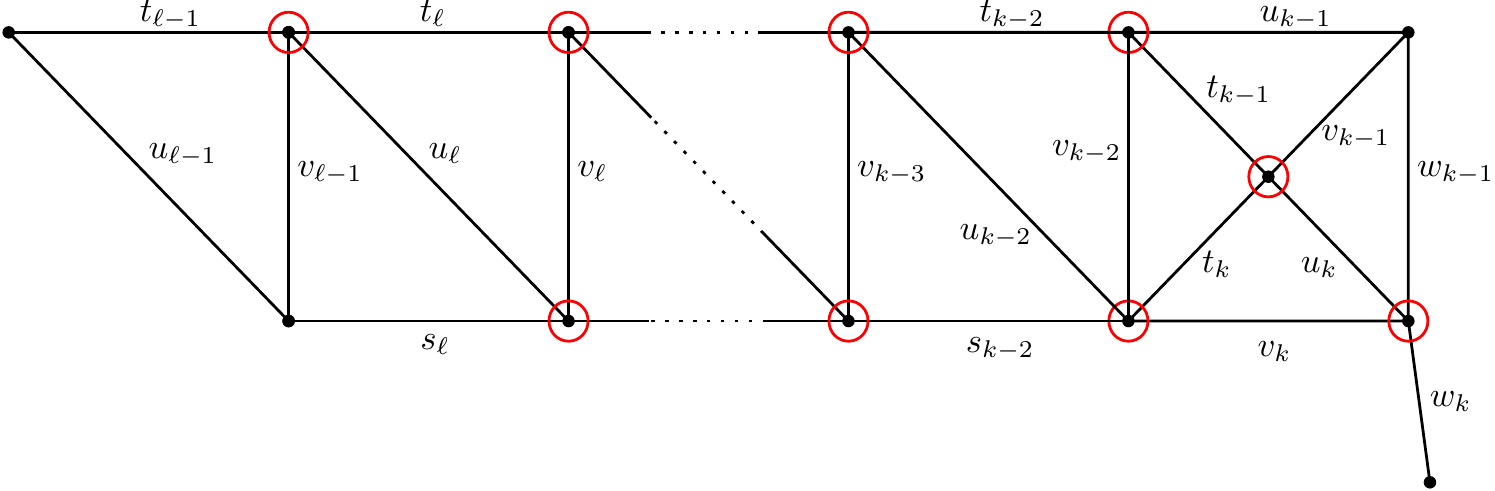}
\caption{$k\geq 2$ and $1\leq\ell\leq k-2$}
\label{caterend}
\end{figure}

We show next that 
\begin{sublemma}
\label{vssfc}
$M\ba v_{\ell -1} ,v_{\ell},\dots ,v_k$ is   \sfc.  
\end{sublemma}

Assume that this is false. 
By \ref{sfcgotit}, $M\ba c_{0},c_1,c_{2},v_0,v_1,\dots ,v_k$ is \sfc. Thus  $M\ba v_{j} ,v_{j+1},\dots ,v_k$ is \thc\ for all $j$ in $\{0,1,\ldots,k\}$.  
Lemma~\ref{deletechainsfc} implies that $M\ba v_{k-2},v_{k-1},v_k$ is \sfc.  
Let $i$ be the largest integer in $\{0,1,\ldots,k-2\}$ such that $M\ba v_i,v_{i+1},\dots ,v_k$ is not \sfc.  
Clearly $i\geq \ell -1$.  
Let $(X,Y)$ be a \ns\ \ths\ in $M\ba v_i ,v_{i +1},\dots ,v_k$.  
Suppose $i\geq 1$. Then, without loss of generality, the triangle $\{t_{i-1},u_{i-1},v_{i-1}\}\subseteq X$.  
If $t_i$ or $u_i$ is in $X$, then we may assume that both are in $X$, and then $(X\cup v_i,Y)$ is a \ns\ \ths\ of $M\ba v_{i+1},v_{i+2},\dots ,v_k$; a \cn.  
Thus $\{t_i,u_i\}\subseteq Y$, and $(X,Y\cup v_i)$ is a \ns\ \ths\ in $M\ba v_{i+1},v_{i+2},\dots ,v_k$; a \cn.  
We conclude that $i = 0$, so $\ell = 1$. 
  Without loss of generality, $T_2\subseteq X$.  Recall that, because of relabelling, $T_0$ is now $\{u_0,b_0,c_0\}$. 
If $T_0\subseteq X$, then $t_0\in \cl _{M\ba v_0,v_1,\dots ,v_k}^*(X)$, so $(X\cup t_0\cup v_0,Y-t_0)$ is a \ns\ \ths\ of $M\ba v_1,v_2,\dots ,v_k$; a \cn.  
Thus we may assume that $T_0\subseteq Y$.  
If $b_1\in X$, then we can assume that $c_1$ and $a_1$ are in $X$. Then we can move $c_0,s_0$, and $u_0$ into $X$; a \cn.  
Thus $b_1\in Y$, and we can move $b_2,a_1,c_1,a_2,c_2$, and $t_0$ into $Y$ and add $v_0$ to again get the \cn\ that  $M\ba v_1,v_2,\dots ,v_k$ has a \ns\ \ths. We conclude that \ref{vssfc} holds.

Now suppose $M\ba v_{\ell -1} ,v_{\ell},\dots ,v_k$ is \ifc. 
Consider the structure in Figure~\ref{caterend}.  
By removing $u_{\ell -1}$, we see that the structure is, after a rotation, the same as that in Figure~\ref{bonesaws0}(a), and deleting $\{v_{\ell -1} ,v_{\ell},\dots ,v_k\}$ is an enhanced ladder win. Hence (v) holds; a \cn.  

We  now know that $M\ba v_{\ell -1} ,v_{\ell},\dots ,v_k$ is not \ifc, so it has a  $4$-fan $(y_1,y_2,y_3,y_4)$. 
Using \ort\   together with~\ref{notbt} and~\ref{not22}, we get that  $\{y_1,y_2,y_3\}$ avoids the elements in Figure~\ref{caterend} with the possible exception of $u_{\ell -1}$. The matroid 
$M$ has a cocircuit $C^*$ such that $\{y_2,y_3,y_4\}\subsetneqq C^*\subseteq \{y_2,y_3,y_4,v_{\ell - 1},v_\ell ,v_{\ell +1},\dots ,v_k\}$. 
 We show next that 

\begin{sublemma}
\label{234ell} $C^* = \{y_2,y_3,y_4,v_{\ell -1}\}$ and $y_4 = s_{\ell}$, while $u_{\ell -1}\in \{y_2,y_3\}$. 
\end{sublemma}

For each $i$ in $\{\ell,\ell +1,\ldots,k-1\}$, the element $v_i$ is in  
 two triangles in Figure~\ref{caterend}, neither of which meets $u_{\ell -1}$ or any other $v_j$. Thus if $v_i \in C^*$, 
 then $\{y_2,y_3,y_4,v_i\}$ meets each of these triangles in two elements. Hence $\{y_2,y_3\}$ contains an element in Figure~\ref{caterend} other than  $u_{\ell -1}$; a \cn.  
Moreover, $v_k \not \in C^*$ otherwise  
  \ort\ implies that $\{y_2,y_3,y_4\}$ meets both $\{t_k,u_k\}$ and $\{u_{k-2},t_{k-2},u_{k-1},w_{k-1}\}$; a \cn.  We conclude that 
$C^* = \{y_2,y_3,y_4,v_{\ell -1}\}$. Then \ort\ with $\{v_{\ell -1},s_\ell,u_\ell \}$ implies that $y_4\in\{s_\ell ,u_\ell\}$ since $\{y_2,y_3\}$ avoids $\{s_\ell ,u_\ell\}$. As $\{y_2,y_3\}$ also avoids $\{t_\ell ,u_\ell ,v_\ell\}$, we deduce that $y_4=s_\ell$.  
Orthogonality between $C^*$ and $\{t_{\ell -1},u_{\ell -1},v_{\ell -1}\}$ implies that $u_{\ell -1}\in \{y_2,y_3\}$. Thus 
\ref{234ell} holds. 

Without loss of generality, $(y_1,y_2,y_3,y_4)=(y_1,y_2,u_{\ell -1},s_\ell)$.
If $\ell >1$, then \ort\ between $\{y_1,y_2,y_3\}$ and   the vertex cocircuits in Figure~\ref{bonesawcap0} implies that $y_1=v_{\ell -2}$.  
This means that we can extend the string of bowties 
$(\{u_{k-2},s_{k-2},v_{k-3}\}, \{s_{k-2},v_{k-3},u_{k-3},s_{k-3}\},\{u_{k-3},s_{k-3},v_{k-4}\},\ldots,\{u_{\ell},s_{\ell},v_{\ell -1}\})$, which  
contradicts our choice of $\ell$. Hence $\ell =1$ and \ort\ implies that $(y_1,y_2,u_0,s_1)=(c_{0},y_2,u_0,s_1)$. Thus we have the structure shown in Figure~\ref{caterpillarwhole0}, where $y_2$ is $s_0$, which is equal to $b_0$ relabelled.  

We will complete the proof of Lemma~\ref{killthesnake} by showing that $M\ba c_{0},c_1,c_{2},v_0,\dots ,v_k$ is \ifc. Since, by \ref{hasN}, the last matroid has an $N$-minor, this will establish that part  (v) of the lemma holds; a \cn.  
By \ref{alldistinct}, all of the elements in Figure~\ref{caterpillarwhole0} are distinct with the possible exception that $d_{2}$ may be $w_k$.  
We know, by~\ref{sfcgotit}, that $M\ba c_{0},c_1,c_{2},v_0,\dots ,v_k$ is \sfc.  
Let $(y_1,y_2,y_3,y_4)$ be a $4$-fan in this matroid.  
Suppose $\{y_1,y_2,y_3\}$ meets the elements in Figure~\ref{caterpillarwhole0}.  
The vertex cocircuits imply that $\{y_1,y_2,y_3\}$ contains $\{d_1,d_{2}\}$ or $\{w_{k-1},w_k\}$.  
If the former occurs, then part (ii) of the lemma holds; a \cn.  
The latter gives a \cn\ to~\ref{notbt}.  
Thus $\{y_1,y_2,y_3\}$ avoids the elements in Figure~\ref{caterpillarwhole0}.  
Now $M$ has a cocircuit $C^*$ such that $\{y_2,y_3,y_4\}\subsetneqq C^*\subseteq \{y_2,y_3,y_4,c_{0},c_1,c_{2},v_0,\dots ,v_k\}$.  
Each element in $\{c_{0},c_1,v_0,v_1,\dots ,v_{k-1}\}$ is in two triangles in Figure~\ref{caterpillarwhole0}. Thus  
$C^*$ avoids the last set otherwise \ort\ implies that $\{y_2,y_3\}$ contains an element of Figure~\ref{caterpillarwhole0}; a \cn. Moreover, $v_k \not \in C^*$ otherwise \ort\ implies that $\{y_2,y_3,y_4\}$ meets both $\{t_k,u_k\}$ and $\{u_{k-2},t_{k-2},u_{k-1},w_{k-1}\}$; a \cn.  By symmetry, $c_2 \not \in C^*$. Hence $C^* = \{y_2,y_3,y_4\}$. This \cn\ completes the proof that $M\ba c_{0},c_1,c_{2},v_0,\dots ,v_k$ is \ifc\ thereby finishing the proof of Lemma~\ref{killthesnake}. 
\end{proof}

\section{More on strings of bowties}
\label{killingfield}

When we deal with a string of bowties that does not wrap around on itself, two situations that arise frequently are shown in Figure~\ref{badguy}.  The next lemma shows that, when one of these situations arises, the main theorem holds. This is another technical lemma although its proof is not as long as that of the  preceding lemma. We continue to follow the practice of using $T_i$ to denote the triangle $\{a_i,b_i,c_i\}$.

\begin{figure}[htb]
\centering
\includegraphics[scale=0.72]{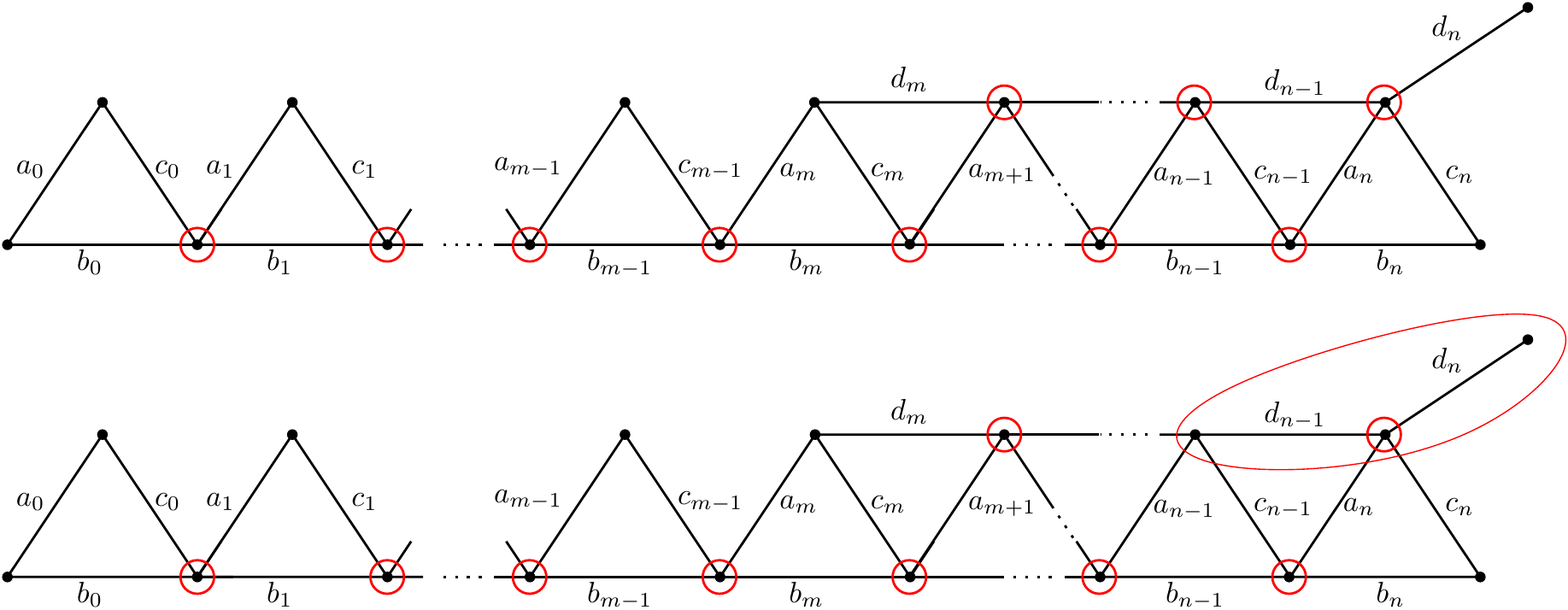}
\caption{Here $0< m < n$.}
\label{badguy}
\end{figure}

\begin{lemma}
\label{killbadguy}
Let $M$ and $N$ be \ifc\ binary matroids with $|E(M)|\geq 13$ and $|E(N)|\geq 7$.  
Let $M$ have a right-maximal bowtie string $T_0,D_0,T_1,D_1,\dots ,T_n$ and suppose that this string is contained in  one of the structures shown in Figure~\ref{badguy}, where $0<m<n$.  
Suppose that $M\ba c_0$ is \ffsc, that $M\ba c_0,c_1,\dots ,c_n$ has an $N$-minor, and that $M\ba c_0,c_1/b_1$ has no $N$-minor.  
Suppose that $M$ has no bowtie of the form $(T_n,T_{n+1},\{x,c_n,a_{n+1},b_{n+1}\})$, where $x\in\{a_n,b_n\}$, and $M$ has no element $d_{m-1}$ such that $\{c_{m-1},d_{m-1},a_m\}$ is a triangle and either $\{d_{m-1},a_m,c_m,d_m\}$ is a cocircuit, or $m+1=n$ and $\{d_{m-1},a_m,c_m,a_{m+1},c_{m+1}\}$ is a cocircuit.  
Then 
\begin{itemize}
\item[(i)] $M$ has a quick win; or
\item[(ii)] $M$ has an open-rotor-chain win or a ladder win; or 
\item[(iii)] $M$ has an enhanced-ladder win.
\end{itemize}  
\end{lemma}
\begin{proof}
We assume that neither (i) nor (ii) holds.  
We begin by showing that 
\begin{sublemma}
\label{cndn}
the elements in each part of Figure~\ref{badguy} are distinct except that $(a_0,b_0)$ may be $(c_n,d_n)$.  
\end{sublemma}

As $T_0,D_0,T_1,D_1,\dots ,T_n$ is a bowtie string, we know that the elements in $T_0\cup T_1\cup\dots\cup T_n$ are all distinct except that $a_0$ may be $c_n$.  

To help prove \ref{cndn}, next we show the following. 

\begin{sublemma}
\label{cndncont}
The elements of  $T_m\cup T_{m+1}\cup\dots \cup T_n\cup\{d_m,d_{m+1},\dots ,d_n\}$ are distinct.
\end{sublemma}

Assume that this fails. Then, by 
 Lemma~\ref{drossdistinct},  $(a_m,b_m)$ is $(c_n,d_n)$ or $(d_{n-1},d_n)$. The first possibility implies that $m= 0$; a \cn.  The second possibility contradicts the hypothesis precluding any bowtie of the form $(T_n,T_{n+1},\{x,c_n,a_{n+1},b_{n+1}\})$ for $x$ in $\{a_n,b_n\}$. We conclude that \ref{cndncont} holds.
 

Next we establish the following.

\begin{sublemma}
\label{cndncont2}
If  $d_j\in T_i$ for some $i$ in $\{0,1,\dots ,m-1\}$ and some $j$ in $\{m,m+1,\dots ,n-1\}$, then $d_j=a_0$, so $a_0\neq c_n$ and $\{d_m,d_{m+1},\dots ,d_{n-1}\}$ avoids $\{b_0,c_0\}\cup T_1\cup T_2\cup\dots\cup T_n$.
\end{sublemma}

Suppose $d_j$ meets one of $D_0, D_1,\dots,D_{m-1}$. Then \ort\ implies that $\{c_j,a_{j+1}\}$ meets the same cocircuit; a \cn.  
Thus $d_j=a_0$. Hence, by \ref{cndncont}, $a_0\neq c_n$, and \ref{cndncont2} follows.

As the next step towards showing that \ref{cndn} holds, we show that 
\begin{sublemma}
\label{dees}
$\{d_m,d_{m+1},\dots,d_{n-1}\}$ avoids $T_0 \cup T_1 \cup \dots \cup T_n$.
\end{sublemma}

Assume that this fails. Then, by \ref{cndncont2}, $d_j = a_0$ for some $j$ in $\{m,m+1,\dots ,n-1\}$, and $a_0 \neq c_n$.  Now $\{d_j,a_{j+1},c_{j+1},d_{j+1}\}$ or $\{d_j,a_{j+1},c_{j+1},a_{j+2},c_{j+2}\}$ is a cocircuit $D^*$. By \ort\ between $D^*$ and $T_0$, we deduce that $D^* = \{d_j,a_{j+1},c_{j+1},d_{j+1}\}$ and $d_{j+1} \in T_0$. If $j < n-1$, then, by \ref{cndncont2}, $d_{j+1} = a_0$. This contradicts the fact that $d_{j} = a_0$. Hence $j = n-1$, so $d_n \in \{b_0,c_0\}$. Therefore $(T_n,T_0,\{a_n,c_n,a_0,d_n\})$ is a bowtie; a \cn. Thus \ref{dees} holds.


To complete the proof of \ref{cndn}, we need to consider the possibility that  
 $d_n\in T_i$ for some $i$ in $\{0,1,\dots ,m-1\}$.  
In that case, by \ort\ between $T_i$ and $\{d_{n-1},a_n,c_n,d_n\}$ using  the fact that  $d_{n-1}$ avoids $T_i$, we deduce that $\{a_n,c_n\}$ meets $T_i$. Thus $i=0$ and $a_0=c_n$.  
Therefore $d_n\in\{b_0,c_0\}$.  
If $d_n=c_0$, then $M\ba c_0,c_1,\dots ,c_n$ has $\{d_{n-1},a_n\}$ as a cocircuit. Hence $M\ba c_0,c_1,\dots ,c_n/a_n$ has an $N$-minor.  But, since $M\ba c_0,c_1/b_1$ has no $N$-minor, this yields a \cn\ to Lemma~\ref{stringybark}. 
We deduce that $d_n = b_0$, so $(a_0,b_0) = (c_n,d_n)$. 
We conclude that  \ref{cndn} holds.

Next we show that 
\begin{sublemma}
\label{nogo}
$M$ has no triangle $\{\al,\be,a_m\}$ such that $\{\be,a_m,c_m,d_m\}$ or $\{\be,a_m,c_m,a_{m+1},c_{m+1}\}$ is a cocircuit.  
\end{sublemma}
 
Assume, instead, that $M$ does have such a triangle and a cocircuit, calling the latter $D^*$.  
Suppose $\{\alpha,\be\}$ meets $\{b_m,c_m\}$. Then $\{\alpha,\be\} = \{b_m,c_m\}$ and so $\beta = c_m$ otherwise the cocircuit $D^*$ contains a triangle. But then $T_m \cup d_m$ or $T_m \cup T_{m+1}$ is $3$-separating in $M$; a \cn. We conclude that $\{\al,\be\}$ avoids $\{b_m,c_m\}$. Orthogonality between $\{\al,\be,a_m\}$ and the cocircuit $\{b_{m-1},c_{m-1},a_m,b_m\}$ implies that $\{\alpha,\be\}$ meets $\{b_{m-1},c_{m-1}\}$. 

Suppose $\be\in\{b_{m-1},c_{m-1}\}$. Then, using \ref{cndn}  and \ort\ between $T_{m-1}$ and  $D^*$, we see that 
$D^* = \{\be,a_m,c_m,a_{m+1},c_{m+1}\}$, that $m+1 = n$, and that $c_{m+1} = a_{m-1}$. Thus $m = 1$ and $n = 2$. Moreover, by \ref{cndn}, $d_2 = b_0$. 
Thus $T_1\cup T_2\cup\{d_1,d_2,c_0\}$ contains $T_0$ and is $3$-separating; a \cn.  
Therefore $\be\notin\{b_{m-1},c_{m-1}\}$.  
We deduce that  $\al\in\{b_{m-1},c_{m-1}\}$. 
 
Suppose $\al =c_{m-1}$. Then it is not difficult to check, by taking $\be = d_{m-1}$, that we violate the hypotheses of the lemma unless $m-1 < n$ and $\{\be,a_m,c_m,a_{m+1},c_{m+1}\}$ is a cocircuit. In the exceptional case, the triangle $\{c_{m+1},d_{m+1},a_{m+2}\}$ implies that $\be \in \{d_{m+1},a_{m+2}\}$. But \ort\ using the triangle $\{\al,\be,a_m\}$ and the cocircuits $\{d_m,a_{m+1},c_{m+1},d_{m+1}\}$ and $D_{m+1}$ gives a \cn. Hence
 $\al \neq c_{m-1}$, so  
 $\alpha =b_{m-1}$. Thus $(\be,b_{m-1},a_m,b_m,c_m)$ is a $5$-fan in $M\ba c_{m-1}$, so $m-1>0$.  
Then \ort\ between $\{b_{m-1},\be,a_m\}$ and the cocircuit $\{b_{m-2},c_{m-2},a_{m-1},b_{m-1}\}$ implies that $\be\in\{b_{m-2},c_{m-2}\}$. By \ort\ between $T_{m-2}$ and $D^*$, we see that $T_{m-2}$ meets $\{a_m,c_m,d_m\}$ or $\{a_m,c_m,a_{m+1},c_{m+1}\}$. In the first case, by \ref{cndn},  $m = 2$ and $(a_0,b_0) = (c_2,d_2)$, so $T_0 \subseteq D^*$; a \cn. We conclude that $T_{m-2}$ meets   $\{a_m,c_m,a_{m+1},c_{m+1}\}$. Then, by \ref{cndn} again, we deduce that $m - 2 = 0$ and $a_0 = c_3$. Hence $b_0 = d_3$, and 
$T_2\cup T_3\cup \{d_2,d_3,c_0\}$ contains $T_0$ and $D^*$. Thus $\lambda (T_2\cup T_3\cup \{d_2,d_3,c_0\})\leq 2$; a \cn.  
We conclude that \ref{nogo} holds.

Let $S=\{ c_m,c_{m+1},\dots ,c_n\}$.  
Next we observe that 
\begin{sublemma}
\label{dandd}
$M\ba c_n$ does not have a $4$-fan, and $\{d_{n-1},d_n\}$ is not contained in a triangle of 
$M$.  
\end{sublemma}

To see this, we note that the first assertion is an immediate consequence of the assumption that $M$ has no bowtie of the form $(\{a_n,b_n,c_n\},\{a_{n+1},b_{n+1},c_{n+1}\},\{x,c_n, a_{n+1},b_{n+1}\})$, where $x\in\{a_n,b_n\}$. The same assumption also gives the second assertion for it implies that if $\{d_{n-1},d_n\}$ is contained in a triangle, then that triangle meets $T_n$, so $\lambda(T_n \cup \{d_{n-1},d_n\}) \le 2$; a \cn.

Next we note that 

\begin{sublemma}
\label{noteit}
$n=m+1$. 
\end{sublemma}

Assume that $n > m+1$. Then we apply Lemma~\ref{dross}. 
Part (i) of that lemma does not hold otherwise $M$ has a ladder win and
part (ii) of this lemma holds; a \cn. By \ref{dandd} and the hypothesis forbidding a certain bowtie, neither part (ii) nor part (iv) of Lemma~\ref{dross} holds.  Hence part (iii) of Lemma~\ref{dross}  holds; that is, $M\ba S$ is \ffsc\ and has a $4$-fan that is a $4$-fan of $M\ba c_n$ or of $M\ba c_m$. The first possibility is excluded by \ref{dandd}. Hence we may assume that $M\ba S$ has a $4$-fan $(x_1,x_2,x_3,x_4)$ that is a $4$-fan of 
$M\ba c_m$.  
Then $M$ has $\{x_2,x_3,x_4,c_m\}$ as a cocircuit and, by \ort, $\{x_2,x_3,x_4\}$ meets both $\{a_m,b_m\}$ and $\{d_m,a_{m+1}\}$. 
Thus $\{x_2,x_3\}$ meets $\{a_m,b_m,d_m,a_{m+1}\}$.  
Suppose $d_m\in\{x_1,x_2,x_3\}$.  
Lemma~\ref{trywhere} implies that $m = n-1$; a \cn.  
Thus $d_m\notin \{x_1,x_2,x_3\}$.  
Furthermore, by Lemma~\ref{trywhere}, we know that $\{a_{m+1},b_m\}$ avoids $\{x_2,x_3\}$. Thus   $a_{m}\in\{x_2,x_3\}$.  
Hence $x_4\in\{d_m,a_{m+1}\}$ and, without loss of generality, $a_m=x_3$.  
If $x_4=d_m$, then we obtain a \cn\ to~\ref{nogo}. Thus $x_4=a_{m+1}$.  
Then \ort\ between $\{x_2,a_m,a_{m+1},c_m\}$ and $T_{m+1}$ implies that $x_2=b_{m+1}$; a \cn\ to Lemma~\ref{trywhere}.  
We conclude that~\ref{noteit} holds.  

We now show that 

\begin{sublemma}
\label{noteit2}
$M$ has   $\{b_m,b_{m+1},c_{m-1}\}$ as a triangle. 
\end{sublemma}

To prove this, we apply Lemma~\ref{minidrossrsv} using a similar argument to that given to prove \ref{noteit}. The hypotheses of the current lemma,  the assumption that (i) of the current lemma does not hold, and \ref{dandd} and \ref{nogo} imply that  either $M$ has a triangle containing  $\{b_m,b_{m+1}\}$,
or $M\ba c_m,c_{m+1}$ has a $4$-fan of the form $(u_1,u_2,u_3,b_{m+1})$ and $M$ has a cocircuit $C^*$ with 
$\{u_2,u_3,c_m,b_{m+1}\}\subseteq C^* \subseteq \{u_2,u_3,c_m,b_{m+1},c_{m+1}\}$. 
Suppose the latter.  
Then \ort\ implies that $\{u_2,u_3\}$ meets $\{a_m,b_m\}$ and   $\{d_m,a_{m+1}\}$, so $\lambda (T_m\cup T_{m+1}\cup d_m)\leq 2$; a \cn.  
We deduce that $\{b_m,b_{m+1}\}$ is contained in a triangle of $M$.  
Let $x$ be the third element of this triangle. 
Then \ort\ implies that $x\in\{b_{m-1},c_{m-1}\}$.  
Suppose $x=b_{m-1}$.  
Then $(b_{m+1},b_{m-1},b_m,a_m,c_m)$ is a $5$-fan in $M\ba c_{m-1}$, so $m-1>0$.  
Then $|\{b_{m-1},b_m,b_{m+1}\}\cap\{b_{m-2},c_{m-2},a_{m-1},b_{m-1}\}|=1$; a \cn\ to \ort.  
We conclude that $x=c_{m-1}$, so \ref{noteit2} holds.

\begin{figure}[htb]
\center
\includegraphics[scale=0.8]{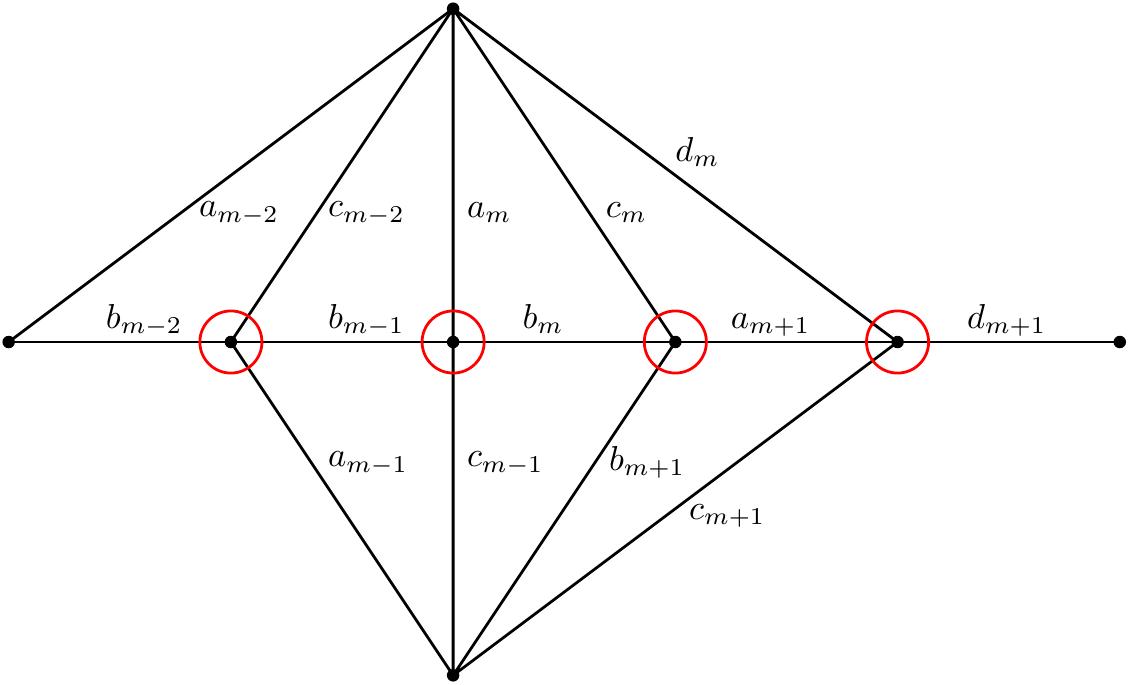}
\caption{}
\label{rotor_graph}
\end{figure}

Next we show that

\begin{sublemma}
\label{not2}
$\{b_{m-1},a_m\}$ is not contained in a triangle of $M$.  
\end{sublemma}

Suppose that $\{b_{m-1},a_m\}$ is contained in a triangle of $M$.  
Then $\{b_m,c_m\}$ is in a $5$-fan in $M\ba c_{m-1}$ with the triangle containing $\{b_{m-1},a_m\}$. 
As $M\ba c_0$ is \ffsc, we deduce that 
$c_{m-1}$ is not $c_0$, so $m>1$.  
Then \ort\ with the cocircuit $\{b_{m-2},c_{m-2},a_{m-1},b_{m-1}\}$ implies that the triangle containing $\{b_{m-1},a_m\}$ contains $c_{m-2}$ or $b_{m-2}$.  
If $\{b_{m-2},b_{m-1},a_m\}$ is a triangle, then $(a_m,b_{m-2},b_{m-1},a_{m-1},c_{m-1})$ is a $5$-fan in $M\ba c_{m-2}$, so $m-2>0$. Thus  $|\{b_{m-3},c_{m-3},a_{m-2},b_{m-2}\}\cap\{b_{m-2},b_{m-1},a_m\}|=1$; a \cn\ to \ort.  
We deduce that $\{b_{m-2},b_{m-1},a_m\}$ is not a triangle, so $\{c_{m-2},b_{m-1},a_m\}$ is a triangle.  
Therefore $M$ contains the configuration shown in Figure~\ref{rotor_graph}.  

From this configuration, we see that 
$((c_{m+1},a_{m+1},b_{m+1}),(c_m,b_m,a_m),\linebreak (c_{m-1},b_{m-1},a_{m-1}),(c_{m-2},b_{m-2},a_{m-2}))$ is a rotor chain in $M$. Extend this to a right-maximal rotor chain  $((c_{m+1},a_{m+1},b_{m+1}),(c_m,b_m,a_m),\dots,(c_k,b_k,a_k))$. Then $k \le m-2$. Moreover, since $M\ba c_0$ is \ffsc, $k \ge 0$. By  Lemma~\ref{deletechainsfc}, 
 $M\ba c_k,c_{k+1},\dots,c_{m+1}$ is \sfc. The last matroid has an $N$-minor and $n = m+1$.  Since  $M$ does not have an open-rotor chain win, we deduce that  $M\ba c_k,c_{k+1},\dots ,c_{m+1}$ is not \ifc.  
 
We now know that $M\ba c_k,c_{k+1},\dots ,c_{m+1}$ has a $4$-fan 
$(y_1,y_2,y_3,y_4)$.  
Suppose first $\{y_1,y_2,y_3\}$ meets $\{b_k,a_{k+1},b_{k+1},\dots ,a_{m+1},b_{m+1},d_m,d_{m+1}\}$.  
Every element in the last set is in a triad of $M\ba c_k,c_{k+1},\dots ,c_{m+1}$.  
Orthogonality implies that $\{y_1,y_2,y_3\}$ contains $\{d_m,d_{m+1}\}$ or $\{b_k,a_{k+1}\}$.  
The former is a \cn\ to~\ref{dandd}, so   the latter holds.  
Then $M\ba c_k$ is not \ffsc, so $k>0$.  
By \ort, $M$ has $\{b_k,a_{k+1},b_{k-1}\}$ or $\{b_k,a_{k+1},c_{k-1}\}$ as a triangle. The first possibility implies, by \ort, that $k = 1$ and hence that $M\ba c_0$ has a $5$-fan; a \cn. Thus $M$ has $\{b_k,a_{k+1},c_{k-1}\}$ as a triangle and therefore has 
 $((c_{m+1},a_{m+1},b_{m+1}),(c_m,b_m,a_m),\dots ,(c_{k-1},b_{k-1},a_{k-1}))$ as   a rotor chain; a \cn.  
We conclude that $\{y_1,y_2,y_3\}$ avoids $\{b_k,a_{k+1},b_{k+1},\dots ,a_{m+1},b_{m+1},d_m,d_{m+1}\}$.  

Now $M$ has a cocircuit $C^*$ such that $\{y_2,y_3,y_4\}\subsetneqq C^*\subseteq \{y_2,y_3,y_4,c_k,c_{k+1},\dots ,c_{m+1}\}$.  
Take $i$ in $\{k,k+1,\dots,m+1\}$ such that $c_i\in C^*$.  
Now $c_i$ is in both $T_i$ and another triangle in the rotor chain  unless  $i = {m+1}$.  
Since $\{y_1,y_2,y_3\}$ avoids $\{b_k,a_{k+1},b_{k+1},\dots ,a_{m+1},b_{m+1},d_m,d_{m+1}\}$, \ort\ implies that either $C^*=\{y_2,y_3,y_4,c_k\}$ where $a_k\in\{y_2,y_3\}$ and $y_4\in\{b_{k+1},a_{k+2}\}$; or  $C^*=\{y_2,y_3,y_4,c_{m+1}\}$ and $y_4\in\{a_{m+1},b_{m+1}\}$.  
Suppose the former.  
As $y_4$ is in $\{b_{k+1},a_{k+2}\}$, \ort\ implies that $a_{k+1}$ or $b_{k+2}$ is in $\{y_2,y_3\}$; a \cn.  
We deduce that  $C^*=\{y_2,y_3,y_4,c_{m+1}\}$ and $y_4\in\{a_{m+1},b_{m+1}\}$.  
Then \ort\ implies that $d_m$ or $b_m$ is in $\{y_2,y_3\}$; a \cn.   
We conclude that \ref{not2} holds. 

It is worth noting at this point that 
\begin{sublemma}
\label{distinct}
all of the elements in $T_{m-1}\cup T_m\cup T_{m+1}\cup\{d_m,d_{m+1}\}$ are distinct.
\end{sublemma}
To verify  this, we simply need to check that $(a_{m-1},b_{m-1})\neq (c_{m+1},d_{m+1})$, since then~\ref{cndn} will imply that~\ref{distinct} holds.  
Suppose instead that $(a_{m-1},b_{m-1})= (c_{m+1},d_{m+1})$.  
Then $T_{m-1}$ is contained in $T_m\cup T_{m+1}\cup\{d_m,d_{m+1},c_{m-1}\}$, so the last  set is $3$-separating; a \cn.  
Thus~\ref{distinct} holds.  

\begin{figure}[htb]
\center
\includegraphics{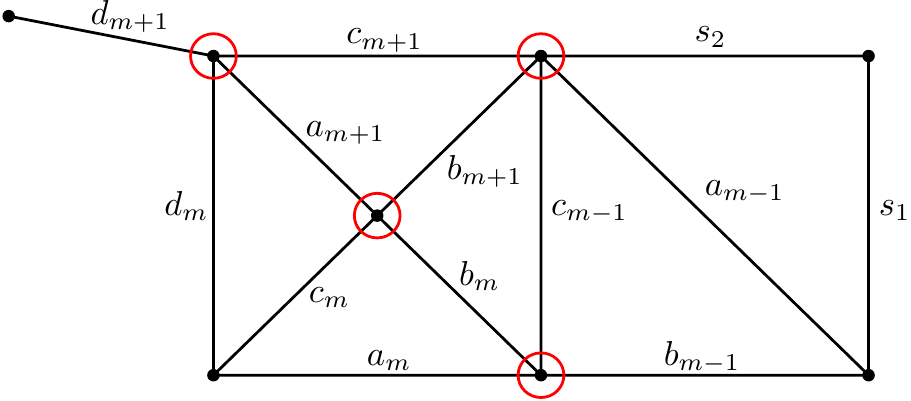}
\caption{}
\label{snake}
\end{figure}

We now apply  Lemma~\ref{minidross2} to the configuration induced by $T_m \cup T_{m+1} \cup \{d_m,d_{m+1},b_{m-1},c_{m-1}\}$ 
noting that \ref{dandd} and \ref{not2} eliminate the possibility that (ii) or (iv) holds; and (i) does not hold otherwise (i) of the current lemma holds. Next we observe that (v) of Lemma~\ref{minidross2} does not hold since if $M$ has a $4$-cocircuit containing 
$\{c_{m-1},b_{m+1},c_{m+1}\}$, then, by \ort, this  cocircuit meets $\{a_{m-1},b_{m-1}\}$ and so we obtain the \cn\ that  $\lambda (T_{m-1}\cup T_m\cup T_{m+1})\leq 2$.  We conclude that part (iii) of   Lemma~\ref{minidross2}  holds, that is, $M$ has a triangle $\{s_1,s_2,s_3\}$ where  $\{c_{m-1},c_{m+1},b_{m+1},s_2,s_3\}$ is a cocircuit, and $\{s_1,s_2,s_3\}$ avoids $\{b_m,c_m,c_{m-1},a_{m+1},b_{m+1},c_{m+1}\}$.  
Orthogonality with $T_{m-1}$ implies that $\{s_2,s_3\}$ meets $\{a_{m-1},b_{m-1}\}$.  
If $b_{m-1}\in\{s_1,s_2,s_3\}$, then, by \ort, it follows that $\{b_{m-1},a_m\}\subseteq \{s_1,s_2,s_3\}$; a \cn\ to~\ref{not2}.  
Thus, without loss of generality, $a_{m-1}=s_3$, so   
$M$ contains  the configuration shown in Figure~\ref{snake}.  

Now $M\ba c_{m+1},c_m,c_{m-1}$ has an $N$-minor and has $(s_1,s_2,a_{m-1},b_{m+1})$ as a $4$-fan. To enable us to apply Lemma~\ref{killthesnake}, we show next that 

\begin{sublemma}
\label{rightdel}
$M\ba c_{m+1},c_m,c_{m-1}\ba s_1$ has an $N$-minor.
\end{sublemma}

Suppose first that $m = 1$. Then, by hypothesis, $M\ba c_m,c_{m-1}/b_{m}$ has no $N$-minor. Thus $M\ba c_{m+1},c_m,c_{m-1}/b_{m+1}$ has no $N$-minor. It follows by Lemma~\ref{2.2} that  
 \ref{rightdel} holds. We may now assume that $m \ge 2$. Then, by \ort\ between the cocircuit $\{b_{m-1},a_{m-1},b_{m-2},c_{m-2}\}$ and the triangle $\{s_1,a_{m-1},s_2\}$, we deduce that $\{s_1,s_2\}$ meets $\{b_{m-2},c_{m-2}\}$. Orthogonality between $T_{m-2}$ and the cocircuit 
$\{c_{m+1},b_{m+1},c_{m-1},a_{m-1},s_2\}$ implies that $s_2 \not\in T_{m-2}$, so $s_1 \in \{b_{m-2},c_{m-2}\}$. Suppose $s_1 = b_{m-2}$. Then $M\ba c_{m-2}$ has a $5$-fan, so $m - 2 > 0$. Then, by \ort, $s_2 \in \{b_{m-3},c_{m-3}\}$ so we have a \cn\ to \ort\ between $T_{m-3}$ and the cocircuit $\{c_{m+1},b_{m+1},c_{m-1},a_{m-1},s_2\}$. We conclude that $s_1 = c_{m-2}$. Then, by assumption, 
$M\ba c_{m+1},c_m,c_{m-1},s_1$ has an $N$-minor, that is, \ref{rightdel} holds.

We can now apply Lemma~\ref{killthesnake} noting that \ref{dandd} and \ref{not2} eliminate the possibility that part  (ii) or (iii) of that lemma holds. Also, by assumption, part (i) 
and part (iv)
of Lemma~\ref{killthesnake} do not hold. Finally, if  part 
(v) of Lemma~\ref{killthesnake}  holds, then part  (iii) of the current lemma holds where, here,  the triple $(c_{m-1},c_m,c_{m+1})$ plays the role of the triple $(c_0,c_1,c_2)$ in 
the configurations  in Figure~\ref{bonesaws0} and Figure~\ref{caterpillarwhole0}, and $M\ba c_{m-1},c_m,c_{m+1},v_0,v_1,\dots, v_k$ is \ifc\ with an $N$-minor. 
\end{proof}

\section{When rotor chains end}
\label{wrce}

In this section, we specify, in Lemma~\ref{betweenbts}, exactly what to expect at the end of a rotor chain. We begin by showing that, when we have a  quasi rotor  of the type whose existence is guaranteed by Lemma~\ref{6.3rsv}, then either outcome (i) of the main theorem occurs, or the quasi rotor can be extended to a right-maximal rotor chain with various desirable properties.

\begin{lemma}
\label{rotorchainends}
Let $(T_0, T_1,T_2,D_0,D_1,\{c_0,b_1,a_2\})$ be a quasi rotor in an \ifc\ binary matroid $M$, where $|E(M)|\geq 13$ and $M\ba c_0$ is \ffsc.  
Suppose $N$ is an \ifc\ proper minor of $M\ba c_0,c_1$ with $|E(N)|\geq 7$.  
Then  one of the following occurs.  
\begin{itemize}
\item[(i)] $M$ has a quick win; or 
\item[(ii)] $M$ has a right-maximal rotor chain $((a_0,b_0,c_0),(a_1,b_1,c_1),\dots ,(a_n,b_n,c_n))$, where $M\ba c_n$ is \ffsc\ and $M\ba c_0,c_1,\dots ,c_n$ has an $N$-minor.  
\end{itemize}
\end{lemma}
\begin{proof}
Let $((a_0,b_0,c_0),(a_1,b_1,c_1),\dots ,(a_n,b_n,c_n))$ be a right-maximal rotor chain.  Then $n \ge 2$. 
We know that $M\ba c_0,c_1$ has an $N$-minor.  
Suppose that there is an element $i$ in $\{2,3,\ldots,n\}$  such that $M\ba c_0,c_1,\dots ,c_{i-1}$ has an $N$-minor, but $M\ba c_0,c_1,\dots ,c_i$ does not.  
As the first matroid has $(c_i,b_i,a_i,b_{i-1})$ as a $4$-fan, we know that $M/b_{i-1}$ has an $N$-minor. Since $M$ has 
$(T_{i-2},T_{i-1},T_{i},D_{i-2},D_{i-1}, \{c_{i-2},b_{i-1},a_i\})$ as a quasi rotor, 
  Lemma~\ref{rotorwin} implies that (i) holds.  
We may therefore assume that $M\ba c_0,c_1,\dots ,c_n$ has an $N$-minor.  

Suppose $M\ba c_n$ is not \ffsc.  
As $(T_{n-1},T_n,D_{n-1}\})$ is a bowtie, Lemma~\ref{6.3rsv} implies that $T_n$ is the central triangle of a quasi rotor $(T_{n-1},T_n,\{d,e,f\},D_{n-1},\{y,c_n,d,e\},\{x,y,d\})$, where $x\in\{b_{n-1},c_{n-1}\}$ and $y\in\{a_n,b_n\}$.  
Suppose $y=a_n$.  
Then \ort\ between $\{y,c_n,d,e\}$ and the triangle $\{c_{n-2},b_{n-1},a_n\}$ implies that $\{d,e\}$ meets $\{c_{n-2},b_{n-1}\}$.  
As $T_{n-1},T_n$, and $\{d,e,f\}$ are disjoint,  $b_{n-1}\notin\{d,e\}$, so $c_{n-2}\in\{d,e\}$ and \ort\ implies that $\{d,e\}\subseteq T_{n-2}$.  
Then $\lambda (T_{n-2}\cup T_{n-1}\cup T_n)\leq 2$, so $|E(M)|\leq 12$; a \cn.    
We deduce that  $y\neq a_n$, so $y=b_n$.  
If $x=b_{n-1}$, then \ort\ implies that $\{b_{n-2},c_{n-2},a_{n-1}\}$ meets $\{b_n,d\}$, so $d\in\{b_{n-2},c_{n-2}\}$.  
Again $\{d,e\}\subseteq T_{n-2}$; a \cn.  
Therefore $x\neq b_{n-1}$, so $x=c_{n-1}$.  

Now $((a_0,b_0,c_0),(a_1,b_1,c_1),\dots,(a_n,b_n,c_n),(d,e,f))$ is not a rotor chain, so $a_0=c_n$,  or $\{d,e,f\}$ meets $T_0\cup T_1\cup \dots \cup T_n$.  As $(T_{n-1},T_n,\{d,e,f\},D_{n-1},\{b_n,c_n,d,e\},\linebreak \{c_{n-1},b_n,d\})$ is a quasi rotor, $\{d,e,f\}$ avoids $T_{n-1} \cup T_n$. 
Suppose $a_0=c_n$.  
Then \ort\ implies that $\{b_0,c_0\}$ meets $\{d,e\}$.  
As $\{b_0,c_0\}\nsubseteq \{d,e,f\}$, \ort\ with $D_0$ implies that $\{a_1,b_1\}$ meets $\{d,e,f\}$.  
If $b_0\in\{d,e,f\}$, then $M\ba c_0$ has $\{d,e,f\}\cup T_1$ as a $5$-fan; a \cn.  
Thus $c_0\in\{d,e\}$ and $M\ba c_0,c_1,\dots ,c_n/b_n$ has an $N$-minor, so Lemma~\ref{stringswitch} implies that $M/b_1$ has an $N$-minor, and Lemma~\ref{rotorwin} implies that (i) holds.  
We deduce that $a_0\neq c_n$.  
Then $\{d,e,f\}$ meets $T_0\cup T_1\cup \dots \cup T_{n-2}$.  

Suppose $\{d,e,f\}$ meets a cocircuit $D_k$ for some $k$  in $\{0,1,\dots ,n-2\}$.  
Then \ort\ implies that two elements of $\{d,e,f\}$ are in $D_k$.  
Orthogonality with $\{b_n,c_n,d,e\}$ implies that $\{d,e\}$ meets at most one of $T_0,T_1,\ldots,T_n$. Hence $k=0$ and $\{d,e,f\} = T_0$. As $\{c_0,b_1,a_2\}$ is a triangle, $c_0 \not \in \{d,e\}$. Hence $\{a_0,b_0\} = \{d,e\}$, so $c_0 = f$. The triangle $\{c_{n-1},b_n,d\}$ and the cocircuit $\{b_0,c_0,a_1,b_1\}$ imply that $d \neq b_0$. Hence $d = a_0$. 
  The symmetric difference of all of the triangles in the rotor chain, that is, each $T_i$ and each $\{c_j,b_{j+1},a_{j+2}\}$,  is $\{a_0,b_0,a_1,c_{n-1},b_n,c_n\}$, and the symmetric difference of the last set with the  triangle $\{c_{n-1},b_n,a_0\}$ is  $\{c_n,b_0,a_1\}$, which must also be a triangle.   
Hence $(c_n,b_0,a_1,b_1,c_1)$ is a $5$-fan in $M\ba c_0$; a \cn.  

We now know that $\{d,e,f\}$ avoids every cocircuit in the rotor chain.  
Then $\{d,e,f\}$ meets the rotor chain in exactly the element $a_0$, so \ort\ implies that $f=a_0$, and we deduce that $((a_0,b_0,c_0),(a_1,b_1,c_1),\dots ,(a_n,b_n,c_n),(d,e,f))$ is a rotor chain; a \cn\ to our selection of a right-maximal rotor chain.  
\end{proof}

\begin{lemma}
\label{betweenbts}
Let $M$ and $N$ be \ifc\ binary matroids where $|E(M)|\geq 13$ and $|E(N)|\geq 7$.  
Suppose that $M$ does not have a proper minor $M'$ such that $|E(M)|-|E(M')|\leq 3$ and $M'$ is \ifc\ with an $N$-minor. 
Let $M$ have a quasi rotor $(T_0,T_1,T_2,D_0,D_1,\{c_0,b_1,a_2\})$ such that $M\ba c_0$ is \ffsc\ and $M\ba c_0,c_1$ has an $N$-minor.  
Then 
$M$ has a right-maximal rotor chain $((a_0,b_0,c_0),(a_1,b_1,c_1),\dots ,(a_{n},b_{n},c_{n}))$ such that $M\ba c_0,c_1,\dots ,c_n $ is \sfc\ with an $N$-minor, $M/b_i$ has no $N$-minor for all $i$ in $\{1,2,\dots ,n -1\}$, and one of the following occurs. 
\begin{itemize}
\item[(i)] $M\ba c_n$ is \ffsc, $M$ has a triangle $\{a_{n+1},b_{n+1},c_{n+1}\}$ and a $4$-cocircuit 
$\{b_n,c_n,a_{n+1},b_{n+1}\}$ such that $T_0,D_0,T_1,D_1,\dots ,T_n,\linebreak 
\{b_n,c_n,a_{n+1},b_{n+1}\}, \{a_{n+1},b_{n+1},c_{n+1}\}$ is a bowtie string in $M$, and $M\ba c_0,c_1,\dots ,c_{n +1}$ has an $N$-minor; or
\item[(ii)] $M$ has an open-rotor-chain win or a ladder win; or 
\item[(iii)] $M$ has an enhanced-ladder win.
\end{itemize}
\end{lemma}
\begin{proof}
Let  $S= \{c_0,c_1,\dots ,c_n\}$. By Lemma~\ref{rotorchainends},  $M$ has a right-maximal rotor chain $((a_0,b_0,c_0),(a_1,b_1,c_1),\dots ,(a_n,b_n,c_n))$ such that $M\ba c_n$ is \ffsc\ and $M\ba S$ has an $N$-minor.  
Suppose $M/b_i$ has an $N$-minor, for some $i$ in $\{1,2,\dots ,n-1\}$. Then applying Lemma~\ref{rotorwin} 
to the quasi rotor $(T_{i-1},T_i,T_{i+1},D_{i-1},D_i,\{c_{i-1},b_i,a_{i+1}\})$ gives that $M$ has a proper minor $M'$ such that $|E(M)|-|E(M')|\leq 3$ and $M'$ is \ifc\ with an $N$-minor; a \cn.  
Therefore, 
\begin{sublemma}
\label{nocon}
$M\ba S$ has an $N$-minor and $M/b_i$ has no $N$-minor for all $i$ in $\{1,2,\dots ,n-1\}$.  
\end{sublemma}

Lemma~\ref{deletechainsfc} implies that 
\begin{sublemma}
\label{gettingsfc}
$M\ba S$ is \sfc.  
\end{sublemma}

We show next that 
\begin{sublemma}
\label{123avoids}
no triangle in $M\ba S$ meets $\{b_0,a_1,b_1,a_2,b_2,\dots,a_n,b_n\}$.  
\end{sublemma}

Suppose $M\ba S$ has a triangle  that meets $\{b_0,a_1,b_1,a_2,b_2,\dots,a_n,b_n\}$. Assume first that $T$ 
 does not contain $\{b_0,a_1\}$.  
Then \ort\ and the fact that $T$ avoids $S$ implies that $T = \{b_i,b_{i+1},b_{i+2}\}$ for some $i$ in $\{0,1,\ldots,n-2\}$. Then  $\lambda (T_i\cup T_{i+1}\cup T_{i+2})\leq 2$ so $|E(M)|\leq 12$; a \cn.  We may now assume that $T$ contains $\{b_0,a_1\}$. 
Then $T \cup \{b_1,c_1\}$ is a $5$-fan in $M\ba c_0$; a \cn.  
We conclude that \ref{123avoids} holds.

If $M\ba S$ is \ifc, then part (ii) 
of the lemma holds, so we assume not. 
Then $M\ba S$ has a $4$-fan $(y_1,y_2,y_3,y_4)$.  
Thus $M$ has a cocircuit $C^*$ such that $\{y_2,y_3,y_4\}\subsetneqq C^*\subseteq \{y_2,y_3,y_4\}\cup S$.  Next we determine the possibilities for $C^*$. 

\begin{sublemma}
\label{cstar}
One of the following occurs. 
\begin{itemize}
\item[(i)] $C^*=\{y_2,y_3,y_4,c_{n}\}$ and $y_4 =b_n$; or 
\item[(ii)] $n=2$ and $a_1=y_4$ and $C^*=\{y_2,y_3,a_1,c_{1}\}$; or
\item[(iii)] $a_0\in\{y_2,y_3\}$ and 
\begin{itemize}
\item[(a)] $n=2$ and $y_4=b_1$ and $C^*=\{y_2,y_3,b_1,c_0,c_1\}$; or
\item[(b)] $n\leq 3$ and $y_4=a_2$ and $C^*=\{y_2,y_3,a_2,c_0,c_2\}$.  
\end{itemize}
\end{itemize}
\end{sublemma}

If $c_i \in C^*$ for some $i$ in $\{1,2,\dots ,n\}$,   
then \ort\ implies that $\{a_i,b_i\}$ meets $\{y_2,y_3,y_4\}$. Hence, by \ref{123avoids}, $y_4\in\{a_i,b_i\}$.  
Thus $c_i$ is the only element of  $\{c_1,c_2,\dots ,c_n\}$ in $C^*$.  
Also, \ort\ implies that $\{c_i,b_{i+1},a_{i+2}\}$ is not a triangle of $M$. Hence $i \geq n-1$, so $C^*\subseteq \{y_2,y_3,y_4,c_0,c_{n-1},c_n\}$.  

Moving towards obtaining \ref{cstar}, we note next that 
\begin{sublemma}
\label{cstar2}
if $C^*$ avoids $c_0$ but contains $c_{n}$, then   $y_4=b_n$.
\end{sublemma}

To see this, 
note that  \ort\ implies that $y_4\in\{a_n,b_n\}$.  
By \ref{123avoids},  $\{y_2,y_3\}$ avoids $\{c_{n-2},b_{n-1}\}$. Hence, by \ort, $y_4 \neq a_n$, so $y_4=b_n$ and \ref{cstar2} holds.  

Next we suppose that $c_{n-1}\in C^*\subseteq \{y_2,y_3,y_4,c_{n-1},c_n\}$.  
Then $y_4\in\{a_{n-1},b_{n-1}\}$, so \ref{cstar2} implies that $C^* = \{y_2,y_3,y_4,c_{n-1}\}$. 
Orthogonoality between $C^*$ and  the circuit $\{c_0,a_1,c_{n-1},a_n\}$ implies that either $\{y_2,y_3\}$ meets $\{a_1,a_n\}$; or $a_1=a_{n-1}=y_4$.  The first possibility gives a \cn\ to \ref{123avoids}. Hence the second possibility holds and therefore so does \ref{cstar}(ii). 

It remains to consider what happens when 
$c_0\in C^*$. In that case, \ort\ with $T_0$ and $\{c_0,b_1,a_2\}$ implies, using \ref{123avoids}, that $a_0\in\{y_2,y_3\}$ and $y_4\in\{b_1,a_2\}$.  If $y_4 = b_1$, then \ort\ 
  implies that  $C^*=\{y_2,y_3,b_1,c_0,c_1\}$ and $n=2$. On the other hand, if $y_4=a_2$, then $C^*=\{y_2,y_3,a_2,c_0,c_2\}$ and $n\leq 3$.  
We conclude that \ref{cstar} holds. 

Next we show the following.
\begin{sublemma}
\label{cstarlet}
If \ref{cstar}(i) holds, then so does the lemma.
\end{sublemma}

As  $(y_1,y_2,y_3,b_n)$ is a $4$-fan in $M\ba S$, and $M\ba S$ has an $N$-minor, we deduce that $M\ba S/b_n$ or $M\ba S\ba y_1$ has an $N$-minor.  
If $M\ba S/b_n$ has an $N$-minor, then so do $M\ba (S-c_n)/b_n\ba a_n$ and hence $M\ba (S-c_n)\ba a_n/b_{n-1}$; a \cn\ to~\ref{nocon}.  
Thus $M\ba S\ba y_1$ has an $N$-minor.  
Now~\ref{123avoids} and \ort\ imply that the elements in $T_0\cup T_1\cup \dots \cup T_n\cup\{y_1,y_2,y_3\}$ are all distinct except that possibly $y_1=a_0$.  Letting $(y_1,y_2,y_3) = (c_{n+1},b_{n+1},a_{n+1})$, we get that part (i) 
of the lemma holds. Thus \ref{cstarlet} holds.
  
We  may now assume that \ref{cstar}(i) does not hold. Thus \ref{cstar}(ii) or \ref{cstar}(iii) holds, so  $n\leq 3$.  At this point, we seek to apply Lemma~\ref{6.3rsv} to the bowtie $(\{a_{n-1},b_{n-1},c_{n-1}\}, \{a_n,b_n,c_n\},\{b_{n-1},c_{n-1},a_n,b_n\})$ to build more structure onto our configuation. We know that 
 $M\ba c_n$  is \ffsc\ but not \ifc, and clearly $M\ba a_{n-1}$ is not \ifc. Thus neither part (i) nor part (iv) of Lemma~\ref{6.3rsv} holds. Next we eliminate the possibility that part (ii) of 
that lemma  holds.

\begin{sublemma}
\label{cstarkers}
If $M$ has a triangle $\{a_{n+1},b_{n+1},c_{n+1}\}$ that is disjoint from $T_{n-1} \cup T_n$ such that 
 $(T_n,\{a_{n+1},b_{n+1},c_{n+1}\},\{x,c_n,a_{n+1},b_{n+1}\})$ is a bowtie for some    $x$ in $\{a_n,b_n\}$, then part (i) of the lemma holds.
\end{sublemma}

Suppose that $M$ has such a triangle and denote it by $T_{n+1}$. Then $(c_{n+1},b_{n+1},a_{n+1},x)$ is a $4$-fan in  $M\ba c_n$.  
Suppose $x=a_n$. Then \ort\ implies that $\{c_{n-2},b_{n-1}\}$ meets $\{a_{n+1},b_{n+1}\}$.  
If $c_{n-2}\in \{a_{n+1},b_{n+1}\}$, then $M\ba S$ has $a_n$ in a $2$-element cocircuit, so $M\ba S/a_n$ has an $N$-minor. Hence $M\ba (S-c_n)\ba b_n/b_{n-1}$ has an $N$-minor, so $M/b_{n-1}$ has an $N$-minor; a \cn. 
 Thus $b_{n-1}\in\{a_{n+1},b_{n+1}\}$, and   $\{a_n,c_n,a_{n+1},b_{n+1}\}$ meets both $T_{n-1}$ and $T_n$; a \cn\ to Lemma~\ref{bowwow}.  
We conclude that $x\neq a_n$. Thus $x=b_n$  and $\{b_n,c_n,a_{n+1},b_{n+1}\}$ is a cocircuit.  

Suppose that $T_0,D_0,T_1,D_1,\dots ,T_n,D_{n+1},T_{n+1}$ is a bowtie string.  
As $M\ba S$ has $(c_{n+1},b_{n+1},a_{n+1},b_n)$ as a $4$-fan, either 
 $M\ba S\ba c_{n+1}$ has an $N$-minor and (i) holds, as desired, or  $M\ba S/b_{n}$ has an $N$-minor, and Lemma~\ref{stringswitch} implies that $M\ba c_0,c_1/b_1$ has an $N$-minor; a \cn.  
It follows that  $T_0,D_0,T_1,D_1,\dots, T_n,D_{n+1},T_{n+1}$ is not a bowtie string.  
Therefore $T_{n+1}$ meets $\{b_0,c_0\}\cup T_1\cup T_2\cup \dots \cup T_{n-2}$   
since we know that $T_{n+1}$ avoids $T_{n-1} \cup T_n$. 
If $n=2$, then $T_{n+1}$, which is $T_3$, meets $\{b_0,c_0,a_1,b_1\}$, so $T_3 = T_0$ and 
$\lambda(T_0 \cup T_1 \cup T_2) \le 2$; a \cn. Hence we may assume that $n = 3$. Therefore \ref{cstar}(iii)(b) holds, so $a_0 \in \{y_2,y_3\}$ and $C^* = \{y_2,y_3,a_2,c_0,c_2\}$.

Suppose $T_4$ meets $\{b_1,c_1\}$. Since $T_4$ avoids $T_2 \cup T_3$, it follows, by \ort\   with the cocircuit $\{b_1,c_1,a_2,b_2\}$,  that $\{b_1,c_1\} \subseteq T_4$, so $T_1 = T_4$. Then $\lambda(T_1 \cup T_2 \cup T_3) \le 2$; a \cn. Hence $T_4$ avoids $\{b_1,c_1\}$. Thus $T_4$ meets $\{b_0,c_0,a_1,b_1\}$. As $T_4$ avoids $b_1$, it follows that either $T_4 = T_0$, or $T_4$ meets $\{b_0,c_0,a_1,b_1\}$ in  $\{b_0,a_1\}$ or $\{c_0,a_1\}$. Next we shall eliminate each of these possibilities. 

Suppose $T_4 = T_0$. We know that $a_0 \in \{y_2,y_3\}$. If $a_0 \in \{a_4,b_4\}$, then, since $\{y_1,y_2,y_3\}$ is a circuit of $M\ba S$, it follows by \ort\ between $\{y_1,y_2,y_3\}$ and the cocircuit $\{b_3,c_3,a_4,b_4\}$ that $b_3 \in \{y_1,y_2,y_3\}$ otherwise $T_4 = \{y_1,y_2,y_3\}$; a \cn.  Thus $\{y_1,y_2,y_3\}$ contains $\{a_0,b_3\}$ and we get a \cn\ to \ort. We deduce that $a_0 \not \in \{a_4,b_4\}$, so $a_0 = c_4$. Hence $\{b_0,c_0\} = \{a_4,b_4\}$. Therefore $M$ has $\{b_0,c_0,a_1,b_1\}$ and $\{b_0,c_0,b_3,c_3\}$ as cocircuits and so has $\{a_1,b_1,b_3,c_3\}$ as a cocircuit. Then $\lambda(T_1 \cup T_2 \cup T_3) \le 2$; a \cn. We conclude that $T_4 \neq T_0$. 

Now $T_4$ does not meet $\{b_0,c_0,a_1,b_1\}$ in $\{b_0,a_1\}$ otherwise $M\ba c_0$ has $T_4 \cup \{b_1,c_1\}$ as a $5$-fan, which contradicts the fact that $M\ba c_0$ is \ffsc. It remains to consider the case when $T_4$ meets $\{b_0,c_0,a_1,b_1\}$ in $\{c_0,a_1\}$. We know that $\{b_0,c_0,a_1,b_1\}$ and $\{b_3,c_3,a_4,b_4\}$ are cocircuits of $M$. As $\{c_0,a_1\} \subseteq T_4$, we see that $\{c_0,a_1\}$ is $\{a_4,b_4\}, \{a_4,c_4\}$, or $\{b_4,c_4\}$. Moreover, $T_4$ avoids $\{a_0,b_0,b_1,c_1\}$ as $T_4$ is not $T_0$ or $T_1$. For each possibility for $\{c_0,a_1\}$, we take the symmetric difference of the cocircuits $\{b_0,c_0,a_1,b_1\}$ and $\{b_3,c_3,a_4,b_4\}$. These symmetric differences are 
$\{b_0,b_1,b_3,c_3\}$, $\{b_0,b_1,b_3,b_4,c_3,c_4\}$, and $\{a_4,b_0,b_1,b_3,c_3,c_4\}$. The triangles in $M$ imply that each is a cocircuit, $D^*$ since $M$ has no triad meeting a triangle. The first case gives an immediate \cn\ to \ort. For the second and third, \ort\ between $D^*$ and the triangle $\{c_0,b_1,a_2\}$ implies that $c_0 \in D^*$. Thus $(c_0,a_1)$ is $(c_4,a_4)$ or $(c_4,b_4)$, respectively. In each case, since $T_4$ avoids $\{b_1,c_1\}$, \ort\ between $D^*$ and $T_1$ gives a \cn.  This completes the proof of \ref{cstarkers}. 

By \ref{cstarkers} and the remarks preceding it, we may assume that part (iii) of Lemma~\ref{6.3rsv} holds, that is, 
 every \ftv\ of $M\ba c_n$ is a $4$-fan of the form $(u,v,w,x)$ where $u$ and $v$ are in $\{b_{n-1},c_{n-1}\}$ and $\{a_n,b_n\}$, respectively, and $|T_{n-1}\cup T_n\cup \{w,x\}|=8$.  Then $\{v,w,x,c_n\}$ is a cocircuit of $M$. 

Next we show that 
\begin{sublemma} 
\label{quack}
$M$ contains the configuration shown in Figure~\ref{schmuckface}. 
\end{sublemma}

Suppose $u=b_{n-1}$. Then \ort\ implies that $w\in \{b_{n-2},c_{n-2},a_{n-1}\}$. 
But $\{u,v,w\} \neq T_{n-1}$, so $w \neq a_{n-1}$. If $w = b_{n-2}$, then   $\lambda (T_{n-2}\cup T_{n-1}\cup T_n)\leq 2$; a \cn.  Thus $w = c_{n-2}$, so $v = a_n$. Then the cocircuit $\{v,w,x,c_n\}$ is $\{a_n,c_{n-2},x,c_n\}$ so, by \ort, $x \in \{b_{n-2}, a_{n-2}\}$ and again we get the \cn\ that $\lambda (T_{n-2}\cup T_{n-1}\cup T_n)\leq 2$. We conclude that 
 $u=c_{n-1}$.  
If $v=a_n$, then  \ort\ between the triangle $\{c_{n-2},b_{n-1},a_n\}$ and the cocircuit $\{a_n,w,x,c_n\}$  implies that $\{c_{n-2},b_{n-1}\}$ meets $\{w,x\}$.  
Then \ort\ implies that $\{a_n,w,x,c_n\}$ is contained in $T_{n-2}\cup T_{n-1}\cup T_n$. Hence this set is $3$-separating in $M$; a \cn.  
Thus $v=b_n$, so $M$ contains the configuration shown in Figure~\ref{schmuckface}.  

\begin{figure}[h]
\centering
\includegraphics[scale=0.9]{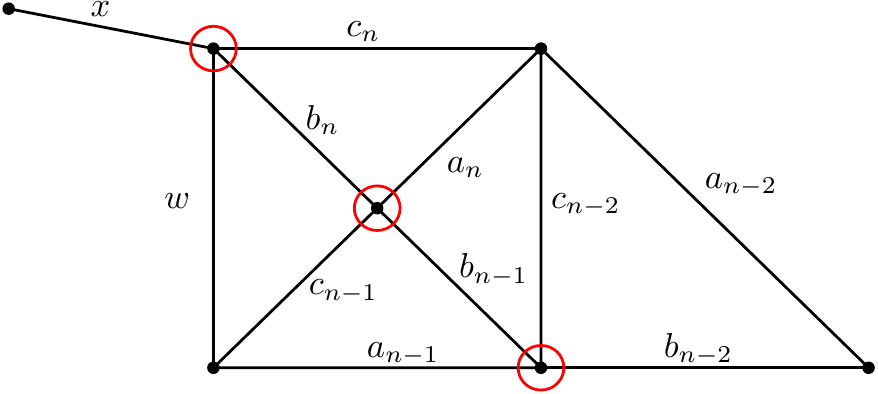}
\caption{}
\label{schmuckface}
\end{figure}

Now recall that \ref{cstar} holds.  Next we show that
\begin{sublemma}
\label{wbn}
$\{w,b_n\}$ avoids $\{y_2,y_3\}$.
\end{sublemma}

By \ref{123avoids}, $b_n \not \in \{y_1,y_2,y_3\}$. Suppose $w \in \{y_2,y_3\}$. Then, by \ort\ between the circuit $\{y_1,y_2,y_3\}$ and the cocircuit $\{w,x,b_n,c_n\}$, we deduce that $\{w,x\} \subseteq \{y_1,y_2,y_3\}$. Thus $M\ba c_n$ has a $5$-fan. This \cn\ establishes that \ref{wbn} holds.

By \ref{cstarlet}, we may assume that \ref{cstar}(i) does not hold. 
Now suppose that \ref{cstar}(iii)  holds.  
Then $n\leq 3$ and, without loss of generality, $a_0=y_3$.  
If \ref{cstar}(iii)(a) holds, then \ort\ between the cocircuit $\{y_2,y_3,b_1,c_0,c_1\}$ and the circuits $\{c_1,b_2,w\}$ and $T_2$ implies that $w\in\{y_2,b_1,c_0\}$.  Since $n = 2$, we see that 
$w\notin \{b_1,c_0\}$, so $w = y_2$; a \cn\ to \ref{wbn}.   
We may now assume that \ref{cstar}(iii)(b) holds.  
If $n=2$, then we have a configuration  of the form shown in Figure~\ref{snake0}.  
Now $M\ba c_0,c_1,c_2$ has an $N$-minor and has $(y_1,y_2,a_0,a_2)$ as a $4$-fan. Thus $M\ba c_0,c_1,c_2/a_2$ or $M\ba c_0,c_1,c_2\ba y_1$ has an $N$-minor. The first case does not arise because $M/b_1$ has no $N$-minor, yet 
$$M\ba c_0,c_1,c_2/a_2 \cong M/a_2\ba b_1,b_2,c_1  \cong M\ba a_2, b_1,b_2,c_1 \cong M/b_1\ba a_2,b_2,c_1.$$
Hence $M\ba c_0,c_1,c_2,y_1$ has an $N$-minor, so we can apply 
Lemma~\ref{killthesnake} to obtain that the lemma holds unless   $\{w,x\}$ or $\{b_0,a_1\}$ is contained in a triangle, $T$.  
In the exceptional cases, either $T\cup \{c_{1},b_2\}$ contains a $5$-fan in $M\ba c_2$, or $T\cup \{b_1,c_1\}$ contains a $5$-fan in $M\ba c_0$, so we get a \cn.  
Thus $n \neq 2$ so $n = 3$.  By \ort\ between $\{w,b_3,c_2\}$ and $\{b_0,c_0,a_1,b_1\}$, we see that $w \neq c_0$. 
Then \ort\ between $\{w,b_3,c_2\}$ and $\{y_2,y_3,a_2,c_0,c_2\}$ implies that $\{w,b_3\}$ meets $\{y_2,y_3\}$; a \cn\ to \ref{wbn}. 
We conclude that \ref{cstar}(iii) does not hold.  

We may now assume that \ref{cstar}(ii) holds,  
that is, $n=2$ and $C^*=\{y_2,y_3,a_1,c_1\}$.  
Since $\{y_1,y_2,y_3\}$ is a triangle and $\{y_2,y_3,a_1,c_1\}$ is a cocircuit,    
we see  that $\{y_2,y_3\}$ avoids $T_1$. Then \ort\ between $\{y_2,y_3,a_1,c_1\}$ and the triangle $\{w,c_{n-1},b_n\}$ implies  that $\{w,b_n\}$ meets $\{y_2,y_3\}$.  This \cn\ to \ref{wbn} completes the proof. 
\end{proof}

\section{Proof of the main theorem}
\label{pomt}

In this section, we complete the proof of Theorem~\ref{killcasek}. We begin by proving a lemma that  treats  the case
in which  a right-maximal bowtie string does not wrap  around into a bowtie 
ring.  One fact that will be used repeatedly in the next proof is that if $T_0, D_0,T_1,D_1,\dots ,T_n$ is a string of bowties, then, relative to this string, there is symmetry between the elements $a_n$ and $b_n$.

\begin{figure}[htb]
\centering
\includegraphics[scale=1.]{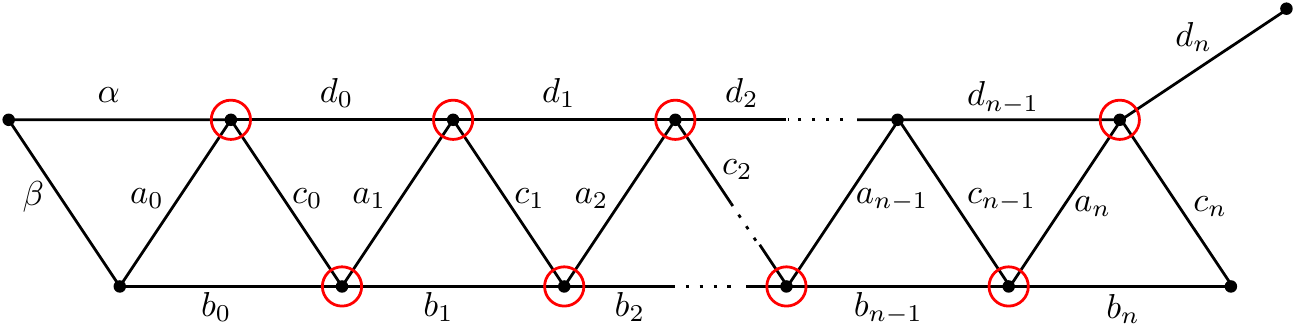}
\caption{$\{d_{n-2},a_{n-1},c_{n-1},d_{n-1}\}$ or $\{d_{n-2},a_{n-1},c_{n-1}, a_n,c_n\}$ is a cocircuit where $d_{-1} = \alpha$ when $n = 1$.}
\label{drossfigiinouv}
\end{figure}

\begin{lemma}
\label{reachtheend}
Let $M$ and $N$ be \ifc\ binary matroids such that  $|E(M)|\geq 13$ and $|E(N)|\geq 7$.  
Suppose that $M$ has a bowtie $(T_0,T_1,D_0)$, such that $M\ba c_0$ is \ffsc, and $M\ba c_0,c_1$ has an $N$-minor.  
Let $T_0, D_0,T_1,D_1,\dots ,T_n$ be a right-maximal string of bowties.  
If $(T_n,T_0,\{x,c_n,a_0,b_0\})$ is not a bowtie for all $x$ in $\{a_n,b_n\}$, then one of the following holds. 
\begin{itemize}
\item[(i)] $M$ has a quick win; or 
\item[(ii)] $M$ has an open-rotor-chain win or a ladder win; or 
\item[(iii)]  $M$ contains the configuration in Figure~\ref{drossfigiinouv}, up to switching the labels of $a_n$ and $b_n$, and $M\ba c_0,c_1,\dots, c_n$ has an $N$-minor; or 
\item[(iv)]  $M\ba c_0,c_1/b_1$ has an $N$-minor; or $n= 1$ and $M\ba c_0,c_1/b_1$ or $M\ba c_0,c_1/a_1$ has an $N$-minor; or
\item[(v)]  $M$ has an enhanced-ladder win.
\end{itemize}
\end{lemma}
\begin{proof}
Suppose that none of (i), (iii), or (iv) holds.  Then, by Lemma~\ref{stringybark},  $M\ba c_0,c_1,\dots ,c_n$ has an $N$-minor, but 
$M\ba c_0,c_1,\dots ,c_i/b_i$ has no $N$-minor for all $i$ in $\{1,2,\dots, n\}$, and    $M\ba c_0,c_1,\dots ,c_n/a_n$ has no $N$-minor.

First we show that 

\begin{sublemma} 
\label{nobthere}
$M$ has no bowtie of the form $(T_n,\{a_{n+1},b_{n+1},c_{n+1}\},\{x,c_n,a_{n+1},b_{n+1}\})$, where $x\in\{a_n,b_n\}$.  
\end{sublemma}

Since $T_0, D_0,T_1,D_1,\dots ,T_n$ is a right-maximal string of bowties, we may assume that $n \ge 2$ otherwise \ref{nobthere} certainly holds. By the symmetry between $a_n$ and $b_n$, it suffices to prove that  $(T_n,T_{n+1},D_n)$ is not a bowtie.  Assume the contrary. Observe that $\{a_{n+1},b_{n+1}\}$ avoids $\{c_0,c_1,\ldots,c_n\}$ otherwise 
$M\ba c_0,c_1,\dots ,c_n$ has $b_n$ in a $1$- or $2$-element cocircuit, so $M\ba c_0,c_1,\dots ,c_n/b_n$ has an $N$-minor; a \cn. 

To enable us to apply Lemma~\ref{ring}, we now show that $a_0 \neq c_n$. Assume otherwise. Then, by \ort\ and symmetry, we may assume that $a_{n+1} \in \{b_0,c_0\}$. From the last paragraph, we know that $a_{n+1} \neq c_0$, so $a_{n+1} = b_0$. Now $c_0 \notin T_{n+1}$ otherwise $T_{n+1} = T_0$, so $c_n \in T_{n+1}$; a \cn. 
By \ort\ between $D_0$ and $T_{n+1}$, we see that  $\{a_1,b_1\}$ meets $\{b_{n+1},c_{n+1}\}$. If $b_{n+1} \in \{a_1,b_1\}$, then \ort\ between $T_1$ and the cocircuit $\{b_n,c_n,a_{n+1},b_{n+1}\}$ gives the \cn\ that $T_1$ meets $\{b_n,c_n\}$. Thus $c_{n+1} \in \{a_1,b_1\}$. Then $M\ba c_0$ has $T_{n+1} \cup T_1$ as a $5$-fan; a \cn. We conclude that $a_0 \neq c_n$.

By Lemma~\ref{ring}, $T_{n+1}=T_j$ for some $j$ in $\{0,1,\dots ,n-2\}$ otherwise we contradict the fact that   $T_0, D_0,T_1,D_1,\dots ,T_n$ is a right-maximal string of bowties. 
If $j=0$, then the hypothesis forbidding $(T_n,T_0, \{x,c_n,a_0,b_0\})$ from being a bowtie implies  that $\{a_0,b_0\}\neq \{a_{n+1},b_{n+1}\}$, so $c_0\in\{a_{n+1},b_{n+1}\}$; a \cn. Thus $j$ is in $\{1,2,\dots ,n-2\}$, and $n\geq 3$.  
Since $\{a_{n+1},b_{n+1}\}$ avoids $c_j$,  we see that $\{a_{n+1},b_{n+1}\}=\{a_j,b_j\}$, and $\{b_n,c_n,a_{n+1},b_{n+1}\}\bigtriangleup \{b_{j-1},c_{j-1},a_j,b_j\}$, which  is   $\{b_n,c_n,b_{j-1},c_{j-1}\}$, is a cocircuit. Thus $M\ba c_0,c_1,\dots, c_n/b_n$ has an $N$-minor; a \cn\ to Lemma~\ref{stringybark}. We conclude that \ref{nobthere} holds.

We now  show that 

\begin{sublemma}
\label{newboater} 
up to relabelling $a_n$ and $b_n$, the matroid $M$ has elements $d_{n-1}$ and $d_n$ such that $\{c_{n-1},d_{n-1},a_n\}$ is a triangle and $\{d_{n-1},a_n,c_n,d_n\}$ is a cocircuit.  
\end{sublemma}

We shall apply Lemma~\ref{6.3rsv} to the bowtie $(T_{n-1},T_n,D_{n-1})$. By \ref{nobthere}, $\{a_{n},b_{n},c_{n}\}$ is not the central triangle of a quasi rotor of the form described in Lemma~\ref{6.3rsv}, and part (ii) of that lemma does not hold. Moreover, as $M\ba c_n$ has an $N$-minor, it is not \ifc\ otherwise part (i) of the current lemma holds. We deduce that (iii) or (iv) of Lemma~\ref{6.3rsv}   holds. 

Suppose that part (iv) of Lemma~\ref{6.3rsv}  holds. Then, as $M\ba a_{n-1}$ is \ifc, we must have that $n = 1$ and that $M$ has  a triangle $\{a_0,7,8\}$   and a cocircuit $\{x,c_1,7,8\}$  for some $x\in\{a_1,b_1\}$.  Then, as part (i) of the current lemma does not hold, we deduce that $M\ba a_0$ has no $N$-minor. 
Now $M\ba c_0,c_1$ has $(a_0,7,8,x)$ as a $4$-fan, so $M\ba c_0,c_1,a_0$ or $M\ba c_0,c_1/x$ has an $N$-minor.  
The former gives the \cn\ that $M\ba a_0$ has an $N$-minor. Thus  the latter holds.  
Let $y$ be the element in $\{a_1,b_1\}-x$.  
As $M\ba c_0,c_1/x\cong M\ba c_0,y/x\cong M\ba c_0,y/b_0\cong M/b_0\ba a_0,y$, we deduce  that $M\ba a_0$ has an $N$-minor; a \cn.  We conclude that part (iv) of Lemma~\ref{6.3rsv} does not hold.

Finally, suppose that part (iii) of Lemma~\ref{6.3rsv}   holds. Then, up to relabelling $a_n$ and $b_n$, the matroid $M$ has elements $d_{n-1}$ and $d_n$ that are not in $T_{n-1} \cup T_n$ such that $\{u,d_{n-1},a_n\}$ is a triangle and $\{d_{n-1},a_n,c_n,d_n\}$ is a cocircuit for some $u$ in $\{b_{n-1},c_{n-1}\}$. If $u = c_{n-1}$, then \ref{newboater} holds. Thus we may  suppose that $u = b_{n-1}$. Then $n> 1$ otherwise we obtain the \cn\ that $M\ba c_0$ has a $5$-fan. By \ort\ between the triangle $\{b_{n-1},d_{n-1},a_n\}$ and the cocircuit $D_{n-2}$, we deduce that $d_{n-1} \in \{b_{n-2},c_{n-2}\}$. Then, by \ort\ between $T_{n-2}$ and the cocircuit $\{d_{n-1},a_n,c_n,d_n\}$, we see that either $d_n \in T_{n-2}$, or $T_{n-2}$ meets $\{a_n,c_n\}$. The latter implies that $n =2$ and $a_0 = c_2$. Thus both cases give the  \cn\ that $\lambda(T_{n-2} \cup T_{n-1} \cup T_n) \le 2$. We conclude that \ref{newboater} holds.

Next we show that 
\begin{sublemma}
\label{nge2}
$n \ge 2$.
\end{sublemma}

Suppose that $n = 1$. By \ref{newboater}  $M$ contains the  configuration shown in Figure~\ref{onehorned}.  We now apply 
Lemma~\ref{minidrossrsv}. Since  part (i) of the current lemma does not hold but \ref{nobthere} does, we deduce that neither part (i) nor part (ii) of Lemma~\ref{minidrossrsv} holds. Moreover, part (iii)   of Lemma~\ref{minidrossrsv} does not hold since if $\{b_0,b_1\}$ is contained in a triangle, then we obtain the \cn\ that $M\ba c_0$ has a $5$-fan containing this triangle and $\{a_1,c_1\}$.  If part (iv) of Lemma~\ref{minidrossrsv} holds, then part (iii) of the current lemma holds. It remains to consider the case when part (v) of Lemma~\ref{minidrossrsv} holds, that is, when $M\ba c_0,c_1$ has a $4$-fan of the form $(y_1,y_2,y_3,b_1)$.  
Then $M$ has a cocircuit $C^*$  such that $\{y_2,y_3,b_1\}\subsetneqq C^*\subseteq \{y_2,y_3,b_1,c_0,c_1\}$.  
Then~\ref{nobthere} implies that $c_0\in C^*$, so \ort\ implies that $\{y_2,y_3\}$ meets $\{a_0,b_0\}$ and $\{d_0,a_1\}$. If $\{y_2,y_3\} \neq \{b_0,a_1\}$, then 
 $\lambda (T_0\cup T_1\cup d_0)\leq 2$; a \cn.  Thus $\{y_2,y_3\} = \{b_0,a_1\}$. Hence, by \ort, $y_1 = d_1$, so  $\lambda (T_0\cup T_1\cup d_0\cup d_1)\leq 2$; a \cn.
We conclude that \ref{nge2} holds.

We  show next that 
\begin{sublemma}
\label{signpost}
$M$ does not contain  the ladder segment shown  in Figure~\ref{drossfigi}, nor does $M$ contain the ladder segment in Figure~\ref{drossfigi} after switching the labels on $a_n$ and $b_n$.
\end{sublemma}

By the symmetry between $a_n$ and $b_n$ in the lemma statement, it suffices to show that $M$ does not contain the first of these ladder segments. Assume the contrary. We show first that the elements in this ladder segment are distinct. If not, then Lemma~\ref{drossdistinct} implies that either $\{b_n,c_n,c_0,d_0\}$ is a cocircuit of $M$, or $(a_0,b_0)=(d_{n-1},d_n)$.  
The former implies that $M\ba c_0,c_1,\dots ,c_n$ has $\{b_n,d_0\}$ as a cocircuit, so $M\ba c_0,c_1,\dots ,c_n/b_n$ has an $N$-minor; a \cn.  
The latter implies that $(T_n,T_0,\{a_0,b_0,a_n,c_n\})$ is a bowtie; a \cn\ to~\ref{nobthere}.  
We conclude that the elements in the ladder segment are distinct.  

We now apply Lemma~\ref{dross}. Since $M\ba c_0,c_1,\dots,c_n$ has an $N$-minor, if it is \ifc, then part (ii) of the current lemma holds. Thus we may assume that part (i) of Lemma~\ref{dross} does not hold. Moreover, by \ref{nobthere}, part (ii) of Lemma~\ref{dross} does not hold.
 Also, part (iv) of  Lemma~\ref{dross} does not hold otherwise 
$M$ is the cycle matroid of a quartic M\"obius ladder as shown in Figure~\ref{drossfigii} and we get a \cn\ to \ref{nobthere}. 
We deduce that part (iii) of Lemma~\ref{dross} holds; that is, either  
$M\ba c_0$ has  a $4$-fan of the form $(\al,\be,a_0,d_0)$, or $M\ba c_n$ has  a $4$-fan of the form $(y_1,y_2,y_3,b_n)$  
The latter implies that $(\{a_n,b_n,c_n\},\{y_1,y_2,y_3\},\{y_2,y_3,b_n,c_n\})$ is a bowtie. This gives a \cn\ to~\ref{nobthere}, so the former holds. Hence so does part (iii) of the current lemma; a \cn. Thus \ref{signpost} holds. 
 
By \ref{newboater}, $M$ contains one of the configurations in Figure~\ref{badguy} with $m = n-1$. 
By \ref{signpost},  if we take $m$ to be as small as possible such that $M$ contains one of the 
configurations in Figure~\ref{badguy}, then $m > 0$. Moreover, the hypotheses of 
Lemma~\ref{killbadguy} hold.  It follows, by that lemma, that (ii) or (v) of the current lemma holds.
\end{proof}

Next we note a helpful property of bowtie rings.

\begin{lemma}
\label{lordoftherings}
Let  $(T_0,D_0,T_1,D_1,\dots, T_n,\{b_n,c_n,a_0,b_0\})$ be a ring of bowties in a matroid $M$. Then 
$$M\ba c_{0},c_{1},\dots c_n/a_1 \cong M\ba a_{0},b_{1},a_2,a_3,\dots a_n/b_2.$$
\end{lemma}

\begin{proof}
We have
\begin{align*}
M\ba c_{0},c_{1},\dots c_n/a_1 &\cong M\ba c_{0},b_{1},c_2,c_3,\dots c_n/a_1\\
&\cong M\ba c_{0},b_{1},c_2,c_3,\dots c_n/b_0\\
&\cong M\ba a_{0},b_{1},c_2,c_3,\dots c_n/b_0\\
&\cong M\ba a_{0},b_{1},c_2,c_3,\dots c_n/b_n\\
&\cong M\ba a_{0},b_{1},c_2,c_3,\dots c_{n-1},a_n/b_{n-1}\\
&\hspace{.2 cm} \vdots \\
&\cong M\ba a_{0},b_{1},a_2,a_3,\dots a_n/b_2.
\end{align*}
\end{proof}

We showed in Lemma~\ref{btring} that, when we have a bowtie ring, we may obtain one of the structures shown in Figure~\ref{ibis_n_haz}. 
The next two lemmas deal with these two structures.

\begin{lemma}
\label{ibhaz}
Let $M$ and $N$ be \ifc\ binary matroids such that  $|E(M)|\geq 13$ and $|E(N)|\geq 7$.   
Suppose that $M$ has a  bowtie ring $(T_0,D_0,T_1,D_1,\dots, T_n,\{b_n,c_n,a_0,b_0\})$,   that $M\ba c_0$ is \ffsc, that $M\ba c_0,c_1$ has  an $N$-minor, and that $M\ba c_0,c_1/b_1$ does not have  an $N$-minor.  
Let  $M$ have a bowtie $(T_j,S_1,C_0)$  for some $j$ in $\{1,2,\dots, n\}$ where $C_0$ is $\{z,c_j,e_1,f_1\}$ for some $z$ in $\{a_j,b_j\}$ and $S_1$ is a triangle $\{e_1,f_1,g_1\}$ that avoids $T_0\cup T_1\cup \dots \cup T_n$.  Let $T_0,D_0,T_1,D_1,\dots ,T_j,C_0,S_1,C_1,\dots ,S_\ell$ be a right-maximal bowtie string where, for all $i$ in $\{2,3,\dots, \ell\}$, the set $S_i$ is a triangle, $\{e_i,f_i,g_i\}$, and $C_{i-1}$ is a cocircuit, $\{f_{i-1},g_{i-1},e_{i},f_{i}\}$. 
If   $M$ has   $\{y,g_{\ell},a_0,b_0\}$ as a cocircuit for some $y$ in $\{e_{\ell},f_{\ell}\}$, then 
 $j =1$ and $z=a_j$, and both of the matroids $M\ba c_{0},c_{1},\dots c_n/a_1$ and $M\ba a_{0},b_{1},a_2,a_3,\dots a_n/b_2$ have $N$-minors.  
\end{lemma}


\begin{proof} Assume that the lemma fails. By symmetry, we may suppose that  $M$ has  $\{f_{\ell},g_{\ell},a_0,b_0\}$ as a 
cocircuit.  Now either this cocircuit equals $\{b_n,c_n,a_0,b_0\}$, or the symmetric difference of these two cocircuits is 
$\{f_\ell,g_\ell,b_n,c_n\}$ and the last set is a cocircuit of $M$. 
By 
Lemma~\ref{stringybark}, $M\ba c_0,c_1,\dots ,c_n$ has an $N$-minor but $M\ba c_0,c_1,\dots, c_i/b_i$ has no $N$-minor for all $i$ in $\{1,2,\dots ,n\}$, and $M\ba c_0,c_1,\dots, c_h/b_h$ has no $N$-minor for all $h$ in $\{2,3,\dots ,n\}$. We  show first that 

\begin{sublemma}
\label{jlen-1}
$j \leq n-1$.
\end{sublemma} 

Suppose $j = n$. By Lemma~\ref{stringybark}, $M\ba c_0,c_1,\dots ,c_n,g_1,\dots,g_{\ell}$ has an $N$-minor. 
Now  $S_{\ell}$ avoids $T_n$, so $\{f_\ell,g_\ell,b_n,c_n\}$ is a cocircuit of $M$. Thus  $M\ba c_0,c_1,\dots ,c_n,g_1,\dots,g_{\ell}$ has 
$b_n$ in a cocircuit of size at most two, so $N\preceq M\ba c_0,c_1,\dots ,c_n/b_n$; a \cn. Hence \ref{jlen-1} holds. 

Next we show that 

\begin{sublemma}
\label{bowtieintn}
$S_{\ell}$ meets $T_n$. 
\end{sublemma} 

Assume that \ref{bowtieintn} fails. 
Then $\{f_\ell,g_\ell,b_n,c_n\}$  is a cocircuit of $M$, and we can adjoin this cocircuit and $T_n$ to the end of the bowtie string 
$T_0,D_0,T_1,D_1,\dots ,T_j,C_0,S_1,C_1,\dots ,S_\ell$ to give a \cn. Thus \ref{bowtieintn} holds.  

Now choose $m$ to be the least integer such that   $S_m$ meets $T_0 \cup T_1 \cup \dots \cup T_n$.   Since $S_1 \cup S_2 \cup \dots \cup S_{\ell}$ avoids $T_0 \cup T_1 \cup \dots \cup T_j$, we see that $S_m$ meets $T_p$ for some $p$  in $\{j+1,j+2,\dots,n\}$.
We show next that  

\begin{sublemma}
\label{smtp}
$S_m = T_p$.
\end{sublemma}

By hypothesis, $m > 1$. 
First observe that if $\{e_m,f_m\}$ meets $T_p$, then \ref{smtp} follows by \ort\ between $T_p$ and $C_{m-1}$. We may now assume that $g_m \in T_p$. Then \ort\ between $S_m$ and one of the cocircuits $D_{p-1}$ and $D_p$ implies that  $\{e_m,f_m\}$ meets $T_{p-1} \cup T_p \cup T_{p+1}$ where we interpret the subscripts on   $D_i$ and $T_i$ modulo $n+1$. Then 
$\{e_m,f_m\}$ meets $T_{p-1}, T_p$, or $T_{p+1}$. Thus, by   the first part of the argument, we see that $S_m \in \{T_{p-1},T_p,T_{p+1}\}$. As $g_m \in T_p$, we deduce that $S_m = T_p$, so \ref{smtp} holds. 

Define $X =   \{c_0,c_1,\dots,c_n\} - \{c_{j-1},c_j\}$. Clearly $T_{j-1},D_{j-1},T_j, C_0,\linebreak S_1,\dots, C_{m-1},S_{m-1}$ is a bowtie string in  $M\ba X$. Moreover, 
 $N\preceq M\ba X\ba c_{j-1},c_j.$  Applying Lemma~\ref{stringybark} 
gives that 
 either 
\begin{itemize}
\item[(a)] $N\preceq M\ba X\ba c_{j-1},c_j/z$;  or 
\item[(b)] $N\preceq M\ba X\ba c_{j-1},c_j,g_1,\dots,g_{m-1}$ and 
$$N\not\preceq M\ba X\ba c_{j-1},c_j,g_1,\dots,g_{m-1}/f_{m-1}.$$
\end{itemize}

Consider the second case. Then  $M\ba c_0,c_1,\dots,c_n,g_1,g_2,\dots, g_{m-1}$ has an $N$-minor but does not have a cocircuit  containing $f_{m-1}$ and having at most  two elements. As $S_m = T_p$, either $c_p \in \{e_m,f_m\}$, or $\{e_m,f_m\} = \{a_p,b_p\}$. Thus either $C_{m-1}$ or $C_{m-1} \btu D_{p-1}$ contains a $1$- or $2$-element cocircuit   of 
 $M\ba c_0,c_1,\dots,c_n,g_1,g_2,\dots, g_{m-1}$ containing $f_{m-1}$. Hence (b) does not hold.
 
We now know that (a) holds. Then we get a \cn\ unless $j = 1$ and $z= a_j$. In the exceptional case, the lemma follows using Lemma~\ref{lordoftherings}. 
\end{proof}

\begin{lemma}
\label{hazwaste}
Let $M$ and $N$ be \ifc\ binary matroids such that  $|E(M)|\geq 13$ and $|E(N)|\geq 7$.   
Suppose that $M$ has a  bowtie  ring $(T_0,D_0,T_1,D_1,\dots, T_n,\{b_n,c_n,a_0,b_0\})$, that $M\ba c_0$ is \ffsc, and that $M\ba c_0,c_1$ has an $N$-minor.  
Suppose $M$ has a bowtie $(\{a_j,b_j,c_j\},\{e_1,f_1,g_1\},\{z,c_j,e_1,f_1\})$  for some $j$ in $\{1,2,\dots, n\}$ and some $z$ in $\{a_j,b_j\}$, where $\{e_1,f_1,g_1\}$ avoids $T_0\cup T_1\cup \dots \cup T_n$.  
Then 
\begin{itemize}
\item[(i)] $M$ has a quick win; or 
\item[(ii)] $M$ has an open-rotor-chain win, a bowtie-ring win, or a ladder win; or 
\item[(iii)] $M$ has an enhanced-ladder win; or
\item[(iv)] $M\ba c_0,c_1/b_1$ has an $N$-minor; or 
\item[(v)] $j =1$ and $z=a_j$, and both of the matroids $M\ba c_{0},c_{1},\dots c_n/a_1$ and $M\ba a_{0},b_{1},a_2,a_3,\dots a_n/b_2$ have $N$-minors.    
\end{itemize}
\end{lemma}

\begin{proof} Assume that the lemma fails. 
First we show the following.
\begin{sublemma}
\label{jis1} When $z = a_j$ and $j = 1$, the matroid $M\ba c_0,c_1/a_1$ has no $N$-minor.
\end{sublemma} 

Assume  that $M\ba c_0,c_1/a_1$ has an $N$-minor. As $M\ba c_0,c_1/a_1$ has 
$(c_2,b_2,a_2,b_1)$ as a $4$-fan,  by Lemma~\ref{2.2}, $M\ba c_0,c_1/a_1/b_1$  or $M\ba c_0,c_1/a_1\ba c_2$ has an $N$-minor.  
The first option gives the \cn\ that $M\ba c_0,c_1/b_1$ has an $N$-minor, so we assume the latter.  
Then $M\ba c_0,c_1/a_1\ba c_2$ has an $N$-minor and has $(c_3,b_3,a_3,b_2)$ as a $4$-fan. 
By repeatedly applying this argument, we deduce that $M\ba c_0,c_1/a_1\ba c_2,c_3,\dots ,c_n$ has an $N$-minor.  
By Lemma~\ref{lordoftherings}, $M\ba c_{0},c_{1},\dots c_n/a_1 \cong M\ba a_{0},b_{1},a_2,a_3,\dots a_n/b_2$. 
Thus we obtain the \cn\ that (v)  holds, so \ref{jis1} is proved.  

Take a right-maximal bowtie string $T_0,D_0,T_1,D_1,\dots ,T_j,C_0,S_1,C_1,\dots ,S_\ell$, where, for all $i$ in $\{2,3,\dots, \ell\}$, the set $S_i$ is $\{e_i,f_i,g_i\}$, a triangle,  and  $C_{i-1}$ is $\{f_{i-1},g_{i-1},e_{i},f_{i}\}$, a cocircuit. By Lemma~\ref{ibhaz}, $M$ has no cocircuit $\{y,g_{\ell},a_0,c_0\}$ with $y$ in $\{e_{\ell},f_{\ell}\}$. Then, by Lemma~\ref{reachtheend} and \ref{jis1}, we deduce that $M$ contains the configuration in 
Figure~\ref{drossfigiinouv} with suitable adjustments to the labelling. Let $z'$ be the element of $\{a_j,b_j\} - z$. 
Now $M$ has $\{c_j,e_1,d_j\}$ as a triangle, and $S_1$ avoids $T_0\cup T_1\cup \dots \cup T_n$. Thus \ort\ between 
$\{c_j,e_1,d_j\}$ and $D_j$ implies that $d_j \in \{a_{j+1},b_{j+1}\}$. Furthermore, $M$ has $\{d_j,d_{j+1},e_1,g_1\}$ or 
$\{d_{j-1},d_j,z',c_j\}$ as a cocircuit. If $\{d_j,d_{j+1},e_1,g_1\}$ is a cocircuit, then \ort\ between this cocircuit and $T_{j+1}$ implies that $d_{j+1} \in T_{j+1}$. Thus $\lambda(T_j \cup T_{j+1} \cup S_1) \le 2$; a \cn. We deduce that $M$ has $\{d_{j-1},d_j,z',c_j\}$ as a cocircuit. As this cocircuit meets both $T_j$ and $T_{j+1}$,  Lemma~\ref{bowwow} implies that it equals $D_j$. Thus $z' = b_j$ and $\{d_{j-1},d_j\} = \{a_{j+1},b_{j+1}\}$. Then $M$ has  $\{c_{j-1},d_{j-1}, b_j\}$ as a triangle, and \ort\ between it and $D_{j+1}$ implies that $d_{j-1} = a_{j+1}$, so $d_j = b_{j+1}$. Thus the triangle $\{c_j,e_1,d_j\}$ meets  $D_{j+1}$ in a single element; a \cn.
\end{proof}

We now assemble the pieces already proved to finish the proof of  the main result.  

\begin{proof}[Proof of Theorem \ref{killcasek}]
For notational convenience, we suppose that our bowtie is $(T_0,T_1,D_0)$, where $M\ba c_0,c_1$ has an $N$-minor, $M\ba c_0$ is \ffsc, and $M\ba c_1$ is not \ffsc.  
Assume that the theorem fails. 

Lemma~\ref{6.3rsv} implies that $T_1$ is the central triangle of a quasi rotor 
$(T_0,T_1,T_2,D_0,\{y,c_1,a_2,b_2\},\{x,y,a_2\})$, for some $x$ in $\{b_0,c_0\}$ and some 
$y$  in $\{a_1,b_1\}$.  
By possibly switching the labels of $a_1$ and $b_1$, we may assume that $y=b_1$.  
If $x=b_0$, then $(a_2,b_0,b_1,a_1,c_1)$ is a $5$-fan in $M\ba c_0$; a \cn.  
Thus $x=c_0$. 

By Lemma~\ref{betweenbts}, $M$ has    a right-maximal rotor chain $((a_0,b_0,c_0),\linebreak (a_1,b_1,c_1),\dots ,(a_n,b_n,c_n))$ for some $n \ge 2$, and $M$ has a triangle $T_{n+1}$ and a $4$-cocircuit $D_n$ where $T_{n+1} = \{a_{n+1},b_{n+1},c_{n+1}\}$ and 
$D_n =\{b_n,c_n,a_{n+1},b_{n+1}\}$ such that $T_0,D_0,T_1,D_1,\dots ,T_n,  
D_n,T_{n+1}$ is a bowtie string, $M\ba c_0,c_1,\dots ,c_{n +1}$ has an $N$-minor, $M\ba c_n$ is \ffsc, and $M/b_i$ has no $N$-minor for all $i$ in $\{1,2,\dots,n-1\}$.  Thus $M$ contains one of the  structures shown in Figure~\ref{rotorchainplus}.

\begin{figure}[htb]
\centering
\includegraphics[scale=0.65]{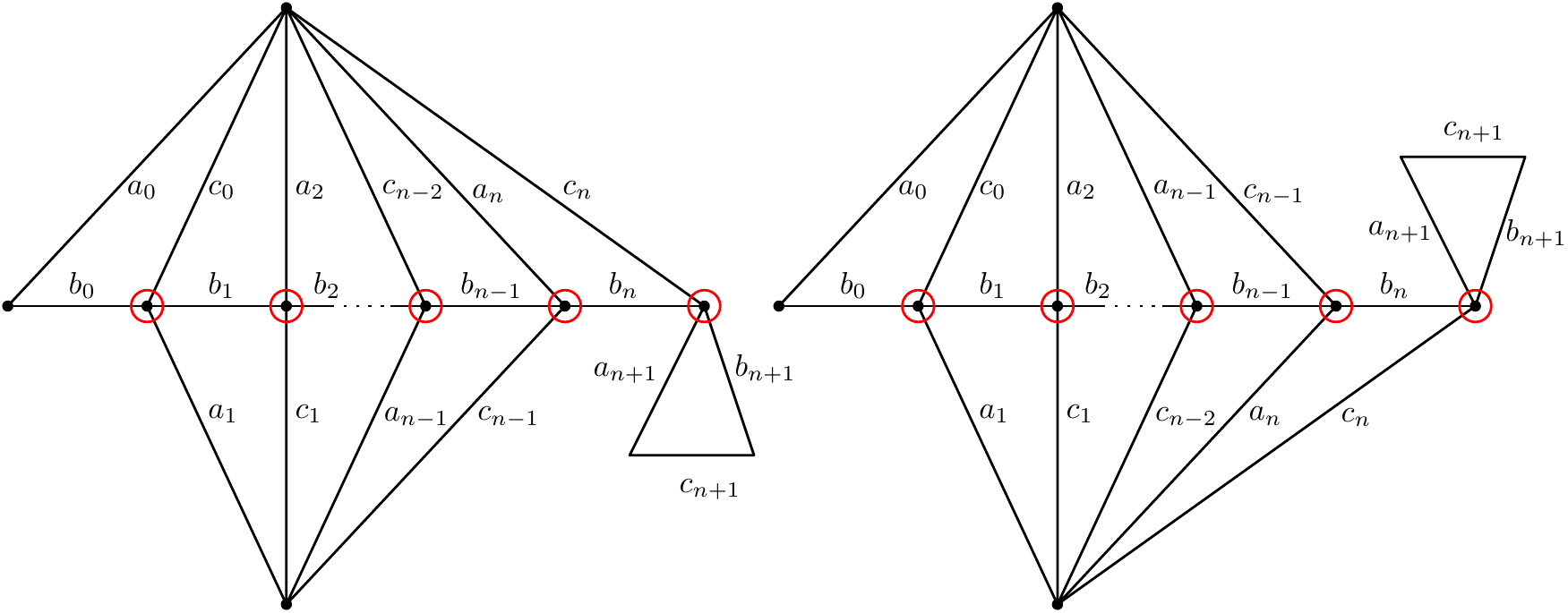}
\caption{ $n$ is even on the left and odd on the right.}
\label{rotorchainplus}
\end{figure}

\setcounter{theorem}{5}
\setcounter{sublemma}{0}

Take a right-maximal bowtie string $T_0,D_0,T_1,\dots ,T_n,D_n,T_{n+1},\dots ,T_k$.  
Then $k\geq n+1 \ge 3$.  By Lemma~\ref{rotorwin}, $M/a_1$ has no $N$-minor. Thus, by 
 Lemma~\ref{stringybark} and the final part of the previous paragraph, 
\begin{sublemma}
\label{notc}   
$M\ba c_0,c_1,\dots ,c_k$ has an $N$-minor, and $M/b_i$ has no $N$-minor for all  $i$ in $\{1,2,\dots ,n-1\}$. Moreover, for all $j$ in  $\{1,2,\dots ,k\}$, neither $M\ba c_0,c_1,\dots ,c_j/b_j$  nor  $M\ba c_0,c_1,\dots ,c_j/a_j$ has an $N$-minor. 
\end{sublemma}

Next we show that

\begin{sublemma}
\label{bowtiestriped}
$(T_k,T_0,\{x,c_k,a_0,b_0\})$ is  a bowtie for  some  $x$ in $\{a_k,b_k\}$.
\end{sublemma}

Assume that this fails. Then, by   
Lemma~\ref{reachtheend},   the bowtie string $T_0,D_0,T_1,\dots ,\linebreak T_n,D_n,T_{n+1},\dots ,T_k$ is contained in a ladder segment  as shown in Figure~\ref{drossfigiinouv}, where $k$ takes the place of $n$ in the figure. Thus $\{c_0,d_0,a_1\}$ and $\{c_1,d_1,a_2\}$ are triangles.  As $\{c_0,a_2,c_1,a_1\}$ is a circuit, using symmetric difference, we deduce that so too is $\{c_1,a_2,d_0\}$. Hence $d_0 = d_1$. But $\{d_0,a_1,c_1,d_1\}$ is a cocircuit as $k \ge 3$, so $\{d_0,a_1,c_1\}$ is a triad that meets a triangle of $M$; a \cn.  Thus \ref{bowtiestriped} holds.

By symmetry between $a_k$ and $b_k$, we may assume that $M$ has $(T_k,T_0,\{b_k,c_k,a_0,b_0\})$ as a bowtie.  
By \ref{notc} and Lemma~\ref{btring}, 
 $M\ba c_0,c_1,\dots ,c_k$ is \ffsc\ and $M$ has a triangle $\{e_1,f_1,g_1\}$ 
that is disjoint from $T_0 \cup T_1 \cup \dots \cup T_k$  
  such that, for some $j$ in $\{0,1,\dots ,k\}$, there is a cocircuit $\{e_1,f_1,h_1,c_j\}$ in $M$  for some $h_1$ in $\{b_j,a_j\}$.  
Suppose $j=0$.  
Then \ort\ implies that $\{e_1,f_1\}$ meets $\{b_1,a_2\}$, a \cn.  
Thus $j\geq 1$. 

 By Lemma~\ref{hazwaste} and \ref{notc}, part (v)  of that lemma holds. Thus $h = a_1$ and $j=1$, and $M\ba a_0,b_1,a_2,a_3,\dots ,a_k/b_2$, and hence $M/b_2$,  has an $N$-minor. Then, by \ref{notc}, $n= 2$ and $M$ contains  the structure in Figure~\ref{lastcase}.  
 
\begin{figure}[htb]
\centering
\includegraphics[scale=0.7]{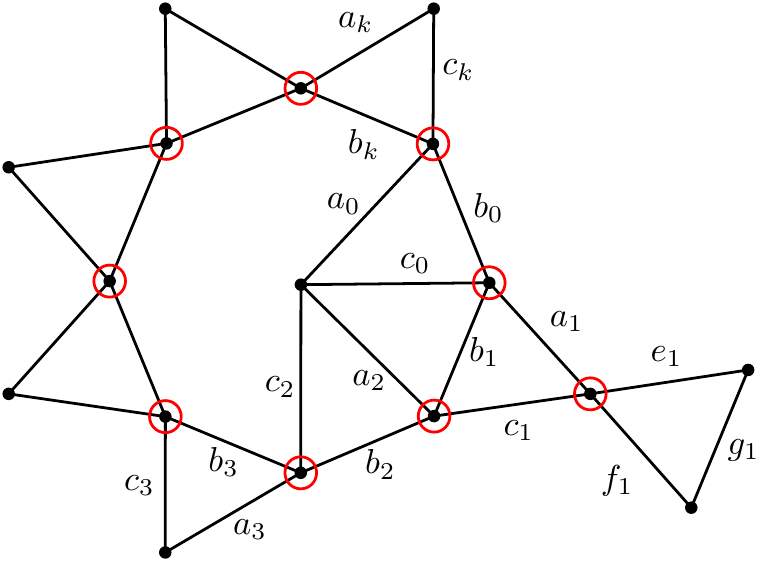}
\caption{}
\label{lastcase}
\end{figure}

To complete the proof of the theorem, we show that 
\begin{sublemma}
\label{mb2c2ifc}
$M/b_2\ba c_2$ is \ifc\ with an $N$-minor.
\end{sublemma}  

Since $M/b_2$ has an $N$-minor,  $M/b_2\ba c_2$ also has an $N$-minor.  
By Lemma~\ref{rotorchainends}, $M\ba c_2$ is \ffsc\ and, as it has $b_2$ as the coguts element of a $4$-fan, we deduce that $M/b_2\ba c_2$ is \thc.  
Suppose that $M/b_2\ba c_2$ has $(U,V)$ as a \ns\ \ths.  
Then, by~\cite[Lemma~3.3]{cmoV}, we may assume that $T_0\cup T_1\cup a_2\subseteq U$.  
Thus $(U\cup \{b_2,c_2\},V)$ is a \ns\ \ths\ of $M$; a \cn.  
Therefore $M/b_2\ba c_2$ is \sfc.  

Let $(\al ,\be,\ga,\de)$ be a $4$-fan in $M/b_2\ba c_2$.  
Suppose $\{\be,\ga,\de\}$ is a triad in $M$.  Then none of $\be, \ga,$ or $\de$ is in a triangle of $M$. 
Hence $\{b_2,\al,\be,\ga\}$ is a circuit in $M$, and \ort\ implies that this circuit meets both $\{b_1,c_1,a_2\}$ and $\{a_3,b_3\}$. Thus $\{\be, \ga\}$ meets a triangle; a \cn. 
We deduce that $\{\be,\ga,\de\}$ is not a triad in $M$. Thus $\{c_2,\be,\ga,\de\}$ is a cocircuit of $M$.  
Then, by \ort, $a_2\in\{\be,\ga,\de\}$, so $\{c_0,b_1\}$ also meets $\{\be,\ga,\de\}$, and either $T_0$ or $T_1$ contains two elements of $\{\be,\ga,\de\}$.  
Thus $\lambda (T_0\cup T_1\cup T_2)\leq 2$; a \cn.  
We conclude that $M/b_2\ba c_2$ is \ifc, so \ref{mb2c2ifc} holds and, hence, so does the theorem.  
\end{proof}

\section*{Acknowledgements} 
The authors thank Dillon Mayhew for numerous helpful discussions.

\end{document}